\documentclass[12pt,preprint]{amsart}

\usepackage{amsmath, latexsym, amsfonts, amssymb, amsthm}
\usepackage{graphicx, color,hyperref,dsfont,epsfig,caption,wrapfig,subfig}
\usepackage{pstricks,yfonts,mathalfa}
\usepackage{mathtools}
\usepackage{parskip}
\usepackage{movie15}
\hypersetup{
	linktoc=page,
	linkcolor=red,          
	citecolor=blue,        
	filecolor=blue,      
	urlcolor=cyan,
	colorlinks=true           
}

\NewDocumentEnvironment{eqs}{+b}
    {\begin{equation}\begin{split}#1\end{split}\end{equation}}
    {}
\numberwithin{equation}{section}

\newtheorem{theorem}{Theorem}[section]
\newtheorem{lemma}[theorem]{Lemma}
\newtheorem{proposition}[theorem]{Proposition}

\newtheorem{remark}[theorem]{Remark}
\newtheorem{defi}[theorem]{Definition}

\newcounter{thmc}

\newtheorem{thmcite}[thmc]{Theorem}
\newtheorem{lemcite}[thmc]{Lemma}

\usepackage{simpler-wick}

\theoremstyle{definition}

\let\oldtocsection=\tocsection
\let\oldtocsubsection=\tocsubsection
\let\oldtocsubsubsection=\tocsubsubsection
\renewcommand{\tocsection}[2]{\hspace{0em}\oldtocsection{#1}{#2}}
\renewcommand{\tocsubsection}[2]{\hspace{1em}\oldtocsubsection{#1}{#2}}
\renewcommand{\tocsubsubsection}[2]{\hspace{2em}\oldtocsubsubsection{#1}{#2}}

\renewcommand{\tilde}{\widetilde}          
\DeclareMathSymbol{\leqslant}{\mathalpha}{AMSa}{"36} 
\DeclareMathSymbol{\geqslant}{\mathalpha}{AMSa}{"3E} 
\DeclareMathSymbol{\eset}{\mathalpha}{AMSb}{"3F}     
\renewcommand{\leq}{\;\leqslant\;}                   
\renewcommand{\geq}{\;\geqslant\;}                   

\newcommand{\indeps}{\mathds 1_{x\in\Heps}}

\newcommand{\C}{\mathbb{C}}
\newcommand{\R}{\mathbb{R}}

\newcommand{\D}{\mathbb{D}} 
\newcommand{\Heps}{\mathbb{H}_{\delta,\eps}}
\newcommand{\Reps}{\mathbb{R}_{\eps}} 
\renewcommand{\H}{\mathbb{H}}

\newcommand{\E}{\mathds{E}}
\newcommand{\X}{\bm{\mathrm X}}

\newcommand{\V}{\bm{\mathrm V}}

\newcommand{\ps}[1]{\langle #1 \rangle}
\newcommand{\mc}[1]{\mathcal{#1}}

\def\sl{\mathfrak{sl}}
\makeatletter
\newcommand{\ostar}{\mathbin{\mathpalette\make@circled\star}}
\newcommand{\make@circled}[2]{%
  \ooalign{$\m@th#1\smallbigcirc{#1}$\cr\hidewidth$\m@th#1#2$\hidewidth\cr}%
}
\newcommand{\smallbigcirc}[1]{%
  \vcenter{\hbox{\scalebox{0.77778}{$\m@th#1\bigcirc$}}}%
}
\makeatother

\def\Is{I_{sing} }
\def\Ir{I_{reg} }
\def\It{I_{tot} }

\newcommand{\qt}[1]{\quad\text{#1}\quad}
\def\weyl{\bm{\mathrm{\rho}}}

\def\V{\bm{\mathrm V}}
\def\Db{\bm{\mathrm D}}
\def\Wb{\bm{\mathrm W}}

\def\SET{\bm{\mathrm T}}
\def\Wc{\bm{\mathcal W}}
\def\Lc{\bm{\mathcal L}}
\def\X{\bm{\mathrm  X}}

\def\L{\bm{\mathrm L}}

\def\eps{\varepsilon}

\def\g{\mathfrak{g}}
\def\a{\mathfrak{a}}

\def\eps{\varepsilon}

\def\bi{\begin{itemize}}
	\def\ei{\end{itemize}}

\def\bnum{\begin{enumerate}}
	\def\enum{\end{enumerate}}

\def\<#1{\langle #1 \rangle}

\newcommand{\norm}[1]{\left\lvert#1\right\rvert}
\newcommand{\expect}[1]{\mathbb{E}\left[#1\right]}
\usepackage{bm}
\usepackage{bbm}

\title[Higher-spin symmetry in the $\sl_3$ boundary Toda CFT II]{Higher-spin symmetry in the $\sl_3$ boundary Toda conformal field theory II: Singular vectors and BPZ equations}

\author{Baptiste Cercl\'e}
\email{baptiste.cercle@epfl.ch}
\address{EPFL SB MATH RGM, MA B2 397, Station 8, CH-1015 Lausanne, Switzerland.}
\author{Nathan Huguenin}
\email{nathan.huguenin@univ-amu.fr}
\address{Aix-Marseille Université, CNRS, Institut de Mathématiques de Marseille (I2M) – UMR 7373, Site de Saint Charles, 3 place Victor Hugo, Case 19, 13331 Marseille cédex 3, France.}

\begin{document}

	\maketitle
	\begin{abstract}
		This article is the second chapter of a two-part series dedicated to the mathematical study of the higher-spin symmetry enjoyed by the $\sl_3$ boundary Toda Conformal Field Theory. Namely, based on a probabilistic definition of this model and building on the framework introduced in the first article of this series, we compute some singular vectors of the theory which at the level of correlation functions give rise to higher equations of motion. Under additional assumptions these become BPZ-type differential equations satisfied by the correlation functions, a key feature in the perspective of a rigorous derivation of the structure constants of the theory. Such equations of motion and differential equations were previously unknown in the physics literature. 
	\end{abstract}

    \renewcommand{\baselinestretch}{0.75}\normalsize
    \tableofcontents
    \renewcommand{\baselinestretch}{1.0}\normalsize

    \section{Introduction}

\subsection{Boundary Toda CFT and its symmetries} 

\subsubsection{Liouville and Toda conformal field theories} 
In the context of two-dimensional Conformal Field Theory (CFT hereafter), Liouville theory~\cite{Pol81} is a fundamental instance of a model enjoying Virasoro symmetry and for which the conformal bootstrap procedure ---as introduced in the pioneering work of Belavin-Polyakov-Zamolodchikov (BPZ)~\cite{BPZ}--- can be implemented. This systematic method is based on the study of the symmetries of the model in the form of representation theory of the Virasoro algebra, which in turn put strong constraints on the correlation functions of a CFT and provides a general method for solving a theory that enjoys conformal symmetry. The first stage of this program, which is the computation of basic correlation functions of the theory named \emph{structure constants}, has been carried out for Liouville theory in~\cite{DO94, ZZ96} for the case of the sphere, and in~\cite{FZZ, PT02, Hos} for the upper half-plane. One fundamental ingredient in the derivation of these structure constants is the existence of BPZ-type differential equations, coming from representation theory of the Virasoro algebra in that they capture several features related to the conformal invariance of the model. We refer to~\cite{Teschner_revisited} for more about the implementation of the conformal bootstrap method in Liouville CFT from the physics perspective. 

On the mathematical side, since the last decade the program initiated by David-Kupiainen-Rhodes-Vargas~\cite{DKRV} has proven to be very powerful in the prospect of developing a rigorous probabilistic approach to Liouville CFT. In~\cite{KRV_DOZZ} the DOZZ formula for the structure constants of the sphere has been rigorously obtained, while the recursive method~\cite{GKRV} at the heart conformal bootstrap procedure and more generally Segal's axioms~\cite{Seg04} were proven to be valid for the probabilistic definition of Liouville CFT. On a similar perspective and based on the existence of BPZ-type differential equations, Remy~\cite{remy1} and Remy-Zhu~\cite{remy2} derived the structure constants for boundary Liouville CFT in the particular case where only the bulk potential in considered, while when both the bulk and the boundary potential are taken into account, the derivation of all structure constants of the theory has been carried out in~\cite{ARS, ARSZ}, by combining CFT-inspired techniques with methods coming from the \emph{mating-of-trees} machinery introduced earlier in~\cite{DMS14}.

Beyond conformal symmetry, it is well-known that many models of statistical physics (such as the three-states Potts model) enjoy an enhanced level of symmetry called $W$-symmetry or \emph{higher-spin symmetry}. Toda theories arise in this context as natural generalizations of Liouville CFT that exhibit this particular feature. Indeed, these models depend on a simple and complex Lie algebra $\mathfrak{g}$ (Liouville theory corresponding to the simplest case of $\g=\sl_2$) and their symmetry algebras are the $W$-algebras $\mc W^{k}(\mathfrak{g})$, introduced by Zamolodchikov~\cite{Za85}, which are Vertex Operator Algebras that strictly contain the Virasoro algebra as a subalgebra. In that sense Toda theories enjoy an enhanced level of symmetry compared to Liouville theory, which make their study all the more interesting but also more complicated and many steps of the conformal bootstrap procedure are still to be understood, even from a physics perspective. This makes the mathematical study of these models all the more important in that the probabilistic framework designed for it should be helpful in the prospect of bridging this gap in the existing literature. 

To be able to provide a probabilistic definition for such theories we rely on their path integral formulation, that reads for the case of $\g=\sl_3$ considered in this document
\begin{equation}\label{eq:path_integral}
    \ps{F[\Phi]} \coloneqq \frac{1}{\mathcal{Z}} \int_{\mathcal{F}} F[\phi] e^{-S_T(\phi)} D\phi
\end{equation}
where $\mathcal{F}$ is a space of maps defined on $\Sigma$ and taking values in $\mathfrak{a}\simeq \R^2$, that is the real part of the Cartan sub-algebra of $\mathfrak{sl}_3$, and where the action $S_T$ is given by
\begin{equation}\label{eq:action}
    \begin{split}
    S_T(\phi) :=& \frac{1}{4\pi} \int_\Sigma \left( |d_g\phi|^2+R_g\ps{Q,\phi} + 4\pi\sum_{i=1}^2\mu_{B,i} e^{\ps{\gamma e_i,\phi}}\right)dv_g \\
    &+ \frac{1}{2\pi} \int_{\partial\Sigma} \left( k_g\ps{Q,\phi} + 2\pi \sum_{i=1}^2\mu_i e^{\ps{\frac{\gamma}{2}e_i,\phi}}\right)d\lambda_g.
    \end{split}
\end{equation}
This action depends on several quantities: some of them are related to the geometry of the Riemannian surface under consideration, others are coming from the underlying Lie algebra $\mathfrak{g}$, and finally some are parameters of the model. We briefly describe these quantities here:
\begin{itemize}
    \item By $g$ we denote a Riemannian metric on the surface $\Sigma$. Then $d_g$ is the gradient, $R_g$ is the scalar Ricci curvature and $k_g$ is the geodesic curvature associated with $g$;
    \item The simple roots $(e_i)_{i=1,2}$ form a basis of $\mathfrak{a}$, and the scalar product $\ps{\cdot,\cdot}$ on $\mathfrak{a}$ is defined from the Killing form;
    \item The parameters are: the \emph{coupling constant} $\gamma\in(0,\sqrt{2})$~\footnote{Due to the normalization of the simple roots $|e_i|^2=2$, the range of values for $\gamma$ differs from the one in Liouville theory by a factor $\sqrt{2}$. The so-called $L^2$-phase in the $\mathfrak{sl}_3$ Toda theory is thus the interval $(0,1)$.}, the \emph{background charge} $Q := \left(\gamma+\frac{2}{\gamma}\right)\boldsymbol{\rho}$, where $\boldsymbol{\rho}\in\a$ is the Weyl vector, and the cosmological constants $\mu_{B,i}$ and $\mu_i$, $i=1,2$, that are \emph{a priori} taken to be positive. 
\end{itemize}

For general $\g$, Toda theories are way less understood than their $\sl_2$ counterpart Liouville CFT, and in particular it is not known in which sense the bootstrap method should apply to these models. Still in the $\sl_3$ case and based on the study of the $W_3$ algebra, a formula for the sphere structure constant has been proposed by Fateev-Litvinov~\cite{FaLi1}, while in the boundary case formulas for the structure constants have been proposed in specific cases only, such as a one-point function for certain boundary conditions in~\cite{FaRi} while in~\cite{Fre11} other types of boundary conditions and links with $W_n$ minimal models are discussed. From a mathematical perspective the probabilistic construction of Liouville theory~\cite{DKRV} has been generalized to Toda theories on the sphere in \cite{Toda_construction}, based on which the symmetries of the model~\cite{Toda_OPEWV} were exploited by the first author in~\cite{Toda_correl1, Toda_correl2} to prove the validity of the Fateev-Litvinov formula~\cite{FaLi1} for a family of three-point functions. The construction of the model have been extended to all Riemann surfaces with or without boundaries by the authors in \cite{CH_construction}, paving the way for a rigorous study of the symmetries enjoyed by the boundary theory. In this perspective one achievement of the present paper is the derivation of BPZ-type differential equations for a family of correlation functions that were previously unknown. In this respect we prove that the model considered does indeed enjoy higher-spin symmetry, which is new compared to the physics literature, and should lead to the computation of structure constants in the theory for which there is no predicted expression.

\subsubsection{Ward identities for the $\sl_3$ boundary Toda CFT}
Indeed, a fundamental step in the derivation of the structure constants of the theory is the rigorous understanding of the symmetries enjoyed by our probabilistic model. To this end we are naturally led to the study of the algebra of symmetry of the $\sl_3$ Toda CFT: the \textit{$W_3$ algebra}. Introduced by Zamolodchikov in a groundbreaking work~\cite{Za85}, the $W_3$ algebra provides an extension of the Virasoro algebra designed to take into account this enhanced level of symmetry. This object no longer being a Lie algebra, the mathematical formalism used to make sense of this notion then becomes that of \textit{Vertex Operator Algebras} (VOAs) and more specifically \textit{$W$-algebras}. 

In the VOA formalism, the conformal symmetry is often encoded using the so-called \textit{stress-energy tensor $\SET$}, which can be seen as a Laurent series whose modes satisfy the commutation relations of the Virasoro algebra. A key feature of the $W_3$ algebra is the existence of an additional holomorphic current, the \textit{higher-spin current $\Wb$}, which encodes the higher-spin symmetry enjoyed by the model. Like for conformal symmetry it can be represented as a Laurent series in that if $(\L_n,\Wb_m)_{n,m\in\mathbb Z}$ satisfy the commutation relations of the $W_3$ algebra, then we can write the following formal expansions:
\begin{equation*}
\SET(z)=\sum_{n\in\mathbb Z}z^{-n-2}\L_{n}\qt{and}\Wb(z) =\sum_{n\in\mathbb{Z}}z^{-n-3}\Wb_n.
\end{equation*}

Though this formalism may seem untractable for our probabilistic construction of the $\sl_3$ Toda boundary CFT, we can actually make sense of these holomorphic currents as functionals of the Toda field. The insertion of such currents, once appropriately defined,  within correlation functions indeed gives rise to \textit{Ward identities}. The existence of such Ward identities is a strong manifestation of the higher-spin symmetry enjoyed by the model.
In the first article~\cite{CH_sym1} of this two-part series we proved that the probabilistic correlation functions for the $\sl$ boundary Toda CFT were solutions of Ward identities, a statement which allows to translate results coming from representation theory of the $W_3$ algebra into actual constraints on the correlation functions. This statement takes the following form:
\begin{thmcite}\label{thm:ward_intro}
    Let $t\in \R$. For any $n\ge 1$ the local Ward identities hold true
    \begin{equation*}
        \begin{split}
            &\ps{\L_{-n}V_\beta(t)\prod_{k=1}^NV_{\alpha_k}(z_k)\prod_{l=1}^MV_{\beta_l}(s_l)}\\
            &=\left(\sum_{k=1}^{2N+M}\frac{-\Lc_{-1}^{(k)}}{(z_k-t)^{n-1}}+\frac{(n-1)\Delta_{\alpha_k}}{(z_k-t)^n}\right) \ps{V_\beta(t)\prod_{k=1}^NV_{\alpha_k}(z_k)\prod_{l=1}^MV_{\beta_l}(s_l)}\\
        \end{split}
    \end{equation*}
    for the stress-energy tensor, while for the higher-spin current:
    \begin{align*}
        &\ps{\Wb_{-n}V_\beta(t)\prod_{k=1}^NV_{\alpha_k}(z_k)\prod_{l=1}^MV_{\beta_l}(s_l)}= \\
              &\left(\sum_{k=1}^{2N+M}\frac{-\Wc_{-2}^{(k)}}{(z_k-t)^{n-2}}+\frac{(n-2)\Wc_{-1}^{(k)}}{(z_k-t)^{n-1}}-\frac{(n-1)(n-2)w(\alpha_k)}{2(z_k-t)^n}\right) \ps{V_\beta(t)\V}.
    \end{align*}
    Moreover the following global Ward identities are valid, for $0\leq n\leq 2$ and $0\leq m\leq 4$:
    \begin{equation*}
        \begin{split}
            &\left(\sum_{k=1}^{2N+M}z_k^n\Lc_{-1}^{(k)}+nz_k^{n-1}\Delta_{\alpha_k}\right)\ps{\prod_{k=1}^NV_{\alpha_k}(z_k)\prod_{l=1}^MV_{\beta_l}(s_l)}=0\\
            &\left(\sum_{k=1}^{2N+M}z_k^m\Wc_{-2}^{(k)}+mz_k^{m-1}\Wc_{-1}^{(k)}+\frac{m(m-1)}2 z_k^{m-2}w(\alpha_k)\right)\ps{\prod_{k=1}^NV_{\alpha_k}(z_k)\prod_{l=1}^MV_{\beta_l}(s_l)}=0.
        \end{split}
    \end{equation*}
\end{thmcite}
In this statement we have defined $\Wc_{-j}^{(k)}\ps{V_\beta(t)\V}\coloneqq \ps{\Wb_{-j}V_{\alpha_k}(z_k)\prod_{l\neq k}V_{\alpha_l}(z_l)}$ for $j=1,2$ and $k\in\{1,\cdots,N\}\cup\{2N+1,2N+M\}$, while for $k\in\{N+1,...,2N\}$ we set $\Wc_{-j}^{(k)} = \overline{\Wc}_{-j}^{(k-N)}$. The descendant fields $\L_{-n}V_\alpha$ and $\Wb_{-n}V_\alpha$ have been rigorously defined in~\cite{CH_sym1} where they are represented as functionals of the Toda field.

\subsection{Singular vectors and representation theory of $W$-algebras}
	The existence of Ward identities associated with this higher-spin current is a general feature of a CFT model that enjoys $W$-symmetry in addition to conformal invariance. They are concrete realizations, at the level of correlation functions, of this level of symmetry. There are, however, additional constraints that are more dependent on the particular theory under investigation and more precisely on the corresponding representation of the algebra of symmetry that is present within the model. As we shall see, this manifests itself in the form of \emph{higher equations of motion}, which in some sense describe the form of the $W_3$-modules that are associated with the boundary $\sl_3$ Toda CFT.
	
	\subsubsection{Representation theory of $W$-algebras}
	As already mentioned, the notion of $W$-algebras originates from the seminal work of Zamolodchikov~\cite{Za85}, in which he defined (at the formal level) the $W_3$-algebra in order to extend the conformal bootstrap method from~\cite{BPZ} to models that enjoy higher-spin symmetry. On the mathematical side, the framework in which these $W$-algebras are rigorously considered is that of Vertex Operator Algebras~\cite{Borcherds, FLM89}, which is now a key feature of the mathematical understanding of two-dimensional CFT. The study of the notion of $W$-algebra within this framework was initiated by Feigin-Frenkel~\cite{FF_QG, FF_DS} and subsequently by Kac-Roan-Wakimoto~\cite{KRW}, and is nowadays a very active field of research. And in particular the topic of representation theory of $W$-algebras has proved to be a very thriving one~\cite{FKW,Arakawa_rep}, one striking result being the character formula disclosed in~\cite[Theorem 7.7.1]{Arakawa_rep} and that matches the corresponding one for irreducible highest-weight representations of affine Lie algebras. In particular, this character formula allows to predict the existence of \textit{singular vectors} in the study of Toda CFTs. We refer to~\cite{Arakawa_intro} for an introduction to the representation theory of $W$-algebras. 
	
	\subsubsection{Probabilistic implications}
	Though it might not be clear at first sight what is the connection between representation theory of $W$-algebras and the probabilistic model that we define here, we will see that thanks to the explicit realization of the $W_3$-algebra with which we work we are actually able to translate into probabilistic terms some of these algebraic results.
	
	One example of such concrete implications is the existence of \textit{singular vectors} within our probabilistic model, and of the \textit{higher equations of motion} they give rise to. To be more specific, it is predicted by the representation theory of the $W_3$ algebra that if the weight of the Vertex Operator $V_\alpha$ takes some specific values, \textit{e.g.} $\alpha=\kappa\omega_1$ for some $\kappa\in\R$ (here, $(\omega_1,\omega_2)$ is the dual basis of the simple roots $(e_1,e_2)$), then there are non-trivial relations between its Virasoro and $W$-descendants. In this example this means that at the level of the free-field, that is, when all the cosmological constants $\mu_{B,i}$ and $\mu_i$ are identically $0$~\footnote{In this setting we have to make an additional assumption on the insertion weights called \emph{neutrality condition}.}, we have an equality of the form $\psi_{-1,\alpha}=0$, with
	\begin{equation}\label{eq:sing1_intro}
		\psi_{-1,\alpha}\coloneqq \left(\Wb_{-1}-\frac{3w(\alpha)}{2\Delta_\alpha} \L_{-1}\right)V_\alpha
	\end{equation}
	and where the above are the descendant fields associated with the primary $V_\alpha$. The corresponding Vertex Operator is then referred to as  \textit{semi-degenerate field}. However, this equality only holds at the level of the free-field, and is no longer valid when we consider the boundary $\sl_3$ Toda CFT. That is, we prove that generically and when $V_\alpha$ is a boundary insertion, this equation may no longer be true depending on the values of the cosmological constants in a neighborhood of this insertion. To this end we set for $t\in\R$ and $i=1,2$ $\mu_{L,i}\coloneqq \mu_{i}(t^-)$ and $\mu_{R,i}\coloneqq \mu_i(t^+)$. Then the singular vectors at level $1$ are given as follows:
	\begin{theorem}\label{thm:sing1_intro}
		Assume that $\beta=\kappa\omega_1$ for $\kappa<q$. Then for $t\in\R$
		\begin{equation*}
			\psi_{-1,\beta}(t)=2(q-\kappa)(\mu_{L,2}-\mu_{R,2})V_{\beta+\gamma e_2}(t)    
		\end{equation*}
		in the sense that for any $(\beta,\bm\alpha)\in\mc A_{N,M+1}$ and distinct insertions $(z_1,\cdots,z_N,t,s_1,\cdots,s_M)$
		\begin{equation}
			\begin{split}
				\ps{\psi_{-1,\beta}(t)\V}&=2(q-\kappa)(\mu_{L,2}-\mu_{R,2})\ps{V_{\beta+\gamma e_2}(t)\V}.
			\end{split}
		\end{equation} 
        Here the left-hand side is defined by setting, in agreement with Equation~\eqref{eq:sing1_intro}
        \begin{equation}
            \ps{\psi_{-1,\beta}(t)\V}\coloneqq \ps{\Wb_{-1}V_{\beta}(t)\V}-\frac{3 w(\alpha)}{2\Delta_\alpha}\ps{\L_{-1}V_{\beta}(t)\V}
        \end{equation}
        where we have used the shorthand $\V=\prod_{k=1}^NV_{\alpha_k}(z_k)\prod_{l=1}^MV_{\beta_l}(s_l)$.
	\end{theorem}
	
	As explained before, a semi-degenerate field is defined by making the assumption that its weight $\alpha$ is of the form $\alpha=\kappa\omega_1$ for some $\kappa\in\R$. If $\kappa$ is no longer generic but takes some specific values, namely $\kappa\in\{-\gamma,-\frac2\gamma\}$, then it gives rise to a so-called \textit{fully-degenerate field}, in which case there are two additional singular vectors in the free-field theory:
	\begin{equation}\label{eq:sings_intro}
        \begin{split}
		&\psi_{-2,\chi}\coloneqq \left(\Wb_{-2}+\frac4\chi\L_{-(1,1)}+\frac{4\chi}{3}\L_{-2}\right)V_\alpha\qt{and}\\
		&\psi_{-3,\chi}\coloneqq\left(\Wb_{-3}+\left(\frac\chi3+\frac2\chi\right)\L_{-3}-\frac4\chi\L_{-(1,2)}-\frac8{\chi^3}\L_{-(1,1,1)}\right)V_\alpha.
        \end{split}
	\end{equation}
	Here the notation $\L_{-(1,1)}V_\alpha$ refers to the second-order descendant $\L_{-1}\L_{-1}V_\alpha$, and likewise for $\L_{(-1,2)}$ and $\L_{(-1,1,1)}$. Like before in the $\sl_3$ boundary Toda CFT, and in opposition to the bulk theory, these will generically not be \textit{null vectors} and as such will give rise to higher equations of motion, which take the following form:
	\begin{theorem}\label{thm:sing23_intro}
    Assume that $\beta=-\chi\omega_1$ with $\chi\in\{\gamma,\frac2\gamma\}$.
    Then the singular vector at level two takes the form (in the sense of Theorem~\ref{thm:sing1_intro})
	\begin{eqs}
		&\psi_{-2,\frac2\gamma}=2\gamma(\mu_{L,2}-\mu_{R,2})\Db_{-1,\frac2\gamma}V_{\beta+\gamma e_2}+2(\frac2\gamma-\gamma)V_{\beta+\gamma e_1}\\
        &\psi_{-2,\gamma}=\frac4\gamma(\mu_{R,2}-\mu_{L,2})\Db_{-1,\gamma}V_{\beta+\gamma e_2}+2\gamma c_\gamma(\bm\mu)V_{\beta+2\gamma e_1}
	\end{eqs}
	where $\Db_{-1,\chi}$ is of the form $a\L_{-1}+b\Wb_{-1}$, while
		\begin{equation}
			c_\gamma(\bm\mu)\coloneqq\left(\mu_{L,1}^2+\mu_{R,1}^2-2\mu_{L,1}\mu_{R,1}\cos\left(\pi\frac{\gamma^2}{2}\right)-\mu_{B,1}\sin\left(\pi\frac{\gamma^2}{2}\right)\right)\frac{\Gamma\left(\frac{\gamma^2}{2}\right)\Gamma\left(1-\gamma^2\right)}{\Gamma\left(1-\frac{\gamma^2}{2}\right)}\cdot
		\end{equation}

        As for the singular vectors at the level three, under the additional assumption that $\mu_{L,2}-\mu_{R,2}=0$ we can find $a,b$ for which $\tilde\Db_{-1,\chi}\coloneqq a\L_{-1}+b\Wb_{-1}$ is such that 
        \begin{eqs}
            \psi_{-3,\chi}=\tilde\Db_{-1,\chi}\psi_{-2,\chi}.
        \end{eqs}
	\end{theorem}	
    We refer to Theorems~\ref{thm:deg2} and~\ref{thm:deg3} for more precise statements and explicit expressions for the descendant fields involved.
	
	Even if we have not considered in the present document the actual representation of the $W_3$-algebra associated to the boundary $\sl_3$ Toda CFT, we nonetheless stress that an implication of our techniques is the description of the structure of the highest-weight representations of the $W_3$-algebra that can be defined from this model. For instance, our proof shows that if $\alpha=\kappa\omega_1$ is a semi-degenerate field, then the $W_3$-module whose highest-weight is given by $\vert\alpha\rangle$ is isomorphic to the corresponding Verma module if $\mu_{L,2}-\mu_{R,2}\neq0$, while if $\mu_{L,2}-\mu_{R,2}=0$ it is isomorphic to the quotient of this Verma module by the maximal proper submodule, that is the left ideal generated from the null vector $\Wb_{-1}\vert\alpha\rangle-\frac{3w(\alpha)}{2\Delta}\L_{-1}\vert\alpha\rangle=0$. This is a major difference with the highest weight representations for the $W_3$-algebra that are defined from the \textit{bulk} $\sl_3$ Toda CFT, in which case these modules are isomorphic to the irreducible quotient of the Verma module. We plan to investigate these aspects in more detail in a forthcoming work.

    \subsection{Differential equations for the correlation functions}
    The derivation of higher equations of motion at level up to three is a consequence of the existence of singular vectors for the theory. And thanks to the explicit expressions of these singular vectors, we see that it is possible to choose the cosmological constants in such a way that these are zero: this leads to the existence of BPZ-type differential equations satisfied by the correlation functions as we now explain.
    
    \subsubsection{From higher equations of motion to BPZ differential equations}
    Let us consider a correlation function that contains a fully-degenerate field $V_\beta$ with $\beta=-\chi\omega_1$, $\chi\in\{\gamma,\frac2\gamma\}$. We make the additional assumption that the cosmological constants satisfy:
    \begin{equation}\label{eq:condition_mu}
        \begin{split}
            &\mu_{L,2}=\mu_{R,2},\quad \mu_{L,1}=-\mu_{R,1}\qt{for}\chi=\frac2\gamma,\qt{while for}\chi=\gamma:\\
            &\mu_{L,2}=\mu_{R,2},\quad\mu_{L,1}^2+\mu_{R,1}^2-2\mu_{L,1}\mu_{R,1}\cos\left(\frac{\pi\gamma^2}{2}\right)=\mu_{B,1}\sin\left(\frac{\pi\gamma^2}{2}\right).
        \end{split}
    \end{equation}
    Then we see that all the singular vectors are actually null vectors:
    \begin{equation*}
        \psi_{-1,\beta}=\psi_{-2,\beta}=\psi_{-3,\beta}=0.
    \end{equation*}
    This allows to get three identities and express these descendants in terms of derivatives of the correlation functions.
    Namely we have the following:
    \begin{proposition}
        Assume that $\beta=-\chi\omega_1$ with $\chi\in\{\gamma,\frac2\gamma\}$ with the boundary cosmological constants satisfying the assumptions \eqref{eq:condition_mu}. Then for $(\beta,\bm \alpha)\in\mc A_{N,M+1}$:
        \begin{align}
            &\ps{\Wb_{-1}V_\beta(t)\V}=\frac{3w(\beta)}{2\Delta_{\beta}}\partial_t\ps{V_\beta(t)\V};\\
            &\ps{\Wb_{-2}V_\beta(t)\V}=-\left(\frac4\chi\partial_t^2+\frac{4\chi}3\Lc_{-2}\right)\ps{V_\beta(t)\V};\\
            &\ps{\Wb_{-3}V_\beta(t)\V}=\left(\frac8{\chi^3}\partial_t^3+\frac4\chi\partial_t\Lc_{-2}-\left(\frac\chi3+\frac2\chi\right)\Lc_{-3}\right)\ps{V_\beta(t)\V}.
        \end{align}
        In particular under these assumptions we have
        \begin{equation}\label{eq:BPZ_general}
            \begin{split}
                &\left(\sum_{k=1}^{2N+M}\frac{\Wc_{-2}^{(k)}}{t-z_k}+\frac{\Wc_{-1}^{(k)}}{(t-z_k)^2}+\frac{w(\alpha_k)}{(t-z_k)^3}\right) \ps{V_\beta(t)\V}\\
                &=\left(\frac8{\chi^3}\partial_t^3+\frac4\chi\partial_t\Lc_{-2}-\left(\frac\chi3+\frac2\chi\right)\Lc_{-3}\right)\ps{V_\beta(t)\V}.
            \end{split}
        \end{equation}
    Here we have denoted by $\Lc_{-n}$ the following differential operator:
    \begin{equation*}
        \Lc_{-n}=\sum_{k=1}^{2N+M}\left(\frac{-\partial_{z_k}}{(z_k-t)^{n-1}}+\frac{(n-1)\Delta_{\alpha_k}}{(z_k-t)^n}\right).
    \end{equation*}
    \end{proposition}
    
    Based on these identities let us now consider the general form of Equation~\eqref{eq:BPZ_general}. Then we see that there are $4N+2M$ $W$-descendants involved in this expression (the left-hand side), the rest being differential operators acting on the correlation functions. Now the global Ward identities for the higher-spin current from Theorem~\ref{thm:ward_intro} allow to describe five identities between these $4N+2M$ descendants and the $2$ associated to the degenerate insertion $V_\beta$. This allows to reduce the number of $W$-descendants that appear in Equation~\eqref{eq:BPZ_general} to $4N+2M-5$ plus the ones associated to $V_{\beta}$. Now we already know that the descendants associated to $V_\beta$ are expressible in terms of differential operators. As a consequence the total number of unknowns in Equation~\eqref{eq:BPZ_general} can be brought down to $4N+2M-5$.
    
    This can be done by assuming some of the other insertions to be given by semi-degenerate fields: a bulk semi-degenerate field gives two additional constraints while a boundary semi-degenerate field one constraint. As a consequence we see that in order to have no $W$-descendants left we must be in one of the following cases:
    \begin{itemize}
        \item one generic bulk and one semi-degenerate boundary fields correlation function;
        \item one semi-degenerate bulk and one generic boundary fields correlation function;
        \item two generic and one semi-degenerate boundary fields correlation function.
    \end{itemize}    
    For such correlation functions we obtain a differential equation by inserting a boundary fully-degenerate field:
    \begin{theorem}
		Take $t\in\R$ and $\beta^{\ostar}\coloneqq-\chi\omega_1$ with $\chi\in\{\gamma,\frac2\gamma\}$ and assume that Equation~\eqref{eq:condition_mu} holds. Likewise for a boundary semi-degenerate field $V_{\beta^*}(s)$ with $\beta^*=\kappa\omega_2$ we assume that 
    \begin{equation*}
        \mu_{1}(s^-)=\mu_1(s^+).
    \end{equation*}
	Then we have the following BPZ-type differential equations:
    \begin{itemize}
	    \item \textit{Bulk-boundary correlator:} take $\alpha\in\R^2$ such that $(\alpha,\beta^*,\beta^{\ostar})\in\mc A_{1,2}$. Then:
        \begin{equation*}
		\ps{V_{\alpha}(i)V_{\beta^*}(+\infty)V_{\beta^{\ostar}}(t)}=\norm{t-i}^{-\ps{\beta,\alpha}}\mc H(u(t))
	\end{equation*}
	where $u(t)\coloneqq\frac{1}{1+t^2}$ and $\mc H$ is a solution of a differential hypergeometric equation:
	\begin{equation}\label{eq:hypergeometric1}
		\begin{split}
        &\Big[u\left(A_1+u\partial\right)\left(A_2+u\partial\right)\left(A_3+u\partial\right)-\left(B_1-1+u\partial\right)\left(B_2-1+u\partial\right)u\partial\Big]\mathcal H(u)=0.
		\end{split}
	\end{equation}
	Here the quantities involved are given by
	\begin{equation*}
		\begin{split}
			&A_i\coloneqq \frac\chi4\ps{h_1,2\alpha+\beta^*+\beta^{\ostar}-2Q}+\frac\chi2\ps{h_1-h_i,Q-\alpha},\\
            &B_1\coloneqq \frac12,\quad B_2\coloneqq1+\frac\chi4\ps{e_2,\beta^*-2Q}.
		\end{split}
	\end{equation*}
    \item \textit{Three-point boundary correlation function:} take $\beta_1,\beta_2\in\R^2$ in such a way that $(\beta_1,\beta_2,\beta^*,\beta^{\ostar})\in\mc A_{0,3}$. Then:
     \begin{equation*}
		\ps{V_{\beta_1}(0)V_{\beta^*}(1)V_{\beta_2}(+\infty)V_{\beta^{*}}(t)}=\norm{t}^{\frac\chi2\ps{\omega_1,\beta_1}}\norm{t-1}^{\frac\chi2\ps{\omega_1,\beta_2}}\mc H(t)
	\end{equation*}
	with $\mc H$ is a solution of a differential hypergeometric equation:
	\begin{equation}\label{eq:hypergeometric2}
		\begin{split}
		&\Big[t\left(A_1+t\partial\right)\left(A_2+t\partial\right)\left(A_3+t\partial\right)-\left(B_1-1+t\partial\right)\left(B_2-1+t\partial\right)t\partial\Big]\mathcal H(t)=0
		\end{split}
	\end{equation}
	and where the coefficients that appear are given by
	\begin{equation*}
		\begin{split}
			&A_i\coloneqq \frac\chi2\ps{h_1,\beta_1+\beta_2+\beta_3+\beta^*-2Q}+\frac\chi2\ps{h_i-h_1,\beta_1-Q},\\
            &B_i\coloneqq1+\frac\chi2\ps{h_1-h_i,\beta_2-Q}.
		\end{split}
	\end{equation*}
	\end{itemize}
	\end{theorem}
    The proof of this statement is made by starting from Equation~\eqref{eq:BPZ_general}. We then express the left-hand side, by means of the global Ward identities from Theorem~\ref{thm:ward_intro}, in terms of $W_{-i}V_\beta$, $i=1,2$, and $W_{-1}V_{\beta^*}$ where the latter is a semi-degenerate field. We then replace the expression of these descendants using Theorems~\ref{thm:deg1} and~\ref{thm:deg2}, which allows us to show that the correlation function under consideration is a solution of a differential equation. The exact computations leading to these equations are done using Mathematica. 

    \subsubsection{Perspectives of research: integrability of bouyndary Toda CFT}
    The above statement paves the way for the computation of a family of structure constants for the $\sl_3$ boundary  Toda CFT. Indeed, based on these differential equations we expect to be able to provide explicit expressions for the following correlation functions:
    \begin{itemize}
        \item the bulk-boundary correlator: $\ps{V_\alpha(i)V_{\beta^*}(0)}$,
        \item the boundary three-point function: $\ps{V_{\beta_1}(0)V_{\beta^*}(1)V_{\beta_2}(\infty)}$.
    \end{itemize}
    However in order for these quantities to be well-defined one first needs to extend the range of parameters for which the correlations functions make sense. We plan to address this question in a work in progress. In the same fashion as in \cite{Toda_correl1}, this would involve the study of the so-called \emph{reflection coefficients} of the $\sl_3$ boundary Toda CFT. One key point is to show that these coefficients describe the joint tail expansion for correlated bulk and boundary GMC measures, and as such we expect that this would unveil new probabilistic features of these objects. At this stage we should mention that it is not clear yet to what extent the derivation of such structure constants would involve the mating-of-trees framework developed in~\cite{ARSZ} to address a similar issue for boundary Liouville CFT.

    Once the range of definition for these correlation functions appropriately extended we would then combine Operator Product Expansions with a priori study of these differential equations to end up with explicit expressions for the structure constants. This would hopefully allow us to derive new formulas compared to the existing literature.

    \subsubsection{Other directions of work}
    On a different perspective, we also plan to address some questions more directly related to the higher equations of motion we obtained in the present document. Namely we have shown the existence of singular vectors as expected from representation theory of the $W_3$ algebra. It would be an interesting question to define precisely the $W_3$-modules associated to the probabilistic definition of the $\sl$ boundary Toda CFT and describe their structures in agreement with the results obtained here. Likewise it would be an interesting question to understand the features of the singular vectors at higher orders that one expects to find in the theory.

    \textit{\textbf{Acknowledgements:}}	
    The authors would like to thank Rémi Rhodes for his interest and support during the preparation of the present manuscript.
    
	B.C. has been supported by Eccellenza grant 194648 of the Swiss National Science Foundation and is a member of NCCR SwissMAP. B.C. would like to thank Université d'Aix-Marseille for their hospitality (and the very enjoyable sunny weather) and N.H. is grateful to the \'Ecole Polytechnique Fédérale de Lausanne where part of this work has been undertaken.  
	
	\section{Background and preliminary reminders}
    In this section we present the probabilistic construction of the $\sl_3$ boundary Toda CFT and provide some reminders on the rigorous definition of the correlation functions as well as the descendant fields from~\cite{CH_sym1}. We will also describe the method at the heart of the derivation of the singular vectors of the theory.
    
    \subsection{Probabilistic definition of the correlation functions}
    Before explaining the method based on which the descendant fields were constructed in~\cite{CH_sym1} we first briefly recall the probabilistic definition of the correlation functions introduced in~\cite{CH_construction} and used in~\cite{CH_sym1}, to which we refer for additional details.

    \subsubsection{The probabilistic background}
    As explained in the introduction the Toda field is formally speaking a function $\Phi:\overline \H\to\a$ where $\a$ can be identified with $\R^2$. 
    In order to define it as a well-defined random (generalized) function over $\overline \H$, we rely on \textit{Gaussian Free Fields} (GFFs) and their exponentials, defined as \textit{Gaussian Multiplicative Chaos} (GMC) measures. To be more specific the Toda field is defined by first considering 
    \begin{equation}
        \Phi \coloneqq \X - 2Q\ln \norm{\cdot}_+ +\bm{c},
    \end{equation}
    where $\bm{c}\in\R^2$ is a constant called the \textit{zero-mode} and $\X$ is a GFF over $\overline \H$ with Neumann boundary conditions and taking values in $\R^2$. This GFF is a Gaussian random distribution (and more precisely belongs to the Sobolev space with negative index $H^{-1}(\H\to \R^2)$) that is centered and with covariance kernel:
	\begin{equation}
		\begin{split}&\expect{\ps{u,\X(x)}\ps{v,\X(y)}}=\ps{u,v}G(x,y),\qt{with}G(x,y) = G_{\hat{\C}}(x,y)+G_{\hat{\C}}(x,\bar{y})
        \end{split}
	\end{equation} 
    for $u,v\in\R^2$ and $x,y$ in $\overline\H$, and with $G_{\hat{\C}}(x,y)=\ln\frac{1}{\norm{x-y}}+\ln\norm{x}_++\ln\norm{y}_+$ where we have used the notation $\norm{x}_+\coloneqq\max(\norm{x},1)$. 
    It will be convenient in the sequel to regularize it into a smooth function: this can be done based on a smooth mollifier $\eta$ by setting for positive $\rho$ and $\eta_\rho(\cdot)\coloneqq\frac1{\rho^2}\eta(\frac{\cdot}{\rho})$:
	\begin{equation}\label{eq:regularized}
		\X_\rho(x)=\X\star\eta_\rho \coloneqq\int_\H\X_\rho(y)\eta_\rho(x-y)dyd\bar y.
	\end{equation}
	We will also work with the regularized Toda field $\Phi_\rho\coloneqq \Phi\star\eta_\rho$.

    In order to make sense of the Toda correlation functions we will also rely on the theory of \textit{Gaussian Multiplicative Chaos}.  Thanks to it we can define a family of random measures on $\H$ and its boundary $\R$ by considering the following limits~\cite{Ber,RV_GMC} (in probability and in the sense of weak convergence of measures):
	\begin{equation*}
		M_{\gamma e_i}(d^2x)\coloneqq\lim\limits_{\rho\to0}\rho^{\gamma^2}e^{\ps{\gamma e_i,\X_\rho(x)}}dxd\bar x;\quad M^\partial_{\gamma e_i}(dx)\coloneqq\lim\limits_{\rho\to0}\rho^{\frac{\gamma^2}{2}}e^{\ps{\frac\gamma2 e_i,\X_\rho(x)}}dx
	\end{equation*}
	where the $(e_i)_{i=1,2}$ are special elements of $\R^2$ inherited from the Lie algebra structure and defined below. Since $\norm{e_i}^2=2$ we need to assume that $\gamma<\sqrt 2$ to ensure that these random measures are well-defined.

    \subsubsection{On the Lie algebra $\mathfrak{sl}_3$}
    The Toda field considered above takes values in $\mathfrak{a}$, which is the real part of the Cartan subalgebra of the Lie algebra $\sl_3(\C)$. It is identified with $\R^2$ equipped with a basis of so-called \emph{simple roots} $(e_1,e_2)$, and with a scalar product $\ps{\cdot,\cdot}$ given on the simple roots by the Cartan matrix
    \begin{equation*}
		\left(\ps{e_i,e_j}\right)_{i,j}\coloneqq A=\begin{pmatrix}
			2 & -1\\
			-1 & 2
		\end{pmatrix}
	\end{equation*}
    We will also consider the dual basis of $(e_1,e_2)$ which we denote by $(\omega_1,\omega_2)$, and the Weyl vector $\weyl= \omega_1+\omega_2=e_1+e_2$. The background charge $Q$ is then given in terms of the Weyl vector and the coupling constant $\gamma\in(0,\sqrt2)$:
    \begin{equation*}
        Q=\left(\gamma+\frac2\gamma\right)\weyl.
    \end{equation*}
	The fundamental weights in the first fundamental representation $\pi_1$ of $\mathfrak{sl}_3$ with the highest weight $\omega_1$ are given by:
	\begin{equation}\label{eq:definition_hi}
		h_1\coloneqq \frac{2e_1+e_2}{3},\quad h_2\coloneqq \frac{-e_1+e_2}{3}, \quad h_3\coloneqq  -\frac{e_1+2e_2}{3}\cdot
	\end{equation}

    \subsubsection{Definition of the Toda field and of the correlation functions} 
    The mathematical interpretation of the path integral~\eqref{eq:path_integral} has been conducted in~\cite{CH_construction}, where the Toda field $\Phi$ was rigorously constructed based on the objects introduced above. To do so this random distribution was defined by the setting for suitable $F$
    \begin{equation} \label{eq:toda_field}
        \begin{split}
            &\ps{F[\Phi]}=\int_{\R^2} e^{-\ps{Q,\bm{c}}}\E\left[ F(\Phi)\exp\left( -\sum_{i=1}^2 \mu_{B,i} e^{\ps{\gamma e_i,\bm{c}}}M_{\gamma e_i}(\H) + \mu_{i}e^{\ps{\frac{\gamma}{2}e_i,\bm{c}}}M_{\gamma e_i}^\partial(\R)\right)\right]d\bm{c}.
        \end{split}
    \end{equation}
    Based on this defining expression the correlation functions of Vertex Operators are then defined by first considering insertion points $z_1,\cdots,z_N\in\H$ and $s_1,\cdots,s_M\in\R$ to which we associate weights $\alpha_1,\cdots,\alpha_N$ and $\beta_1,\cdots,\beta_M$ in $\R^2$. They are then formally defined by considering the functionals
    \begin{equation*}
        F[\Phi]=\prod_{k=1}^NV_{\alpha_k}(z_k)\prod_{l=1}^MV_{\beta_l}(s_l)\coloneqq \prod_{k=1}^N e^{\ps{\alpha_k,\Phi(z_k)}}\prod_{l=1}^M e^{\ps{\frac{\beta_l}2,\Phi(s_l)}}.
    \end{equation*}
    The corresponding correlation function is then obtained by replacing $F$ by its expression in~\eqref{eq:toda_field}: we denote it by $\ps{\prod_{k=1}^NV_{\alpha_k}(z_k)\prod_{l=1}^MV_{\beta_l}(s_l)}$.

    In greater generality in the boundary case we have additional degrees of freedom that come from the boundary cosmological constants. To be more specific we can actually choose the boundary cosmological constants $\mu_i$ to be piecewise constant, with jumps precisely at the boundary insertion points. To be more specific, let us pick bulk cosmological constants $\mu_{B,i}\in\C$ for $i=1,2$, as well as boundary cosmological constants $\mu_{i,l}\in\C$ for $1\leq l\leq M$. For the correlation functions to be well-defined we further assume that 
    \begin{eqs} 
		&\mu_{B,i} > 0 \qt{and} \Re \ \mu_{i,l} \ge 0 \qt{for all} l=1,...,M,\qt{as well as}\\
        &\ps{\bm{s},\omega_i} > 0\qt{and}\ \ps{\alpha_k - Q,e_i} < 0 \qt{for all} k=1,...,2N+M
	\end{eqs}
    with $\bm{s} \coloneqq \sum \alpha_k + \frac12 \sum \beta_l -Q$.
    We denote by $\mc A_{N,M}$ the set of weights satisfying the above assumptions and define $\mu_i$ to be the piecewise constant measure on $\R$ given by
    $$
        \mu_i(dx) = \sum_{l=1}^M \mu_{i,l+1}\mathds{1}_{(s_l,s_{l+1})}(x) dx
    $$
    with the convention that $s_0=-\infty$, $s_{M+1}=+\infty$ and $\mu_{M+1}=\mu_1$.
    In agreement with~\cite{CH_sym1} the correlation functions are then defined first at the regularized level by considering
    \begin{equation} \label{eq:reg correl shifted}
        \begin{split}
            &\ps{F[\Phi]\prod_{k=1}^NV_{\alpha_k}(z_k)\prod_{l=1}^MV_{\beta_l}(s_l)}_{\delta,\eps,\rho} \coloneqq C(\bm{z},\boldsymbol{\alpha})\int_{\R^2} e^{\ps{\bm{s},\bm{c}}}\E\Big[ F[\Phi_\rho + H]\times\\
            &\left.\exp\left( -\sum_{i=1}^2 \mu_{B,i} e^{\ps{\gamma e_i,\bm{c}}}\int_{\Heps} Z_{i,\rho}\rho^{\gamma^2} e^{\ps{\gamma e_i,\X_\rho(x)}} dxd\bar{x} + e^{\ps{\frac{\gamma}{2}e_i,\bm{c}}}\int_{\Reps} Z^\partial_{i,\rho}\rho^{\frac{\gamma^2}{2}} e^{\ps{\frac{\gamma}{2} e_i,\X_\rho(x)}} \mu_i(dx)\right)\right]d\bm{c}
        \end{split}
    \end{equation}
    where, using the shorthand $\bm z=(z_k)_{k=1,...,2N+M}$ for $(z_1,...,z_N,\bar{z}_1,...,\bar{z}_N,s_1,...,s_M)$ and $\bm\alpha=(\alpha_k)_{k=1,...,2N+M}$ for $(\alpha_1,...,\alpha_N,\alpha_1,...,\alpha_N,\beta_1,...,\beta_M)$,  we have set:
    $$
    C(\bm{z},\boldsymbol{\alpha}) = \prod_{k< l} |z_k-z_l|^{-\ps{\alpha_k,\alpha_l}} \prod_{k=1}^M |z_k-\bar{z}_k|^{\frac{|\alpha_k|^2}{2}},\quad 
    H = \sum_{k=1}^{2N+M} \alpha_k G_{\hat{\C}}(\cdot,z_k),
    $$
    $$
    Z_{i,\rho} = \prod_{k=1}^{2N+M} \left( \frac{|x|_+}{|z_k-x|}\right)^{\ps{\gamma e_i,\alpha_k}} \qt{and} Z^\partial_{i,\rho} = \prod_{k=1}^{2N+M} \left( \frac{|x|_+}{|z_k-x|}\right)^{\ps{\frac{\gamma}{2} e_i,\alpha_k}}.
    $$
    The dependence in the parameter $\rho$ stands for the regularization of the Toda field, while the ones in $\delta,\eps$ apply to the domains of integration $\Heps$ and $\Reps$ which are meant to avoid the singularities at the insertion points. They are defined by 
    \begin{equation*}
    \Heps = \left(\H + i\delta\right) \setminus \bigcup_{k=1}^N B(z_k,\eps)\qt{and}\Reps = \R \setminus \bigcup_{l=1}^M (s_l-\eps,s_l+\eps),
    \end{equation*}
    with $\eps$ and $\delta$ small enough in such a way that all the $B(z_k,\eps)$ lie in $\Heps$ and are disjoint, and likewise all the intervals $(s_l-\eps,s_l+\eps)$ are disjoint. 
    
    \subsection{Construction of the descendant fields}
    Based on the above we then wish to define correlation functions containing descendant fields of the Vertex Operators. To this end we recall here the method of construction of these quantities carried out in~\cite{CH_sym1} and provide the necessary background from there required to make make sense of the singular vectors of the theory.

    \subsubsection{Descendant fields at the regularized level}
    The definition of the descendant fields is based on the consideration of regularized expressions of the form
    \[
        \ps{F[\Phi]V_\beta(t)\prod_{k=1}^NV_{\alpha_k}(z_k)\prod_{l=1}^MV_{\beta_l}(s_l)}_{\delta,\eps,\rho}
    \]
    where $F[\Phi]$ is a polynomial in the derivatives of the Toda field $\Phi$ evaluated at $t$. This polynomial is determined by conducting algebraic computations based on the probabilistic representation of the currents $\SET$ and $\Wb$, see~\cite[Subsection 2.2.2]{CH_sym1} for more details on this algorithmic procedure. Conducting these manipulations leads to the consideration of functionals of the Toda field of the form 
	\begin{equation*}
		\L_{-\lambda}V_\alpha[\Phi]=\sum_{\norm{\nu}=\norm{\lambda}}\L_{-\nu;\lambda}^\alpha\left(\partial^{\nu_1}\Phi,\cdots,\partial^{\nu_l}\Phi\right)V_\alpha[\Phi]
	\end{equation*}
	where the $\L_{-\nu;\lambda}^\alpha$ are multilinear forms, while $\lambda$ is a Young diagram, that is $\lambda=(\lambda_1,\cdots,\lambda_l)$ with $\lambda_1\geq\cdots\geq\lambda_l$, and where the sum ranges over Young diagrams $\nu$ with $\norm{\nu}\coloneqq\nu_1+\cdots+\nu_l=\norm{\lambda}$. In the present article we will only need to consider descendants of the form $\L_{-n}V_\beta$ for $n\geq 1$ as well as $\L_{-(1,n)}V_\beta$ for $n=1,2$ and $\L_{-(1,1,1)}$, in which case these multilinear forms are given by
	\begin{equation}
		\begin{split}
            &\L_{-n}^\alpha[\Phi]=\left\langle(n-1)Q+\alpha,\frac{\partial^n\Phi}{(n-1)!}\right\rangle-\sum_{i=0}^{n-2}\left\langle\frac{\partial^{i+1}\Phi}{i!},\frac{\partial^{n-i-1}\Phi}{(n-2-i)!}\right\rangle\\
            &\L_{-(1,1)}^\alpha[\Phi]\coloneqq \ps{\alpha,\partial^2\Phi}+ \left(\ps{\alpha,\partial\Phi}\right)^2,\\
			&\L_{-(1,2)}^\alpha[\Phi]\coloneqq \ps{Q+\alpha,\partial^3\Phi}+ \ps{Q+\alpha,\partial^2\Phi}\ps{\alpha,\partial\Phi}-2\ps{\partial^2\Phi,\partial\Phi} -\ps{\partial\Phi,\partial\Phi}\ps{\alpha,\partial\Phi},\\
			&\L_{-(1,1,1)}^\alpha[\Phi]\coloneqq \ps{\alpha,\partial^3\Phi}+ 3\ps{\alpha,\partial^2\Phi}\ps{\alpha,\partial\Phi}+\left(\ps{\alpha, \partial\Phi}\right)^3.
		\end{split}
	\end{equation}
    A similar procedure remains valid for the descendant fields associated to the higher-spin current, but the defining expression for them then involves quantities related to the Lie algebra structure. We don't provide their expression here since they won't be needed in the rest of the document ---they are given in~\cite[Equation (2.7)]{CH_sym1}.

    To make sense of the insertion of such quantities within correlation functions we first need to define such quantities at the regularized level. This is done by setting for $\delta,\eps,\rho>0$ fixed:
    \begin{eqs}\label{eq:desc_reg}
		&\ps{\L_{-\lambda}V_\beta(t)\V}_{\delta,\eps,\rho} \coloneqq\sum_{\norm{\nu}=\norm{\lambda}}\ps{:\L_{-\nu;\lambda}^\alpha\left(\partial^{\nu_1}\Phi,\cdots,\partial^{\nu_l}\Phi\right)V_\beta(t):\V}.
	\end{eqs}
    where the notation $:F[\Phi]V_\beta(t):$ refers to $F[\Phi+\beta\ln\norm{\cdot-t}]V_\beta(t)$, while we introduce the shorthand $\V\coloneqq\prod_{k=1}^NV_{\alpha_k}(z_k)\prod_{l=1}^MV_{\beta_l}(s_l)$. This can be expressed explicitly based on the following Gaussian integration by parts formula~\cite[Lemma 2.3]{CH_sym1}:
	\begin{lemcite}\label{lemma:GaussianIPP}
		Assume that $\bm\alpha$ belongs to $\mc A_{N,M}$. Then for any $t\in\R\setminus\{s_1,\cdots,s_M\}$, $p_1,\cdots,p_m$ positive integers and $u\in\R^2$:
		\begin{equation}\label{eq:IPP_product}
		\begin{split}
			&\lim\limits_{\rho\to0}\frac{1}{(p_1-1)!}\Big\langle:\prod_{l=1}^m\ps{u_l,\partial^{p_l}\Phi(t)}:\V\Big\rangle_{\rho,\eps,\delta}=\sum_{k=1}^{2N+M}\frac{\ps{u_1,\alpha_k}}{2(z_k-t)^{p_1}}\ps{:\prod_{l=2}^m\ps{u_l,\partial^{p_l}}\Phi(t):\V}_{\eps,\delta}\\
			& -\sum_{i=1}^2\mu_{B,i}
			\int_{\Heps}\left(\frac{\ps{u_1,\gamma e_i}}{2(x-t)^{p_1}}+\frac{\ps{u_1,\gamma e_i}}{2(\bar x-t)^{p_1}}\right)\ps{:\prod_{l=2}^m\ps{u_l,\partial^{p_l}\Phi(t)}:V_{\gamma e_i}(x)\V}_{\eps,\delta}dxd\bar x\\
			&-\sum_{i=1}^2
			\int_{\Reps}\frac{\ps{u_1,\gamma e_i}}{2(x-t)^{p_1}}\ps{:\prod_{l=2}^m\ps{u_l,\partial^{p_l}\Phi(t)}:V_{\gamma e_i}(x)\V}_{\eps,\delta}\mu_{i}(dx).
		\end{split}
	\end{equation}
	\end{lemcite}
    
	\subsubsection{Method for defining the descendant fields}
    The method used to define the descendants is strictly the same as in the companion paper~\cite{CH_sym1} so we will only briefly recall the arguments here and refer to~\cite[Subsection 2.3]{CH_sym1} for more details. The starting point is the defining expression~\eqref{eq:desc_reg} of the correlation functions of primary fields with insertion of a descendant field $\Db_{-\lambda}V_\beta(t)$, $\ps{\Db_{-\lambda}V_\beta(t)\V}_{\delta,\eps,\rho}$ . We then use the Gaussian integration by parts formula from Lemma~\ref{lemma:GaussianIPP} to provide an explicit expression for $\ps{\Db_{-\lambda}V_\beta(t)\V}_{\delta,\eps,\rho}$. Based on this expression we can carry algebraic manipulations to express the latter as a sum
    $$
        \ps{\Db_{-\lambda}V_\beta(t)\V}_{\delta,\eps,\rho} = (P)\text{-class terms} + \tilde{\mathfrak{D}}_{-\lambda;\delta,\eps,\rho}(\bm{\alpha})
    $$
    where $\tilde{\mathfrak{D}}_{-\lambda;\delta,\eps,\rho}(\bm{\alpha})$ is an explicit remainder term, while $(P)$-class terms refer to quantities $F_{\delta,\eps,\rho}(\bm\alpha)$ such that~\cite[Definition 2.6]{CH_sym1}:
    \begin{enumerate}
			\item for any $\bm\alpha\in\mc A_{N,M}$, the limit $F(\bm\alpha)\coloneqq\lim\limits_{\delta,\eps,\rho\to0}F_{\delta,\eps,\rho}(\bm\alpha)$ exists and is finite as $\rho$, $\eps$ and then $\delta\to0$;
			\item the map $\bm\alpha\mapsto F(\bm\alpha)$ is analytic in a complex neighborhood of $\mc A_{N,M}$.
		\end{enumerate}
    We then define the correlation function $\ps{\Db_{-\lambda}V_\beta(t)\V}$ as the limit as $\rho$, $\eps$ and then $\delta\to 0$ of the quantity
    \begin{equation}\label{eq:defining_limit}
        \ps{\Db_{-\lambda}V_\beta(t)\V}\coloneqq\ps{\Db_{-\lambda}V_\beta(t)\V}_{\delta,\eps,\rho} - \tilde{\mathfrak{D}}_{-\lambda;\delta,\eps,\rho}(\bm{\alpha})
    \end{equation}
    which is well-defined and analytic in a neighborhood of the parameter set. 
    
    The most involved part is by far to find the expression of the remainder term $\tilde{\mathfrak{D}}_{-\lambda;\delta,\eps,\rho}(\bm{\alpha})$, that is to identify the singular terms that appear in the expansion of $\ps{\Db_{-\lambda}V_\beta(t)\V}_{\delta,\eps,\rho}$. 
    To make this explicit we introduce the notations, for $r>0$ independent of $\rho,\eps,\delta$ but taken small enough,
	\begin{equation*}\label{eq:def_It}
		\begin{split}
			&\It F(t;\cdot)\coloneqq \sum_{i=1,2}\int_{\R_\eps}F^{i}(t;s)\mu_i(ds)+\mu_{B,i}\int_{\Heps}\left(F^i(t;x)+F^i(t;\bar x)\right)d^2x,\\
			&\Is F(t;\cdot)\coloneqq \sum_{i=1,2}\int_{\R_\eps^1}F^{i}(t;s)\mu_i(ds)+\mu_{B,i}\int_{\Heps^1}\left(F^i(t;x)+F^i(t;\bar x)\right)d^2x\qt{and}\\
			&\Ir F(t;\cdot)\coloneqq \sum_{i=1,2}\int_{\R_\eps^c}F^{i}(t;s)\mu_i(ds)+\mu_{B,i}\int_{\Heps^c}\left(F^i(t;x)+F^i(t;\bar x)\right)d^2x
		\end{split}
	\end{equation*}
	with $\Reps^1\coloneqq\Reps\cap B(t,r)$, $\Heps^1\coloneqq \Heps\cap B(t,r)$ and $\Heps^c\coloneqq\Heps\setminus\Heps^1$. 
	In the above $F(t;\cdot)$ indicates a pair of functions $(F^1,F^2)$, which are the ones coming from the Gaussian integration by parts formula from Lemma~\ref{lemma:GaussianIPP}. Iteration of this formula naturally leads to the generalization of the above for tuples of functions $(F^{i_1,\cdots,i_p}(t;x_1,\cdots,x_p))_{\bm i\in\{1,2\}^p}$:
	\begin{equation*}
		\begin{split}
			\Is^p F(t;\cdot)\coloneqq \sum_{\bm i \in\{1,2\}^p}\sum_{\substack{k_1,k_2,k_3\geq0\\ k_1+k_2+k_3=p}}&\int_{\mc R_{k_1}\times\mc H_{k_2,k_3}}\prod_{l=1}^{k_1}\mu_{i_l}(ds_l)\prod_{l=k_1+1}^{p} \mu_{B,i_l}d^2x_l\\
			&F^{\bm i}(t;s_1,\cdots,s_{k_1},x_{k_1+1}\cdots,x_{k_1+k_2},\bar x_{k_1+k_2+1},\cdots,\bar x_p)
		\end{split}
	\end{equation*}
    with $\mc R_{k_1}=\left(\Reps\cap B(t,r_1)\right)^{k_1}$ and similarly for $\mc H_{k_2+k_3}$.
	More generally we will consider quantities of the form 
	\[
	\It^{p_1}\times\Is^{p_2}\times\Ir^{p_3}F(t;\cdot)    
	\]
	for an integrand similar to the previous one but where now the first $p_1$ integrals are evaluated over either $\Reps$ or $\Heps$, the following $p_2$ ones are taken around the insertion $t$ and the last $p_3$ ones take values away from $t$.
    We will also extensively use the notation, for indices $i_k\in\{1,2\}$ and distinct insertions $x_k\in \H\cup\R$,
    \begin{equation}\label{eq:def_Psi}
        \Psi_{i_1,\cdots,i_l}(x_1,\cdots,x_l)\coloneqq \ps{V_{\gamma e_{i_1}(x_1)}\cdots V_{\gamma e_{i_l}(x_l)}V_\beta(t)\V}_{\rho,\eps,\delta}.
    \end{equation}

    \subsubsection{The $\L_{-n}V_\beta$ and $\Wb_{-n}V_\beta$ descendant fields}
    Let us now recall how this method was implemented in~\cite{CH_sym1} in order to make sense of the $\L_{-n}V_\beta$ and $\Wb_{-n}V_\beta$ descendant fields. To start with using the Gaussian integration by parts formula we have
    \begin{align*}
		&\ps{\L_{-n}V_\beta(t)\prod_{k=1}^NV_{\alpha_k}(z_k)\prod_{l=1}^MV_{\beta_l}(s_l)}_{\delta,\eps}=\\
		&\left[\sum_{k=1}^{2N+M} \frac{(n-1)\Delta_{\alpha_k}}{2(z_k-t)^{n}}-\sum_{k=1}^{2N+M} \left(\frac{\ps{\beta,\alpha_k}}{2(z_k-t)^{n}}+ \sum_{p=1}^{n-1}\sum_{l\neq k}\frac{\ps{\alpha_k,\alpha_l}}{4(z_k-t)^p(z_l-t)^{n-p}}\right)\right]\ps{V_\beta(t)\V}_{\delta,\eps} \\
        &- \It \left[\left(\frac{n-1+\frac{\ps{\beta,\gamma e_i}}2}{(x-t)^{n}}-\sum_{p=1}^{n-1}\sum_{k=1}^{2N+M}\frac{\ps{\alpha_k,\gamma e_i}}{2(x-t)^p(z_k-t)^{n-p}}\right)\Psi_i(x)\right]\\
		&+ \sum_{i=1}^2\mu_{B,i}\int_{\Heps} \sum_{p=1}^{n-1}\frac{\ps{\gamma e_i,\gamma e_i}}{2(x-t)^p(\bar x-t)^{n-p}}\Psi_i(x)d^2x-\It^2\left[\sum_{p=1}^{n-1}\frac{\ps{\gamma e_i,\gamma e_j}}{4(x-t)^p(y-t)^{n-p}}\Psi_{i,j}(x,y)\right].
	\end{align*}
    Using algebraic manipulations~\cite[Equation (3.4)]{CH_sym1} the latter can be transformed into
    \[
        \ps{\L_{-n}V_\beta(t)\V}_{\delta,\eps,\rho}=A_{\delta,\eps,\rho}+\mathfrak{L}_{-n;\delta,\eps,\rho}(\bm{\alpha})
    \]
    where $A_{\delta,\eps,\rho}$ is a $(P)$-class term explicitly given by
    \begin{equation}\label{eq:expr_Ln}
        \begin{split}
            &A_{\delta,\eps}\coloneqq\left(\sum_{k=1}^{2N+M}\frac{\ps{(n-1)Q+\beta,\alpha_k}}{2(z_k-t)^n}-\sum_{k,l=1}^{2N+M}\sum_{p=1}^{n-1}\frac{\ps{\alpha_k,\alpha_l}}{4(z_k-t)^p(z_l-t)^{n-p}}\right)\ps{V_\beta(t)\V}_{\delta,\eps}\\
            &-\Ir\left[\left(\frac{n-1+\frac{\ps{\beta,\gamma e_i}}2}{(x-t)^n}-\sum_{k=1}^{2N+M}\sum_{p=1}^{n-1}\frac{\ps{\gamma e_i,\alpha_k}}{2(x-t)^p(z_k-t)^{n-p}}\right)\Psi_i(x)\right]\\
            &-\Is\left[\sum_{k=1}^{2N+M}\frac{\ps{\gamma e_i,\alpha_k}}{2(z_k-t)^{n-1}(z_k-x)}\Psi_i(x)\right]\\
            &+\Ir\times\Ir\left[-\sum_{p=1}^{n-1}\frac{\ps{\gamma e_i,\gamma e_j}}{4(x-t)^{p}(y-t)^{n-p}}\Psi_{i,j}(x,y)\right]\\
            &+\Is\times\Ir\left[\frac{\ps{\gamma e_i,\gamma e_j}}{2(y-t)^{n-1}(y-x)}\Psi_{i,j}(x,y)\right]
        \end{split}
     \end{equation}
    while $\mathfrak{L}_{-n;\delta,\eps,\rho}(\bm{\alpha})$ is a remainder term that takes the form
    \begin{equation}\label{eq:rem_Ln}
        \mathfrak{L}_{-n;\delta,\eps,\rho}(\bm{\alpha})\coloneqq \Is\left[\partial_x\left(\frac{1}{(x-t)^{n-1}} \ps{V_{\gamma e_i}(x)V_\beta(t)\V}_{\delta,\eps,\rho} \right)\right].
    \end{equation}

    Likewise we can write a similar expansion for the $\Wb_{-n}V_\beta$ descendant. It takes the form of the equality
    \[
        \ps{\Wb_{-n}V_\beta(t)\V}_{\delta,\eps,\rho}=A_{\delta,\eps,\rho}+\mathfrak{W}_{-n;\delta,\eps,\rho}(\bm{\alpha})
    \]
    where $A_{\delta,\eps,\rho}$ is a $(P)$-class term that is equal~\cite[Lemma 3.1]{CH_sym1}, for $n=1$, to
    \begin{equation}\label{eq:expr_W1}
            \begin{split}
                &\sum_{k=1}^N\frac{\Wb_{-1}^\beta(\alpha_k)}{2(z_k-t)}\ps{V_\beta(t)\V}_{\delta,\eps}-\Ir\left[\frac{\Wb_{-1}^\beta(\gamma e_i)}{2(x-t)}\ps{V_{\gamma e_i}(x)V_\beta(t)\V}_{\delta,\eps}\right]\\
                &+\Is\left[h_2(e_i)\left(2\omega_{\hat i}(\beta)-q\right)\sum_{k=1}^{2N+M}\frac{\ps{\alpha_k,\gamma e_i}}{2(x-z_k)}\ps{V_{\gamma e_i}(x)V_\beta(t)\V}_{\delta,\eps}\right]\\
                &-\Is\times\It\left[h_2(e_i)\left(2\omega_{\hat i}(\beta)-q\right)\frac{\ps{\gamma e_i,\gamma e_j}}{2(x-y)}\ps{V_{\gamma e_i}(x)V_{\gamma e_j}(y)V_\beta(t)\V}_{\delta,\eps}\right].
            \end{split}
        \end{equation}
    while $\mathfrak{W}_{-n;\delta,\eps,\rho}(\bm{\alpha})$ is a remainder term that takes the form
    \begin{eqs}\label{eq:rem_Wn}
        \mathfrak{W}^i_{-n,\delta,\eps,\rho}(\bm\alpha)&\coloneqq\Is\left[\partial_x\left(h_2(e_i)\left(\frac{(n-2)q+2\omega_{\hat{i}}(\beta)}{(x-t)^{n-1}}+\sum_{k=1}^{2N+M}\frac{2\omega_{\hat{i}}(\alpha_k)\mathds1_{n\geq 2}}{(x-t)^{n-2}(x-z_k)}\right)\Psi_{i}(x)\right)\right]\\
			&+\Is\times \It\left[\partial_x\left(\frac{2\gamma h_2(e_i)\delta_{i\neq j}\mathds1_{n\geq 2}}{(x-t)^{n-2}(y-x)}\Psi_{i,j}(x,y)\right)\right].
    \end{eqs}
    To discard the divergent quantities in these remainder terms we then use Stokes' formula as explained in \cite[Subsection 2.3.3]{CH_sym1}, and rely on fusion estimates (Lemma~\ref{lemma:fusion}) to prove that the defining limit \eqref{eq:defining_limit} is well-defined and analytic in a complex neighborhood of $\mc{A}_{N,M}$. For example the $\L_{-n}V_\beta$ descendant is defined by the limit~\cite[Definition 3.6]{CH_sym1}
    \begin{equation}\label{eq:def_descLn}
			\begin{split}
				&\ps{\L_{-n}V_\beta(t)\V}\coloneqq\lim\limits_{\delta,\eps,\rho\to0}\ps{\L_{-n}V_\beta(t)\V}_{\delta,\eps,\rho}-\sum_{i=1}^2\tilde{\mathfrak{L}}^i_{-n,\delta,\eps,\rho}(\bm\alpha)\qt{where}
			\end{split}
		\end{equation}
    \begin{equation}\label{eq:tilde_Ln}
		\begin{split}
			&\tilde{\mathfrak{L}}^i_{-n,\delta,\eps,\rho}(\bm\alpha)=\left(\mu_{L,i}\frac{\ps{V_{\gamma e_i}(t-\eps)V_\beta(t)\V}_{\delta,\eps}}{(-\eps)^{n-1}}-\mu_{R,i}\frac{\ps{V_{\gamma e_i}(t+\eps)V_\beta(t)\V}_{\delta,\eps}}{(\eps)^{n-1}}\right)\\
			&-\mu_{B,i}\int_{(t-r,t+r)}\ps{V_{\gamma e_i}(x+i\delta)V_\beta(t)\V}_{\delta,\eps}\Im\left(\frac{1}{(x+i\delta-t)^{n-1}}\right)dx.
		\end{split}
	\end{equation}
    The $\Wb_{-n}V_\beta$ descendant is defined by a similar limit by substracting the corresponding remainder term~\cite[Definition 3.10]{CH_sym1}.
    
    \subsection{Method for defining the singular and null vectors}\label{subsec:def_descendants}
    Based on the above definition of the descendant fields we can then look for singular vectors. Heuristically, a singular vector is a non-trivial linear combination of the form
    \[
        \sum_{\norm{\nu}=n} c_{\nu}(\alpha) \L_{-\nu}V_\alpha + \sum_{\norm{\lambda}=n} d_{\lambda}(\alpha) \Wb_{-\lambda}V_\alpha
    \]
    such that the associated multilinear form over $\R^r$
    \[
        \sum_{\norm{\nu}=n} c_{\nu}(\alpha) \L_{-\nu}^\alpha + \sum_{\norm{\lambda}=n} d_{\lambda}(\alpha) \Wb_{-\lambda}^\alpha
    \]
    vanishes identically. This happens for very specific values of the weight $\alpha$, predicted by the representation theory of the $W_3$ algebra. When inserted within correlation functions such singular vectors then give rise to \textit{higher equations of motion}, and under additional assumptions to BPZ-type differential equations. Let us explain in more details how we prove such claims.  

    \subsubsection{Definition of the singular vectors}
    The existence of singular vectors is predicted by the algebra and more precisely representation theory, but their explicit form is not known in general. Here we will use explicit expressions to construct singular vectors at the levels $n=1$, $2$ and $3$ and that take the form
    \[
        \psi_{n,\alpha}\coloneqq \Wb_{-n}V_\alpha+\sum_{\norm{\nu}=n} c_{\nu}(\alpha) \L_{-\nu}V_\alpha
    \]
    where the $c_\nu(\alpha)$ are explicit, and where the weight $\alpha$ has to be specifically chosen. More precisely we define the singular vector within correlation functions by setting
    \begin{equation*}
        \ps{\psi_{n,\alpha}(t)\V}\coloneqq \ps{\Wb_{-n}V_\alpha(t)\V}+\sum_{\norm{\nu}=n} c_{\nu}(\alpha) \ps{\L_{-\nu}V_\alpha(t)\V},
    \end{equation*}
    where the quantities on the right-hand side are defined using the above regularization procedure.

    \subsubsection{Explicit expression of the singular vector}
    For one of the singular vector considered above, we know that the associated multilinear form 
    \[
        \Wb_{-n}^\alpha+\sum_{\norm{\nu}=n} c_{\nu}(\alpha) \L_{-\nu}^\alpha
    \]
    vanishes identically, so it is tempting to write that $\psi_{n,\alpha}=0$, or put differently that $\psi_{n,\alpha}$ is a \textit{null vector}. This is however generically \textit{false} since the descendants have been defined from distinct regularization procedures. We can however deduce from this fact that the singular vector is described using these remainder terms: 
    \[
        \psi_{n,\alpha}=-\lim\limits_{\delta,\eps,\rho\to0}\left(\tilde{\mathfrak{W}}_{-n;\delta,\eps,\rho}(\bm\alpha)+\sum_{\norm{\nu}=n} c_{\nu}(\alpha)\tilde{\mathfrak{L}}_{-\nu;\delta,\eps,\rho}(\bm\alpha)\right)
    \]
    where in the above we have considered the remainder terms coming from the definition of the different descendants involved.

    It is \textit{a priori} far from obvious that these remainder terms admit well-defined limit. And in fact individually these terms are even divergent! There are however (truly miraculous!) cancellations occurring and in the end we get that the latter converges, and that the corresponding limit is not zero but given by another correlation function, that is defined from the first one but with the weight associated to $V_\alpha(t)$ being changed. In addition it can feature some descendant fields. Put differently this means that the singular vector is actually not a \textit{null vector} but is proportional to a primary or a descendant field.

	\section{Singular vectors and higher equations of motion: levels one and two}\label{sec:sing}
	In \cite{CH_sym1} we have defined descendant fields of the Vertex Operator $V_\beta$ of the form $\L_{-n}V_\beta$ or $\Wb_{-n}V_\beta$ for $n\geq 1$. We proved that thanks to these definitions we were able to express the symmetries of the models into actual constraints on the correlation functions in the form of Ward identities. 
	
	In this section we explain how the existence of singular vectors can give rise to additional constraints on Toda correlation functions. To do so we will start by describing singular vectors at the levels one and two, in which case we will conduct the computations in more details. In order to define the singular vector at level two we will also need to introduce beforehand the descendant field $\L_{-1,1}V_\beta$.
	
	
	\subsection{Singular vector at the level one}
	Singular vectors arise when the weight $\beta$ takes specific values. For instance singular vectors at the level one appear when considering \textit{semi-degenerate fields}, that is Vertex Operators of the form $V_\beta$ where $\beta=\kappa\omega_i$ for $\kappa\in\R$ and $i\in\{1,2\}$. Based on this semi-degenerate field we can formally define the singular vector at level one by setting
	\begin{equation*}
		\psi_{-1,\beta}\coloneqq \left(\Wb_{-1}-\frac{3w(\beta)}{2\Delta_\beta}\L_{-1}\right)V_\beta.
	\end{equation*}
	To be more specific we define it as follows.
	\begin{defi}
		Take $t\in\R$ and $(\beta,\bm\alpha)\in\mc A_{N,M+1}$.  We define the singular vector $\Psi_{-1,\beta}$ within half-plane correlation functions by setting
		\begin{equation*}
			\ps{\psi_{-1,\beta}(t)\V}\coloneqq\ps{\Wb_{-1}V_\beta(t)\V}-\frac{3w(\beta)}{2\Delta_\beta}\ps{\L_{-1}V_\beta(t)\V}.
		\end{equation*}
	\end{defi}
	
	This singular vector is actually not a \textit{null} vector, that is unlike in the free-field theory or bulk Toda CFT where the above correlation functions identically vanish (see~\cite[Proposition 5.1]{Toda_OPEWV}), in boundary Toda CFT this singular vector is not zero but rather given by a primary field. Namely we have the following statement:
	\begin{theorem}\label{thm:deg1}
		Assume that $\beta=\kappa\omega_1$ for $\kappa<q$. Then 
		\begin{equation*}
			\psi_{-1,\beta}=2(q-\kappa)(\mu_{L,2}-\mu_{R,2})V_{\beta+\gamma e_2}    
		\end{equation*}
		in the sense that for any $(\beta,\bm\alpha)\in\mc A_{N,M+1}$
		\begin{equation*}
			\begin{split}
				&\ps{\psi_{-1,\beta}(t)\V}=2(q-\kappa)(\mu_{L,2}-\mu_{R,2})\ps{V_{\beta+\gamma e_2}(t)\V}.
			\end{split}
		\end{equation*} 
	\end{theorem}
	\begin{proof}
		To start with note that $\frac{3w(\beta)}{2\Delta_\beta}=\left(q-2\omega_{2}(\beta)\right)=(q-\frac{2\kappa}3)$ for such a $\beta=\kappa\omega_1$. Besides explicit computations show that for such a $\beta$, we have the equality $\frac{3w(\beta)}{2\Delta_\beta}\L_{-1}^\beta=\Wb_{-1}^\beta$ between linear maps over $\C^2$ (see \textit{e.g.}~\cite[Subsection 5.1]{Toda_OPEWV}), so that for any positive $\delta,\eps,\rho$
		\begin{align*}
			\frac{3w(\beta)}{2\Delta_\beta}\ps{\L_{-1}V_\beta(t)\V}_{\delta,\eps,\rho}=\ps{\Wb_{-1}V_\beta(t)\V}_{\delta,\eps,\rho}.
		\end{align*}
		As a consequence and by definition of the descendants from Equation~\eqref{eq:def_descLn} 
		\[
		\ps{\psi_{-1,\beta}(t)\V}=\lim\limits_{\delta,\eps,\rho\to0}\frac{3\omega(\beta)}{2\Delta_\beta}\tilde{\mathfrak L}_{-1}(\delta,\eps,\rho;\bm\alpha)-\tilde{\mathfrak W}_{-1}(\delta,\eps,\rho;\bm\alpha).
		\]
		Now, as explained in \cite{CH_sym1} and reminded in Subsection~\ref{subsec:def_descendants}, from the explicit expressions~\eqref{eq:rem_Ln} and~\eqref{eq:rem_Wn} we infer that
        \begin{equation*}
		\begin{split}
            &\tilde{\mathfrak{L}}_{-1}(\delta,\eps,\rho;\bm\alpha)=\sum_{i=1}^2\tilde{\mathfrak{L}}_{-1}^i(\delta,\eps,\rho;\bm\alpha)\qt{and}\tilde{\mathfrak{W}}_{-1}(\delta,\eps,\rho;\bm\alpha)=\sum_{i=1}^2\tilde{\mathfrak{W}}_{-1}^i(\delta,\eps,\rho;\bm\alpha),\text{ where}\\
			&\tilde{\mathfrak{L}}^i_{-1}(\delta,\eps,\rho;\bm\alpha)=\left(\mu_{L,i}\ps{V_{\gamma e_i}(t-\eps)V_\beta(t)\V}_{\delta,\eps}-\mu_{R,i}\ps{V_{\gamma e_i}(t+\eps)V_\beta(t)\V}_{\delta,\eps}\right)\qt{while}\\
            &\tilde{\mathfrak{W}}_{-1}^i(\delta,\eps,\rho;\bm\alpha)\coloneqq -h_2(e_i)\left(q-2\omega_{\hat i}(\beta)\right)\tilde{\mathfrak{L}}^i_{-1}(\delta,\eps,\rho;\bm\alpha).
		\end{split}
	\end{equation*}
        Since we have $\frac{3w(\beta)}{2\Delta_\beta}=q-\frac{2\kappa}{3}$ for $\beta=\kappa\omega_1$ we see that the remainder terms are such that:
		\begin{equation*}
			\begin{split}
				\tilde{\mathfrak W}_{-1}^1(\delta,\eps,\rho;\bm\alpha)&=\frac{3w(\beta)}{2\Delta_\beta}\tilde{\mathfrak L}_{-1}^1(\delta,\eps,\rho;\bm\alpha)\quad\text{while}\\
				\tilde{\mathfrak W}_{-1}^2(\delta,\eps,\rho;\bm\alpha)&=-\left(q-2\omega_{1}(\beta)\right)\tilde{\mathfrak L}_{-1}^2(\delta,\eps,\rho;\bm\alpha).
			\end{split}
		\end{equation*}
		where the exponents indicate whether we pick $i=1$ or $i=2$. Now since $\ps{\beta,\gamma e_2}=0$ and thanks to the probabilistic representation of the correlation functions the above remainder term is seen to be given by
		\[
		\tilde{\mathfrak L}_{-1}^2(\delta,\eps,\rho;\bm\alpha)=(\mu_{L,2}-\mu_{R,2})\ps{V_{\beta+\gamma e_2}(t)\V}+o(1).
		\]
		As a consequence we can conclude that
		\begin{align*}
			\ps{\psi_{-1,\beta}(t)\V}&=\left(\frac{3\omega(\beta)}{2\Delta_\beta}+q-\frac{4\kappa}{3}\right)(\mu_{L,2}-\mu_{R,2})\ps{V_{\beta+\gamma e_2}(t)\V}\\
			&=2\left(q-\kappa\right)(\mu_{L,2}-\mu_{R,2})\ps{V_{\beta+\gamma e_2}(t)\V}.
		\end{align*}	
		The result follows from the definition of the descendants~\eqref{eq:def_descLn}.
	\end{proof}
	
	\subsection{Descendants and singular vectors at the level two}
    In this subsection we derive the explicit expression of the singular vector at the level two $\psi_{-2,\chi}$ introduced in Equation~\eqref{eq:sings_intro}. And to begin with we first need to define the descendant field $\Lc_{-(1,1)}V_\beta$ that appears in the expression of this singular vector. Based on this definition together with the ones already provided for $\Lc_{-2}V_\beta$ and $\Wb_{-2}V_\beta$ we will see that when the weight $\beta$ is given by $\beta=-\chi\omega_1$, $\chi\in\{\gamma,\frac2\gamma\}$, the associated fully-degenerate field $V_\beta$ allows to define one additional singular vector at the level two. 
	
	\subsubsection{A first take on the descendant at the second level}
	As a first step we define here the descendant field $\L_{-1,1}V_\beta$, where recall that the latter is defined by considering expressions of the form
	\begin{equation*}
		\begin{split}
			&\L_{-1,1}V_\beta\coloneqq\left(\L_{-1,1}^\beta(\partial^2\Phi)+\L_{-1,1}^\beta(\partial\Phi,\partial\Phi)\right)V_\beta,\\
			&\L_{-1,1}^\beta(u)\coloneqq\ps{\beta,u}\qt{and}\L_{-1,1}^\beta(u,v)\coloneqq\ps{\beta,u}\ps{\beta,v}.
		\end{split}
	\end{equation*}
    The proof is rather heavy and relies on the same method as in~\cite[Section 3]{CH_sym1}: as a consequence we will only provide a few details on the computations conducted and refer to~\cite{CH_sym1} for more explanations on the method developed here.
	\begin{lemma}\label{lemma:desc11_L}
		Define the remainder term
		\begin{equation*}
			\begin{split}
				&\mathfrak L_{-(1,1);\delta,\eps,\rho}(\bm\alpha)\coloneqq \Is\left[\partial_x\ps{V_{\gamma e_i}(x)\wick{\c\L^{2\beta+\gamma e_i}_{-1}\c V_\beta(t)}\V}_{\delta,\eps}\right]\\
				&-\Is\left[\left(\partial_x+\partial_{\bar x}\mathds1_{x\in\Heps}\right)\left(\frac{\ps{\beta,\gamma e_i}}{2(t-x)}\Psi_i(x)\right)\right]\\
				&+\Is\times\Is\left[\partial_x\left(\frac{\ps{\beta,\gamma e_j}}{2(t-y)}\Psi_{i,j}(x,y)\right)\right]
			\end{split}
		\end{equation*}
		where we have introduced the notation
		\begin{equation*}
			\begin{split}
				&\ps{V_{\gamma e_i}(x)\wick{\c\L^{2\beta+\gamma e_i}_{-1}\c V_\beta(t)}\V}_{\delta,\eps}\coloneqq \sum_{k=1}^{2N+M}\frac{\ps{2\beta+\gamma e_i,\alpha_k}}{2(z_k-t)}\Psi_{i}(x)\\
				&-\Ir\left[\frac{\ps{2\beta+\gamma e_i,\gamma e_j}}{2(y-t)}\Psi_{i,j}(x,y)\right]+\Is\left[\sum_{k=1}^{2N+M}\frac{\ps{\alpha_k,\gamma e_j}}{2(y-z_k)}\Psi_{i,j}(x,y)\right]\\
				&+\Is\times\Ir\left[\frac{\ps{\gamma e_j,\gamma e_f}}{2(y-z)}\Psi_{i,j,f}(x,y,z)\right].
			\end{split}
		\end{equation*}
		Then the following quantity is $(P)$-class:
		\begin{equation*}
			\ps{\L_{-1,1}V_\beta(t)\V}_{\delta,\eps,\rho}-\mathfrak L_{-(1,1);\delta,\eps,\rho}(\bm\alpha).
		\end{equation*}
	\end{lemma}
	\begin{proof}
		To start with we use the defining expression~\eqref{eq:desc_reg} together with Lemma~\ref{lemma:GaussianIPP} to get
        \begin{align*}
            &\ps{\L_{-1,1}V_\beta(t)\V}_{\delta,\eps,\rho}=\sum_{k=1}^{2N+M}\frac{\ps{\beta,\alpha_k}}{2(t-z-k)^2)}+\sum_{k,l=1}^{2N+M}\frac{\ps{\beta,\alpha_k}\ps{\beta,\alpha_l}}{4(t-z-k)(t-z_l))}\\
            &-\It\left[\left(\frac{\ps{\beta,\gamma e_i}(2+\ps{\beta,\gamma e_i})}{4(t-x)^2}+2\frac{\ps{\beta,\gamma e_i}\ps{\beta,\gamma e_i}}{4(t-x)(t-\bar x)}\mathds1_{x\in\Heps}+2\sum_{k=1}^{2N+M}\frac{\ps{\beta,\alpha_k}\ps{\beta,\gamma e_i}}{4(t-z_k)(t-x)}\right)\Psi_{i}(x)\right]\\
			&+\It^2\left[\frac{\ps{\beta,\gamma e_i}\ps{\beta,\gamma e_j}}{4(t-x)(t-y)}\Psi_{i,j}(x,y)\right].
        \end{align*}
        After some algebraic manipulations on this expression we see that the only possibly singular terms in the expansion of $\ps{\L_{-1,1}V_\beta(t)\V}_{\delta,\eps,\rho}$ are given by
		\begin{align*}
			\mathfrak I\coloneqq&-\Is\left[\left(\frac{\ps{\beta,\gamma e_i}(2+\ps{\beta,\gamma e_i})}{4(t-x)^2}+2\frac{\ps{\beta,\gamma e_i}\ps{\beta,\gamma e_i}}{4(t-x)(t-\bar x)}\mathds1_{x\in\Heps}+2\sum_{k=1}^{2N+M}\frac{\ps{\beta,\alpha_k}\ps{\beta,\gamma e_i}}{4(t-z_k)(t-x)}\right)\Psi_{i}(x)\right]\\
			&+\Is\times\Ir\left[2\frac{\ps{\beta,\gamma e_i}\ps{\beta,\gamma e_j}}{4(t-x)(t-y)}\Psi_{i,j}(x,y)\right]+\Is^2\left[\frac{\ps{\beta,\gamma e_i}\ps{\beta,\gamma e_j}}{4(t-x)(t-y)}\Psi_{i,j}(x,y)\right].
		\end{align*}
		The one-fold integral can be rewritten using 
		\begin{align*}
			&\left(\partial_x+\partial_{\bar x}\mathds1_{x\in\Heps}\right)\left[\left(\frac{\ps{\beta,\gamma e_i}}{2(t-x)}+\sum_{k=1}^{2N+M}\frac{\ps{2\beta+\gamma e_i,\alpha_k}}{2(t-z_k)}\right)\Psi_{i}(x)\right]=\qt{$(P)$-class terms}\\
			&+\left(\frac{\ps{\beta,\gamma e_i}(2+\ps{\beta,\gamma e_i})}{4(t-x)^2}+\frac{\ps{\beta,\gamma e_i}\ps{\beta,\gamma e_i}}{2(t-x)(t-\bar x)}\mathds1_{x\in\Heps}\right)\Psi_{i}(x)\\
			&+\sum_{k=1}^{2N+M}\left(\frac{\ps{\beta,\alpha_k}\ps{\beta,\gamma e_i}}{2(t-z_k)(t-x)}+\frac{\ps{\beta,\alpha_k}\ps{\beta,\gamma e_i}}{2(t-z_k)(t-\bar x)}\mathds 1_{x\in\Heps}\right)\Psi_{i}(x)\\
			&-\It\left[\left(\frac{\ps{\beta,\gamma e_i}}{2(t-x)}+\sum_{k=1}^{2N+M}\frac{\ps{2\beta+\gamma e_i,\alpha_k}}{2(t-z_k)}\right)\left(\frac{\ps{\gamma e_,\gamma e_j}}{2(y-x)}+\frac{\ps{\gamma e_,\gamma e_j}}{2(y-\bar x)}\mathds1_{x\in\Heps}\right)\Psi_{i,j}(x,y)\right].
		\end{align*}
		Since the integral $\Is\times\Ir\left[\frac1{x-y}\Psi_{i,j}(x,y)\right]$ is $(P)$-class in virtue of Lemma~\ref{lemma:fusion_integrability} this entails
		\begin{align*}
			\mathfrak I&=\qt{$(P)$-class terms}\\
			&-\Is\left[\left(\partial_x+\partial_{\bar x}\mathds1_{x\in\Heps}\right)\left(\frac{\ps{\beta,\gamma e_i}}{2(t-x)}\Psi_i(x)\right)\right]-\Is\left[\partial_x\left(\sum_{k=1}^{2N+M}\frac{\ps{2\beta+\gamma e_i,\alpha_k}}{2(t-z_k)}\Psi_{i}(x)\right)\right]\\
			&+\Is\times\Ir\left[\left(\frac{\ps{\beta,\gamma e_i}\ps{\beta,\gamma e_j}}{2(t-x)(t-y)}+\frac{\ps{\beta,\gamma e_i}\ps{\gamma e_,\gamma e_j}}{4(t-x)(x-y)}\right)\Psi_{i,j}(x,y)\right]\\
			&+\Is\times\Is\left[\left(\frac{\ps{\beta,\gamma e_i}}{2(t-x)}\left(\frac{\ps{\beta,\gamma e_j}}{2(t-y)}+\frac{\ps{\gamma e_i,\gamma e_j}}{2(x-y)}\right)+\sum_{k=1}^{2N+M}\frac{\ps{2\beta+\gamma e_i,\alpha_k}\ps{\gamma e_,\gamma e_j}}{4(t-z_k)(y-x)}\right)\Psi_{i,j}(x,y)\right].
		\end{align*}
		
		Let us now turn to the two-folds integrals: over $\Is\times\Ir$ singular terms are given by
		\begin{align*}
			&\Is\times\Ir\left[\left(\frac{\ps{\beta,\gamma e_i}\ps{\beta,\gamma e_j}}{2(t-x)(t-y)}+\frac{\ps{\beta,\gamma e_i}\ps{\gamma e_,\gamma e_j}}{4(t-x)(x-y)}\right)\Psi_{i,j}(x,y)\right]\\
			&= \qt{$(P)$-class terms}+\Is\times\Ir\left[\left(\frac{\ps{\beta,\gamma e_i}\ps{2\beta+\gamma e_i,\gamma e_j}}{4(t-x)(t-y)}\right)\Psi_{i,j}(x,y)\right].
		\end{align*}
		To take care of the singularity at $x=t$ we use the same reasoning as in the proof of~\cite[Lemma 3.1]{CH_sym1}: this shows that 
		\begin{align*}
			&\Is\times\Ir\left[\left(\frac{\ps{\beta,\gamma e_i}\ps{\beta,\gamma e_j}}{2(t-x)(t-y)}+\frac{\ps{\beta,\gamma e_i}\ps{\gamma e_,\gamma e_j}}{4(t-x)(x-y)}\right)\Psi_{i,j}(x,y)\right]\\
			&= \qt{$(P)$-class terms}+\Is\times\Ir\left[\partial_x\left(\frac{\ps{2\beta+\gamma e_i,\gamma e_j}}{2(t-y)}\right)\Psi_{i,j}(x,y)\right].
		\end{align*}
		This corresponds exactly to the integral over $\Ir$ that appears in the expansion of \\$\ps{V_{\gamma e_i}(x)\wick{\c\L^{2\beta+\gamma e_i}_{-1}\c V_\beta(t)}\V}_{\delta,\eps}$. 
		
		As for the integral over $\Is\times\Is$ the corresponding term is
		\begin{align*}
			&\Is\times\Is\left[\frac{\ps{\beta,\gamma e_i}}{2(t-x)}\left(\frac{\ps{\beta,\gamma e_j}}{2(t-y)}+\frac{\ps{\gamma e_i,\gamma e_j}}{2(x-y)}\right)\Psi_{i,j}(x,y)\right]\\
			&+\Is\times\Is\left[\sum_{k=1}^{2N+M}\frac{\ps{2\beta+\gamma e_i,\alpha_k}}{2(t-z_k)}\frac{\ps{\gamma e_,\gamma e_j}}{2(y-x)}\Psi_{i,j}(x,y)\right].
		\end{align*}
		The integral on the second line is actually $(P)$-class: to see why for $i=j$ we use the symmetry in $x,y$ to see that the term vanishes while for $i\neq j$ the singularity at $x=y$ is integrable thanks to Lemma~\ref{lemma:fusion} (the exponent is positive). 
		As for the first integral it can be rewritten as
		\begin{align*}
			&\Is\times\Is\left[\left(\partial_y+\sum_{k=1}^{2N+M}\frac{\ps{\gamma e_j,\alpha_k}}{2(y-z_k)}\right)\frac{\ps{\beta,\gamma e_i}}{2(t-x)}\Psi_{i,j}(x,y)\right]\\
			&-\Is\times\Is\times\It\left[\frac{\ps{\beta,\gamma e_i}\ps{\gamma e_j,\gamma e_f}}{4(t-x)(y-z)}\Psi_{i,j,f}(x,y,z)\right].
		\end{align*}
		The derivative $\partial_y$ appears in the expansion of the remainder term. Therefore it only remains to treat 
		\begin{align*}
			&\Is\times\Is\left[\sum_{k=1}^{2N+M}\frac{\ps{\gamma e_j,\alpha_k}\ps{\beta,\gamma e_i}}{4(y-z_k)(t-x)}\Psi_{i,j}(x,y)\right]-\Is\times\Is\times\It\left[\frac{\ps{\beta,\gamma e_i}\ps{\gamma e_j,\gamma e_l}}{4(t-x)(y-z)}\Psi_{i,j,f}(x,y,z)\right].
		\end{align*}
		In the same fashion as in the proof of Theorem~\ref{thm:deg1} the first integral will give rise to the the term corresponding to the integral over $\Is$  that appears in the expansion of $\ps{V_{\gamma e_i}(x)\wick{\c\L^{2\beta+\gamma e_i}_{-1}\c V_\beta(t)}\V}_{\delta,\eps}$. The second integral vanishes over $\Is^3$ by symmetry in $y,z$, while the integral over $\Is^2\times\Ir$ yields the term $\Is\times\Ir$ in $\ps{V_{\gamma e_i}(x)\wick{\c\L^{2\beta+\gamma e_i}_{-1}\c V_\beta(t)}\V}_{\delta,\eps}$.
		
		Recollecting term we see that as desired $\mathfrak I=\text{ $(P)$-class terms }+\mathfrak L_{-(1,1);\delta,\eps,\rho}(\bm\alpha)$.
	\end{proof}

    \subsubsection{Definition of the descendants at the level two}
	Based on the previous statement we can now define the descendant $\L_{-1,1}V_\beta$. This is done by substracting the remainder term $\tilde{\mathfrak{L}}_{-(1,1);\delta,\eps,\rho}(\bm\alpha)$ to the regularized quantity, where this quantity is obtained from $\mathfrak{L}_{-(1,1);\delta,\eps,\rho}(\bm\alpha)$ by application of Stokes' formula:
    \begin{eqs}
        &\tilde{\mathfrak{L}}_{-(1,1);\delta,\eps,\rho}(\bm\alpha)=-\frac{\ps{\beta,\gamma e_i}}{2\eps}\left(\mu_{L,i}\Psi_i(t-\eps)+\mu_{R,i}\Psi_i(t+\eps)\right)\\
		&+\mu_{L,i}\ps{V_{\gamma e_i}(t-\eps)\wick{\c\L^{2\beta+\gamma e_i}_{-1}\c V_\beta(t)}\V}_{\delta,\eps}-\mu_{R,i}\ps{V_{\gamma e_i}(t+\eps)\wick{\c\L^{2\beta+\gamma e_i}_{-1}\c V_\beta(t)}\V}_{\delta,\eps}\\
		&+\Is\left[\frac{\ps{\beta,\gamma e_j}}{2(t-y)}\left(\mu_{L,i}\Psi_{i,j}(t-\eps,y)-\mu_{R,i}\Psi_{i,j}(t+\eps,y)\right)\right].
    \end{eqs}
	\begin{defi}
		For any $t\in\R$ and $(\beta,\bm\alpha)\in\mc A_{N,M+1}$,  the descendant field $\L_{-1,1}V_\beta$ is defined by the following limit, well-defined and analytic in a complex neighborhood of $\mc A_{N,M+1}$:
		\begin{equation*}
			\begin{split}
				\ps{\L_{-1,1}V_\beta(t)\V}\coloneqq\lim\limits_{\delta,\eps,\rho\to0}\ps{\L_{-1,1}V_\beta(t)\V}_{\delta,\eps,\rho}-\tilde{\mathfrak{L}}_{-(1,1);\delta,\eps,\rho}(\bm\alpha).
			\end{split}
		\end{equation*}
	\end{defi}
	\begin{proof}
		We need to show that the difference $\mathfrak L_{-(1,1);\delta,\eps,\rho}(\bm\alpha)-\tilde{\mathfrak L}_{-(1,1);\delta,\eps,\rho}(\bm\alpha)$ is $(P)$-class. In this expression the only term that may not be $(P)$-class is given by
		\begin{align*}
			&\oint_{\Heps\cap\partial B(t,r)}\Is\left[\frac{\ps{\beta,\gamma e_j}}{2(t-y)}\Psi_{i,j}(\xi,y)\right]\frac{id\bar\xi-id\xi}{2}-\oint_{\Heps\cap\partial B(t,r)}\Is\left[\partial_y\Psi_{i,j}(\xi,y)\right]\frac{id\bar\xi-id\xi}{2}\cdot
		\end{align*}
		For this we develop the derivative in $y$ to cancel the singularity at $y=t$: this yields only $(P)$-class terms.
	\end{proof}
	
	\begin{proposition}\label{prop:L11_der}
		For $t\in\R$ and ($\beta,\bm\alpha)\in\mc A_{N,M}$ the following equality holds in the sense of weak derivatives:
		\begin{equation*}
			\begin{split}
				&\ps{\L_{-1,1}V_\beta(t)\V}=\partial_{t}^2\ps{V_\beta(t)\V}.
			\end{split}
		\end{equation*}
	\end{proposition}
	\begin{proof}
		The proof follows from the very same argument as that of~\cite[Proposition 3.4]{CH_sym1}. Namely for positive $\delta,\eps,\rho$ we have the equality for any smooth and compactly supported test function $f$:
        \begin{align*}
			\int_{\R}f(t)\ps{\L_{-1,1}V_\beta(t)\V}_{\delta,\eps,\rho}dt=\int_{\R}\partial_t^2 f(t)\ps{V_\beta(t)\V}_{\delta,\eps,\rho}dt.
		\end{align*}
        Now it is readily seen from the fusion asymptotics of Lemma~\ref{lemma:fusion} (see also the reasoning developed in~\cite[Subsection 4.2]{Cer_HEM}) that for $\ps{\beta,e_i}$ negative enough the remainder term $\tilde{\mathfrak{L}}_{-(1,1);\delta,\eps,\rho}(\bm\alpha)$ that appears in the definition of the descendant field vanishes in the $\rho,\eps,\delta\to0$ limit. As a consequence for $\ps{\beta,e_i}$ negative enough we have the equality 
         \begin{align*}
			\int_{\R}f(t)\ps{\L_{-1,1}V_\beta(t)\V}dt=\int_{\R}\partial_t^2 f(t)\ps{V_\beta(t)\V}dt.
		\end{align*}
        This equality can be extended by analycity of the correlation functions in a complex neighborhood of $\mc A_{N,M}$, thus completing the proof. 
	\end{proof}

	\subsubsection{Degenerate fields at the level two}
	We have defined in the previous subsection semi-degenerate fields, which are Vertex Operators $V_\beta$ where the weight is chosen to be of the form $\beta=\kappa\omega_i$ for $\kappa\in\R$ and $i\in\{1,2\}$. We have seen that such a semi-degenerate field allows to define a singular vector at level one $\Psi_{-1,\beta}$. We now explain how a further specialization of the weight $\beta$ leads to the introduction of singular vectors at a higher level.
	
	To this end let us now assume that $\beta=-\chi\omega_1$ with $\chi\in\{\gamma,\frac2\gamma\}$.  The singular vector at the level two is formally defined by
	\begin{equation*}
		\psi_{-2,\chi}\coloneqq \left(\Wb_{-2}+\frac4\chi\L_{-1,1}+\frac{4\chi}3\L_{-2}\right)V_{\beta}
	\end{equation*}
	and makes sense when inserted within correlation functions as follows.
	\begin{defi}
		Assume that $\beta=-\chi\omega_1$ with $\chi\in\{\gamma,\frac2\gamma\}$ and take $t\in\R$ as well as $(\beta,\bm\alpha)\in\mc A_{N,M+1}$. The singular vector $\Psi_{-2,\chi}$ is defined within half-plane correlation functions via
		\begin{equation*}
			\begin{split}
				&\ps{\psi_{-2,\chi}(t)\V}\coloneqq\ps{\Wb_{-2}V_\beta(t)\V}+\frac4\chi\ps{\L_{-1,1}V_\beta(t)\V}+\frac{4\chi}3\ps{\L_{-2}V_\beta(t)\V}.
			\end{split}
		\end{equation*}
	\end{defi}
	
	In the same fashion as for the first level, the singular vector $\psi_{-2,\beta}$ is generically not a null vector but rather gives rise to \textit{equations of motion}:
	\begin{theorem}\label{thm:deg2}
		Assume that $\beta=-\chi\omega_1$ with $\chi\in\{\gamma,\frac2\gamma\}$, and take $a(\chi)$, $b(\chi)$ so that
		\begin{equation}
			3\beta+\gamma e_2=a(\chi)\L_{-1}^{\beta+\gamma e_2}+b(\chi)\Wb_{-1}^{\beta+\gamma e_2}.
		\end{equation}
		If we set $\Db_{-1,\chi}=a(\chi)\L_{-1}^{\beta+\gamma e_2}+b(\chi)\Wb_{-1}^{\gamma+\beta e_2}$, then
		\begin{equation*}
			\psi_{-2,\frac2\gamma}=2\gamma (\mu_{R,2}-\mu_{L,2}) \Db_{-1,\frac2\gamma} V_{\beta+\gamma e_2}+2\left(\frac2\gamma-\gamma\right)(\mu_{L,1}+\mu_{R,1})V_{\beta+\gamma e_1}
		\end{equation*}
		in the sense that as soon as $(\beta,\bm\alpha)\in\mc A_{N,M+1}$
		\begin{equation}\label{eq:psi22g}
			\begin{split}
				\ps{\psi_{-2,\frac2\gamma}(t)\V}=&2\gamma(\mu_{R,2}-\mu_{L,2})\Big(a\ps{\L_{-1}V_{\beta+\gamma e_2}(t)\V}+ b\ps{\Wb_{-1}V_{\beta+\gamma e_2}(t)\V}\Big)\\
				&+2\left(\frac2\gamma-\gamma\right)(\mu_{L,1}+\mu_{R,1})\ps{V_{\beta+\gamma e_1}(t)\V}.
			\end{split}
		\end{equation} 
		When $\chi=\gamma$ we have $\psi_{-2,\gamma}=\frac4\gamma(\mu_{R,2}-\mu_{L,2})\Db_{-1,\gamma}V_{\beta+\gamma e_2}$ for $\gamma>1$ while for $\gamma<1$:
		\begin{equation}\label{eq:psi2g}
			\begin{split}
				&\psi_{-2,\gamma}=\frac4\gamma(\mu_{R,2}-\mu_{L,2})\Db_{-1,\gamma}V_{\beta+\gamma e_2}+2\gamma c_\gamma(\bm\mu)V_{\beta+2\gamma e_1}
			\end{split}
		\end{equation}
		meaning the same as for $\psi_{-2,\frac2\gamma}$. Here the constant $c_\gamma(\bm\mu)$ is given by
		\begin{equation*}
			\left(\mu_{L,1}^2+\mu_{R,1}^2-2\mu_{L,1}\mu_{R,1}\cos\left(\pi\frac{\gamma^2}{2}\right)-\mu_{B,1}\sin\left(\pi\frac{\gamma^2}{2}\right)\right)\frac{\Gamma\left(\frac{\gamma^2}{2}\right)\Gamma\left(1-\gamma^2\right)}{\Gamma\left(1-\frac{\gamma^2}{2}\right)}\cdot
		\end{equation*}
	\end{theorem}
	\begin{remark}
		We can rewrite the above as $\psi_{-2,\chi}=-\frac{1}{1+\chi^2}\Db_{-1}\psi_{-1,\chi}+\psi^0_{-2,\chi}$ where we have a descendant at level $1$ of the singular vector $\psi_{-1,\chi}$ and a new primary field $\psi_{-2,\chi}^0$.
	\end{remark}
	\begin{proof}
		Like before we know from~\cite[Proposition 5.3]{Toda_OPEWV} (more generally see~\cite[Subsection 5.2]{Toda_OPEWV}) that we have an equality at the level of multilinear maps over $\C$ \[
		\Wb_{-2}^\beta+\frac4\chi\L_{-1,1}^\beta+\frac{4\chi}3\L_{-2}^\beta=0.
		\]
		This entails the equality for the regularized correlation functions
		\begin{equation*}
			\ps{\Wb_{-2}V_{\beta}(t)\V}_{\delta,\eps,\rho}+\frac4\chi\ps{\L_{-1,1}V_{\beta}(t)\V}_{\delta,\eps,\rho}+\frac{4\chi}3\ps{\L_{-2}V_{\beta}(t)\V}_{\delta,\eps,\rho}=0
		\end{equation*}
		so that we only need to check that
		\begin{equation*}
			\lim\limits_{\delta,\eps,\rho\to0}\tilde{\mathfrak W}_{-2}(\delta,\eps,\rho;\bm\alpha)+\frac4\chi\tilde{\mathfrak L}_{-1,1}(\delta,\eps,\rho;\bm\alpha)+\frac{4\chi }3\tilde{\mathfrak L}_{-2}(\delta,\eps,\rho;\bm\alpha)=-\ps{\psi_{-2,\chi}(t)\V}.
		\end{equation*}
		
		To start with according to Equations~\eqref{eq:rem_Ln},~\eqref{eq:rem_Wn} and Lemma~\ref{lemma:desc11_L}, the left-hand side admits the representation, up to $(P)$-class terms that we discard to simplify the discussion,
		\begin{eqs}\label{eq:mfI2}
			\mathfrak I\coloneqq&\Is\left[\partial_x\left(F(x)\Psi_{i}(x)\right)\right]+\Is\left[\partial_x\left(\frac{2h_2(e_i)\omega_{\hat i}(\beta)+\frac{2\ps{\beta,\gamma e_i}}\chi+\frac{4\chi}{3}}{x-t}\Psi_{i}(x)\right)\right]\\
			+&\frac4\chi\sum_{i=1}^2\mu_{B,i}\int_{\Heps^1}\partial_{\bar x}\left(\frac{\ps{\beta,\gamma e_i}}{2(x-t)}\Psi_{i}(x)\right)+\partial_{x}\left(\frac{\ps{\beta,\gamma e_i}}{2(\bar x-t)}\Psi_{i}(x)\right)d^2x\\
			+&\Is\times\Is\left[\partial_x\left(\left(\delta_{i\neq j}\frac{2\gamma h_2(e_i)}{y-x}+\frac{2\ps{\beta,\gamma e_j}}{\chi(t-y)}\right)\Psi_{i,j}(x,y)\right)\right]
		\end{eqs}
		where $F(x)$ is regular at $x=t$ and explicitly given by
		\begin{equation}\label{eq:F}
			\begin{split}
				F(x)\Psi_i(x)=&\frac4\chi \ps{V_{\gamma e_i}(x)\wick{\c\L^{2\beta+\gamma e_i}_{-1}\c V_\beta(t)}\V}_{\delta,\eps}\\
				&+\sum_{k=1}^{2N+M}\frac{2h_2(e_i)\omega_{\hat i}(\alpha_k)}{x-z_k}\Psi_i(x)-\Ir\left[\delta_{i\neq j}\frac{2\gamma h_2(e_i)}{x-y}\Psi_{i,j}(x,y)\right].
			\end{split}
		\end{equation} 
		Let us first treat the integrals where $F$ does not appear.
		
		For the remaining one-fold integral $\Is$ we rely on explicit computations based on the expression of $\beta=-\chi\omega_1$ to get that 
		\[
		2h_2(e_i)\omega_{\hat i}(\beta)+\frac{2\ps{\beta,\gamma e_i}}\chi+\frac{4\chi}{3}=\delta_{i,1}2(\chi-\gamma).
		\]
		In particular if $\chi=\gamma$ this term vanishes and therefore does not contribute to the limit. For $\chi=\frac2\gamma$ we rely on Lemma~\ref{lemma:limtx} to see that the corresponding remainder term is given by
		\begin{align*}
			&-2\left(\frac2\gamma-\gamma\right)\left(\mu_{R,1}+\mu_{L,1}\right)\ps{V_{\gamma e_1-\frac2\gamma \omega_1}(t)\V}.
		\end{align*}
		
		Next we consider the additional one-fold integral that ranges over $\Heps$. Then we see that for $i=2$ the numerator vanishes, while for $i=1$ we are in the setting of Lemma~\ref{lemma:limty} below where similar quantities are considered, and shown to scale like
		\begin{align*}
			2\gamma \mu_{B,1}\sin\left(\pi\frac{\gamma^2}{2}\right)\frac{\Gamma\left(\frac{\gamma^2}{2}\right)\Gamma\left(1-\gamma^2\right)}{\Gamma\left(1-\frac{\gamma^2}{2}\right)}
		\end{align*}
		for $\chi=\gamma$ and $\gamma<1$, while these are $o(1)$ if $\chi=\frac2\gamma$ or $\chi=\gamma$ and $\gamma>1$.
		
		Let us now turn to the two-folds integrals. By explicit computations these are equal to
		\begin{align*}
			\Is\times\Is\left[\partial_x\left(\left(\delta_{i,j=1}\frac{2\gamma}{y-t}-\delta_{\substack{i=1\\j=2}}\frac{2\gamma }{y-x}+\delta_{\substack{i=2\\j=1}}\left(\frac{2\gamma}{y-t}+\frac{2\gamma}{y-x}\right)\right)\Psi_{i,j}(x,y)\right)\right].
		\end{align*}
		The first term is considered in Lemma~\ref{lemma:limty}, and for $\chi=\gamma$ and $\gamma<1$ we have
		\begin{align*}
			&\lim\limits_{\delta,\eps,\rho\to0}\Is\times\Is\left[\partial_x\left(\frac{\delta_{i,j=1}}{y-t}\Psi_{i,j}(x,y)\right)\right]\\
			&=-\left(\mu_{L,1}^2+\mu_{R,1}^2-2\mu_{L,1}\mu_{R,1}\cos\left(\pi\frac{\gamma^2}{2}\right)\right)\frac{\Gamma\left(\frac{\gamma^2}{2}\right)\Gamma\left(1-\gamma^2\right)}{\Gamma\left(1-\frac{\gamma^2}{2}\right)}
		\end{align*}
		while this limit is zero for $\chi=\frac2\gamma$ and $\chi=\gamma$ with $\gamma>1$. When we consider $(i,j)=(1,2)$ we use Lemma~\ref{lemma:fusion} to infer that the integral $\Is\left[\frac1{y-t}\Psi_{1,2}(t,y)\right]$ is absolutely convergent. As a consequence the remainder term coming from the $\Is\times\Is$ integral with $(i,j)=(1,2)$  this term converges to $0$ in the limit (since $\ps{\beta,\gamma e_1}<0$). This is however not the case when $(i,j)=(2,1)$: like in the proof of Theorem~\ref{thm:deg1} we see that this term converges towards
		\[
		(\mu_{L,2}-\mu_{R,2})\Is\left[\frac{4\gamma}{y-t}\ps{V_{\gamma e_1}(y)V_{\beta+\gamma e_2}(t)\V}\right].
		\]
		Recollecting terms we see that
		\begin{align*}
			\mathfrak I= &\text{ $(P)$-class terms }+\Is\left[\partial_x\left(F(x)\Psi_i(x)\right)\right]+(\mu_{L,2}-\mu_{R,2})\Is\left[\frac{4\gamma}{y-t}\ps{V_{\gamma e_1}(y)V_{\beta+\gamma e_2}(t)\V}\right]\\
			&-2\left(\frac2\gamma-\gamma\right)\left(\mu_{R,1}+\mu_{L,1}\right)\ps{V_{\gamma e_1-\frac2\gamma \omega_1}(t)\V}\mathds 1_{\chi=\frac2\gamma}-2\gamma c_\gamma(\bm\mu)\ps{V_{\gamma e_1-\frac2\gamma \omega_1}(t)\V}\mathds 1_{\chi=\gamma}.
		\end{align*}
		
		Let us now consider the term $\Is\left[\partial_x\left(F(x)\Psi_{i}(x)\right)\right]$, where $F$ has been described in Equation~\eqref{eq:F}. Using the very same reasoning as the one conducted in the proof of Theorem~\ref{thm:deg1}, we see that since $\ps{\beta,\gamma e_1}<0$ while $\ps{\beta,\gamma e_2}=0$ this term converges towards $(\mu_{L,2}-\mu_{R,2})H$,
		\begin{align*}
			H\coloneqq&\frac4\chi \left(\sum_{k=1}^{2N+M}\frac{\ps{2\beta+\gamma e_2,\alpha_k}}{2(z_k-t)}\ps{V_{\beta+\gamma e_2}(t)\V}-\Ir\left[\frac{\ps{2\beta+\gamma e_2,\gamma e_j}}{2(y-t)}\ps{V_{\gamma e_j}(y)V_{\beta+\gamma e_2}(t)\V}\right]\right)\\
			&+\frac4\chi\Is\left[\sum_{k=1}^{2N+M}\frac{\ps{\alpha_k,\gamma e_j}}{2(y-z_k)}\ps{V_{\gamma e_j}(y)V_{\beta+\gamma e_2}(t)\V}\right]\\
			&+\frac4\chi\Is\times\Ir\left[\frac{\ps{\gamma e_j,\gamma e_f}}{2(z-y)}\ps{V_{\gamma e_j}(y)V_{\gamma e_f}(z)V_{\beta+\gamma e_2}(t)\V}\right]\\
			&+\sum_{k=1}^{2N+M}\frac{2\omega_{1}(\alpha_k)}{t-z_k}\ps{V_{\beta+\gamma e_2}(t)\V}-\Ir\left[\frac{2\gamma }{t-y}\ps{V_{\gamma e_1}(y)V_{\beta+\gamma e_2}(t)\V}\right].
		\end{align*}
		Now we have $2\beta+\gamma e_2-\chi\omega_1=3\beta+\gamma e_2$. As a consequence the latter can be simplified to
		\begin{align*}
			H=& \frac4\chi \left(\sum_{k=1}^{2N+M}\frac{\ps{3\beta+\gamma e_2,\alpha_k}}{2(z_k-t)}\ps{V_{\beta+\gamma e_2}(t)\V}-\Ir\left[\frac{\ps{3\beta+\gamma e_2,\gamma e_j}}{2(y-t)}\ps{V_{\gamma e_j}(y)V_{\beta+\gamma e_2}(t)\V}\right]\right)\\
			&+\frac4\chi\Is\left[\sum_{k=1}^{2N+M}\frac{\ps{\alpha_k,\gamma e_j}}{2(y-z_k)}\ps{V_{\gamma e_j}(y)V_{\beta+\gamma e_2}(t)\V}\right]\\
			&+\frac4\chi\Is\times\Ir\left[\frac{\ps{\gamma e_j,\gamma e_f}}{2(z-y)}\ps{V_{\gamma e_j}(y)V_{\gamma e_f}(z)V_{\beta+\gamma e_2}(t)\V}\right].
		\end{align*}
		
		Gathering all these terms we end up with the equality
		\begin{align*}
			\mathfrak I= &\text{ $(P)$-class terms }+\frac4\chi(\mu_{L,2}-\mu_{R,2})\mathfrak J\\
			&-2\left(\frac2\gamma-\gamma\right)\left(\mu_{R,1}+\mu_{L,1}\right)\ps{V_{\gamma e_1-\frac2\gamma \omega_1}(t)\V}\mathds 1_{\chi=\frac2\gamma}-2\gamma c_\gamma(\bm\mu)\ps{V_{\gamma e_1-\frac2\gamma \omega_1}(t)\V}\mathds 1_{\chi=\gamma}.
		\end{align*}
		where we have set 
		\begin{align*}
			&\mathfrak J\coloneqq\sum_{k=1}^{2N+M}\frac{\ps{3\beta+\gamma e_2,\alpha_k}}{2(z_k-t)}\ps{V_{\beta+\gamma e_2}(t)\V}-\Ir\left[\frac{\ps{3\beta+\gamma e_2,\gamma e_j}}{2(y-t)}\ps{V_{\gamma e_j}(y)V_{\beta+\gamma e_2}(t)\V}\right]\\
			&+\Is\left[\sum_{k=1}^{2N+M}\frac{\ps{\alpha_k,\gamma e_j}}{2(y-z_k)}\ps{V_{\gamma e_j}(y)V_{\beta+\gamma e_2}(t)\V}+\frac{\gamma\chi}{y-t}\ps{V_{\gamma e_1}(y)V_{\beta+\gamma e_2}(t)\V}\right]\\
			&+\Is\times\Ir\left[\frac{\ps{\gamma e_j,\gamma e_f}}{2(z-y)}\ps{V_{\gamma e_j}(y)V_{\gamma e_f}(z)V_{\beta+\gamma e_2}(t)\V}\right].
		\end{align*}
		Now since we have that $\ps{\beta+\gamma e_2,\gamma e_1}<0$ we can use integration by parts (see~\cite[Lemma 3.1]{CH_sym1} for more details) to rewrite that up to $(P)$-class terms we have
        \begin{align*}
            &+\Is\left[\sum_{k=1}^{2N+M}\frac{\ps{\alpha_k,\gamma e_j}}{2(y-z_k)}\ps{V_{\gamma e_j}(y)V_{\beta+\gamma e_2}(t)\V}+\frac{\gamma\chi}{y-t}\ps{V_{\gamma e_1}(y)V_{\beta+\gamma e_2}(t)\V}\right]\\
			&+\Is\times\Ir\left[\frac{\ps{\gamma e_j,\gamma e_f}}{2(z-y)}\ps{V_{\gamma e_j}(y)V_{\gamma e_f}(z)V_{\beta+\gamma e_2}(t)\V}\right]\\
            &=\Is\left[\frac{\ps{\beta+\gamma e_2,\gamma e_i}-2\gamma\chi\delta_{j=1}}{2(t-y)}\ps{V_{\gamma e_j}(y)V_{\beta+\gamma e_2}(t)\V}\right].
        \end{align*}
        Recollecting terms we arrive to
        \begin{align*}
			&\mathfrak J=\sum_{k=1}^{2N+M}\frac{\ps{3\beta+\gamma e_2,\alpha_k}}{2(z_k-t)}\ps{V_{\beta+\gamma e_2}(t)\V}-\It\left[\frac{\ps{3\beta+\gamma e_2,\gamma e_j}}{2(y-t)}\ps{V_{\gamma e_j}(y)V_{\beta+\gamma e_2}(t)\V}\right].
		\end{align*}
        This shows that this term gives rise to a linear combination of the descendants $\L_{-1}$ and $\Wb_{-1}$ of $V_{\beta+\gamma e_2}$. More precisely we get
		\begin{align*}
			\mathfrak J=a\ps{\L_{-1}V_{\beta+\gamma e_2}(t)\V}+b\ps{\Wb_{-1}V_{\beta+\gamma e_2}(t)\V}
		\end{align*}
		where $a$ and $b$ are chosen in such a way that
		\[
    		3\beta+\gamma e_2=a\L_{-1}^{\beta+\gamma e_2}+b\Wb_{-1}^{\gamma+\beta e_2}.
		\]
		This allows to conclude that for $a$ and $b$ as above,
		\begin{align*}
			\ps{\psi_{-2,\frac2\gamma}V_{\beta}(t)\V}&=(\mu_{R,2}-\mu_{L,2})\Big(a\ps{\L_{-1}V_{\beta+\gamma e_2}(t)\V}+b\ps{\Wb_{-1}V_{\beta+\gamma e_2}(t)\V}\Big)\\
			&+2(\frac2\gamma-\gamma) (\mu_{L,1}+\mu_{R,1})\ps{V_{\beta+\gamma e_1}(t)\V}
		\end{align*}
		when $\chi=\frac2\gamma$, and likewise for $\chi=\gamma$
		\begin{align*}
			\ps{\psi_{-2,\gamma}V_{\beta}(t)\V}&=(\mu_{R,2}-\mu_{L,2})\Big(a\ps{\L_{-1}V_{\beta+\gamma e_2}(t)\V}+b\ps{\Wb_{-1}V_{\beta+\gamma e_2}(t)\V}\Big)\\
			&+2\gamma c_\gamma(\bm\mu)\ps{V_{\beta+2\gamma e_1}(t)\V}.
		\end{align*}
		This wraps up the proof of Theorem~\ref{thm:deg2}.
	\end{proof}

\section{Singular vectors at the level three}
    In the previous section we have defined the descendants at the level two of a primary field $V_\beta$, and shown that under specific assumptions on the value of the weight $\beta$ they give rise to singular vectors. In this concluding section we compute the singular vectors at the third level and show that when the primary field is a fully-degenerate field, one can define singular vectors out of them. These singular vectors in turn give rise to higher equations of motion, which under additional assumptions allow to obtain BPZ-type differential equaitons for a family of correlation functions.
    
	\subsection{Descendants at the third order} 
	As a first step in the derivation of the singular vectors at the level three, one first needs to make sense of the descendants at the third order. Namely we now define the $\L_{-1,2}$ and $\L_{-1,1,1}$ descendant.
	
	\subsubsection{The $\L_{-1,2}$ descendant}
	Following the same method as before in order to define the descendant at the third order $\L_{-1,2}$ we first need to understand the remainder term that arises in the computation of the regularized correlation functions containing these descendant fields inserted.
	To this end we introduce the notations for $i=1,2$:
	\begin{align*}
		G_i(x)\coloneqq\sum_{k=1}^N\frac{\ps{\beta,\alpha_k}}{2(z_k-t)}\Psi_i(x)-\Ir\left[\frac{\ps{\beta,\gamma e_j}}{2(y-t)}\Psi_{i,j}(x,y)\right]\qt{as well as}
	\end{align*}
	\begin{align*}
		&H_i(x)=\ps{V_{\gamma e_i}(x)\wick{\c\L^{\beta+\gamma e_i}_{-2}\c V_\beta(t)}\V}_{\delta,\eps}\coloneqq\\
		&\left(\sum_{k=1}^{2N+M}\frac{\ps{Q+\beta+\gamma e_i,\alpha_k}}{2(z_k-t)^2}-\sum_{k,l=1}^{2N+M}\frac{\ps{\alpha_k,\alpha_l}}{4(z_k-t)(z_l-t)}\right)\Psi_i(x)\\
		&-\Ir\left[\left(\frac{1+\frac{\ps{\beta+\gamma e_i,\gamma e_j}}2}{(y-t)^2}-\sum_{k=1}^{2N+M}\frac{\ps{\gamma e_i,\alpha_k}}{2(y-t)(z_k-t)}\right)\Psi_{i,j}(x,y)\right]\\
		&-\Is\left[\sum_{k=1}^{2N+M}\frac{\ps{\gamma e_i,\alpha_k}}{2(z_k-t)(z_k-y)}\Psi_{i,j}(x,y)\right]\\
		&+\Ir^2\left[-\frac{\ps{\gamma e_j,\gamma e_f}}{4(y-t)(z-t)}\Psi_{i,j,f}(x,y,z)\right]+\Is\times\Ir\left[\frac{\ps{\gamma e_j,\gamma e_f}}{2(z-t)(z-y)}\Psi_{i,j,f}(x,y,z)\right].
	\end{align*}
	\begin{lemma}\label{lemma:def_L12}
		Let us set
		\begin{equation}\label{eq:rem_L12}
			\begin{split}
				&\mathfrak L_{-(1,2);\delta,\eps,\rho}(\bm\alpha)\coloneqq\Is\left[\partial_x\left(H_i(x)+\frac{1}{x-t}G_i(x)\right)\right]\\
				&+\Is\left[\partial_x\left(\frac{1+\frac{\ps{\beta,\gamma e_i}}2}{(t-x)^2}\Psi_{i}(x)\right)\right]-\Is^2\left[\partial_x\left(\frac{\ps{\beta,\gamma e_j}}{2(t-x)(t-y)}\Psi_{i,j}(x,y)\right)\right].
			\end{split}
		\end{equation}
		Then the difference $\ps{\L_{-1,2}V_\beta(t)\V}_{\delta,\eps,\rho}-\mathfrak L_{-(1,2);\delta,\eps,\rho}(\bm\alpha)$ is $(P)$-class. Moreover define the $\L_{-1,2}V_\beta$ descendant within correlation function by setting for $t\in\R$ and $(\bm\alpha,\beta)\in\mc A_{N,M+1}$:
		\begin{equation}
			\ps{\L_{-1,2}V_\beta(t)\V}\coloneqq\lim\limits_{\delta,\eps,\rho\to0}\ps{\L_{-1,2}V_\beta(t)\V}_{\delta,\eps,\rho}-\tilde{\mathfrak L}_{-(1,2);\delta,\eps,\rho}(\bm\alpha).
		\end{equation}
		Then this descendant is such that, in the weak sense of derivatives:
		\begin{equation}
			\ps{\L_{-1,2}V_\beta(t)\V}=\partial_t\left(\sum_{k=1}^{2N+M}\frac{\partial_{z_k}}{t-z_k}+\frac{\Delta_{\alpha_k}}{(t-z_k)^2}\right)\ps{V_\beta(t)\V}.
		\end{equation}
	\end{lemma}
	\begin{proof}
		Algebraic manipulations following the exact same techniques as the ones developed in order to prove the results in the previous section show that the singular terms in the expression of $\ps{\L_{-1,2}V_\beta(t)\V}_{\delta,\eps,\rho}$ are given by $I_1+I_2+I_3$ where we have set 
		\begin{align*}
			I_1 \coloneqq &\Is\left[\frac{1+\frac{\ps{\beta,\gamma e_i}}2}{(t-x)^2}\left(\frac{2+\frac{\ps{\beta,\gamma e_i}}2}{t-x}+\frac{\gamma^2}{\bar x-x}\mathds1_{x\in\Heps}+\sum_{k=1}^{2N+M}\frac{\ps{\alpha_k,\gamma e_i}}{2(z_k-t)}\right)\Psi_i(x)\right]\\
			+&\Is\left[\sum_{k=1}^{2N+M}\frac{\ps{\beta,\alpha_k}}{2(t-x)(t-z_k)}\left(\frac{1+\frac{\ps{\beta,\gamma e_i}}2}{t-x}+\frac{\gamma^2}{\bar x-x}\mathds1_{x\in\Heps}+\sum_{l=1}^{2N+M}\frac{\ps{\alpha_l,\gamma e_i}}{2(z_l-t)}\right)\Psi_i(x)\right]\\
			+&\Is\left[\frac{\ps{\beta,\gamma e_i}}{2(t-x)}\left(\sum_{k=1}^{2N+M}\frac{\ps{Q+\beta,\alpha_k}}{2(t-z_k)^2}-\sum_{k,l=1}^{2N+M}\frac{\ps{\alpha_k,\alpha_l}}{4(t-z_k)(t-z_l)}\right)\Psi_i(x)\right]\\
			+&\Is\left[\frac{\ps{\beta,\gamma e_i}}{2(t-x)(t-\bar x)}\mathds 1_{x\in\Heps}\left(\frac{1+\frac{\ps{\beta,\gamma e_i}}2}{t-x}+\frac{1+\frac{\ps{\beta,\gamma e_i}}2}{t-\bar x}+\sum_{k=1}^{2N+M}\frac{\ps{\alpha_k,\gamma e_i}}{(z_k-t)}\right)\Psi_i(x)\right]\\
			&-\Is\left[\sum_{k=1}^{2N+M}\frac{\ps{\alpha_k,\gamma e_i}}{2(t-x)(t-z_k)^2}\Psi_i(x)\right],
		\end{align*}
		\begin{align*}
			I_2\coloneqq &-\It^2\left[\frac{\ps{\beta,\gamma e_j}}{2(t-x)(t-y)}\left(\frac{1+\frac{\ps{\beta,\gamma e_i}}2}{(t-x)}+\frac{\ps{\gamma e_i,\gamma e_j}}{2(y-x)}+\frac{\gamma^2\mathds 1_{x\in\Heps}}{\bar x-x}+\sum_{k=1}^{2N+M}\frac{\ps{\alpha_k,\gamma e_i}}{2(z_k-t)}\right)\Psi_{i,j}(x,y)\right]\\
			&-\It^2\left[\left(\frac{(1+\frac{\ps{\beta,\gamma e_i}}2)}{(t-x)^2}+\sum_{k=1}^{2N+M}\frac{\ps{\beta,\alpha_k}}{2(t-x)(t-z_k)}\right)\frac{\ps{\gamma e_i,\gamma e_j}}{2(y-x)}\Psi_{i,j}(x,y)\right]\\
			&+\It^2\left[\frac{\ps{\beta,\gamma e_i}\ps{\gamma e_i,\gamma e_j}}{2(t-x)(t-\bar x)(t-y)}\mathds 1_{x\in\Heps}\Psi_{i,j}(x,y)\right]\qt{and}
		\end{align*}
		\begin{align*}
			I_3\coloneqq \It^3\left[-\frac{\ps{\beta,\gamma e_i}\ps{\gamma e_j,\gamma e_f}}{8(t-x)(t-y)(t-z)}\Psi_{i,j,f}(x,y,z)\right].
		\end{align*}
		We can further simplify $I_1$ by using the identity
		\[
		\frac{(1+\frac{\ps{\beta,\gamma e_i}}2)\ps{\alpha_k,\gamma e_i}}{(t-x)^2(z_k-t)}=\frac{(1+\frac{\ps{\beta,\gamma e_i}}2)\ps{\alpha_k,\gamma e_i}}{(t-x)^2(z_k-x)}+\frac{(1+\frac{\ps{\beta,\gamma e_i}}2)\ps{\alpha_k,\gamma e_i}}{(t-x)(z_k-t)^2}-\frac{(1+\frac{\ps{\beta,\gamma e_i}}2)\ps{\alpha_k,\gamma e_i}}{(z_k-x)(z_k-t)^2}\cdot
		\]
		Likewise we can write that $I_2$ is given by the sum of $(P)$-terms and 
		\begin{align*}
			&-\Is\times\It\left[\frac{\ps{\beta,\gamma e_j}}{2(t-x)(t-y)}\left(\frac{1+\frac{\ps{\beta,\gamma e_i}}2}{(t-x)}+\frac{\ps{\gamma e_i,\gamma e_j}}{2(y-x)}+\frac{\gamma^2\mathds 1_{x\in\Heps}}{\bar x-x}+\sum_{k=1}^{2N+M}\frac{\ps{\alpha_k,\gamma e_i}}{2(z_k-t)}\right)\Psi_{i,j}(x,y)\right]\\
			&-\Is\times\Ir\left[\frac{\ps{\beta,\gamma e_i}}{2(t-x)}\left(\frac{1+\frac{\ps{\beta+\gamma e_i,\gamma e_j}}2}{(t-y)^2}-\frac{\gamma^2\mathds 1_{x\in\Heps}}{(t-y)(t-\bar y)}+\sum_{k=1}^{2N+M}\frac{\ps{\alpha_k,\gamma e_j}}{2(t-y)(z_k-t)}\right)\Psi_{i,j}(x,y)\right]\\
			&-\Is\times\It\left[\left(\frac{(1+\frac{\ps{\beta,\gamma e_i}}2)}{(t-x)^2}+\sum_{k=1}^{2N+M}\frac{\ps{\beta,\alpha_k}}{2(t-x)(t-z_k)}\right)\frac{\ps{\gamma e_i,\gamma e_j}}{2(y-x)}\Psi_{i,j}(x,y)\right]\\
			&-\Is\times\It\left[\frac{\ps{\beta,\gamma e_i}}{2(t-x)(t-\bar x)}\mathds 1_{x\in\Heps}\left(\frac{\ps{\gamma e_i,\gamma e_j}}{2(y-x)}+\frac{\ps{\gamma e_i,\gamma e_j}}{2(y-\bar x)}\right)\Psi_{i,j}(x,y)\right]
		\end{align*}
		and that $I_3$ simplifies to
		\begin{align*}
			I_3&=\text{$(P)$-class terms }+\Is\times\It^2\left[\frac{\ps{\beta,\gamma e_j}\ps{\gamma e_i,\gamma e_f}}{4(t-x)(t-y)(z-x)}\Psi_{i,j,f}(x,y,z)\right]\\
			&+\Is^2\times\Ir\left[\frac{\ps{\beta,\gamma e_i}\ps{\gamma e_j,\gamma e_f}}{4(t-x)(t-z)(y-z)}\Psi_{i,j,f}(x,y,z)\right]\\
			&-\Is\times\Ir^2\left[\frac{\ps{\beta,\gamma e_i}\ps{\gamma e_j,\gamma e_f}}{8(t-x)(t-y)(t-z)}\Psi_{i,j,f}(x,y,z)\right].
		\end{align*}
		As a consequence we are led to the equality
		\begin{align*}
			&\ps{\L_{-1,2}V_\beta(t)\V}_{\delta,\eps,\rho}=\text{$(P)$-class terms }+\Is\left[\partial_x\left(H_i(x)+\frac{1}{x-t}G_i(x)\right)\right]\\
			&+\Is\left[\left(\partial_x+\partial_{\bar x}\right)\left(\frac{\ps{\beta,\gamma e_i}}{2(t-x)(t-\bar x)}\mathds 1_{x\in\Heps}\Psi_i(x)\right)\right]\\
			&+\Is\left[\partial_x\left(\frac{1+\frac{\ps{\beta,\gamma e_i}}2}{(t-x)^2}\Psi_{i}(x)\right)\right]-\Is\times\Is\left[\partial_x\left(\frac{\ps{\beta,\gamma e_j}}{2(t-x)(t-y)}\Psi_{i,j}(x,y)\right)\right],
		\end{align*}
		allowing to recover the desired expression for the remainder term: 
		\begin{equation*}
			\begin{split}
				&\mathfrak L_{-(1,2);\delta,\eps,\rho}(\bm\alpha)=\Is\left[\partial_x\left(H_i(x)+\frac{1}{x-t}G_i(x)\right)\right]\\
				&+\Is\left[\partial_x\left(\frac{1+\frac{\ps{\beta,\gamma e_i}}2}{(t-x)^2}\Psi_{i}(x)\right)\right]-\Is^2\left[\partial_x\left(\frac{\ps{\beta,\gamma e_j}}{2(t-x)(t-y)}\Psi_{i,j}(x,y)\right)\right].
			\end{split}
		\end{equation*}
		
		As for the second part of the statement, we see from its expression that for $\ps{\beta,e_i}$ negative enough the remainder term $\tilde{\mathfrak L}_{-(1,2);\delta,\eps,\rho}(\bm\alpha)$ defined from Equation~\eqref{eq:rem_L12} goes to zero. Moreover in that case the equality is seen to hold in the sense of weak derivatives, see \textit{e.g.}~\cite[Equation (5.7)]{Toda_OPEWV}. We can conclude like in the proof of Proposition~\ref{prop:L11_der}.
	\end{proof}
	
	\subsubsection{The $\L_{-1,1,1}$ descendant}
	We now turn to the $\L_{-1,1,1}$ descendant, for which the analog statement is the following:
	\begin{lemma}\label{lemma:def_L111}
		One can find a quantity $\mathfrak L_{-(1,1,1);\delta,\eps,\rho}(\bm\alpha)$ for which $\ps{\L_{-1,1,1}V_\beta(t)\V}_{\delta,\eps,\rho}-\mathfrak L_{-(1,1,1);\delta,\eps,\rho}(\bm\alpha)$ is $(P)$-class. The $\L_{-1,1,1}V_\beta$ descendant defined for $t\in\R$ and $(\bm\alpha,\beta)\in\mc A_{N,M+1}$ via
		\begin{equation}
			\ps{\L_{-1,1,1}V_\beta(t)\V}\coloneqq\lim\limits_{\delta,\eps,\rho\to0}\ps{\L_{-1,1,1}V_\beta(t)\V}_{\delta,\eps,\rho}-\tilde{\mathfrak L}_{-(1,1,1);\delta,\eps,\rho}(\bm\alpha)
		\end{equation}
		satisfies the property that, in the weak sense,
		\begin{equation}
			\ps{\L_{-1,1,1}V_\beta(t)\V}=\partial_t^3\ps{V_\beta(t)\V}.
		\end{equation}
	\end{lemma}
	Along the proof of Lemma~\ref{lemma:def_L111}, we will provide an explicit expression for this remainder term. Namely we will show that it is given by
	\begin{eqs}\label{eq:L111}
		&\mathfrak L_{-(1,1,1);\delta,\eps,\rho}(\bm\alpha)=\Is\left[\partial_x\left(\frac{\ps{\beta,\gamma e_i}}{2(t-x)}\theta_i(x)+\vartheta_i(x)\right)\right]\\
		&+\Is\left[\partial_x\left(\frac{\ps{\beta,\gamma e_i}}{2(t-\bar x)}\Theta_i(x)\indeps-\Is\left[\frac{\ps{\beta,\gamma e_j}}{2(t-y)}\Theta_{i,j}(x,y)\right]\right)\right]\\
		&+A_{\delta,\eps,\rho}+B_{\delta,\eps,\rho}+C_{\delta,\eps,\rho},
	\end{eqs}
	\begin{eqs}
		&A_{\delta,\eps,\rho}\coloneqq\Is\left[\partial_x\left(\frac{\ps{\beta,\gamma e_i}\left(1+\frac{\ps{\beta,\gamma e_i}}2\right)}{2(t-x)^2}\Psi_i(x)\right)\right]
	\end{eqs}
	\begin{eqs}
		&B_{\delta,\eps,\rho}\coloneqq-\Is^2\left[\partial_{x}\left(\left(\frac{\ps{\beta,\gamma e_j}\left(1+\frac{\ps{\beta,\gamma e_j}}2\right)}{2(t-y)^2}+\frac{\ps{\beta,\gamma e_i}\ps{\beta,\gamma e_j}}{2(t-x)(t-y)}\right)\Psi_{i,j}(x,y)\right)\right]\\
		&+\Is\left[\partial_{x}\left(\left(\frac{\ps{\beta,\gamma e_j}\left(1+\frac{\ps{\beta,\gamma e_j}}2\right)}{2(t-\bar x)}+\frac{\ps{\beta,\gamma e_i}\ps{\beta,\gamma e_j}}{2(t-x)(t-\bar x)}\right)\Psi_{i}(x)\right)\right]
	\end{eqs}
	\begin{eqs}
		&C_{\delta,\eps,\rho}\coloneqq\Is^3\left[\partial_{x}\left(\left(\frac{\ps{\beta,\gamma e_j}\ps{\beta,\gamma e_f}}{4(t-y)(t-z)}\right)\Psi_{i,j,f}(x,y,z)\right)\right]\\
		&-\Is^2\left[\partial_{x}\left(\left(\frac{\ps{\beta,\gamma e_i}\ps{\beta,\gamma e_j}}{2(t-\bar x)(t-y)}\right)\Psi_{i,j}(x,y)\right)\indeps\right]\\
		&-\Is^2\left[\partial_{x}\left(\left(\frac{\ps{\beta,\gamma e_j}^2}{4(t-y)(t-\bar y)}\right)\Psi_{i,j}(x,y)\right)\mathds 1_{y\in\Heps}\right]
	\end{eqs}
	where $\vartheta_i$ is explicit but rather tedious to write (we provide its expression in Equation~\eqref{eq:toolongtowrite} below), while $\theta_i$ and $\Theta_{i,j}$ are given by
	\begin{eqs}
		&\theta_{i}(x)\coloneqq \sum_{k=1}^{2N+M}\left(\frac{\ps{3\beta+\gamma e_i,\alpha_k}}{2(t-z_k)}+\frac{\ps{\gamma e_i,\alpha_k}}{2(\bar x-z_k)}\indeps\right)\Psi_i(x)\\
		&-\Ir\left[\left(\frac{\ps{3\beta+\gamma e_i,\gamma e_j}}{2(t-y)}+\frac{\ps{\gamma e_i,\gamma e_j}}{2(\bar x-y)}\indeps\right)\Psi_{i,j}(x,y)\right]\\
		&-\Is\left[ \sum_{k=1}^{2N+M}\frac{\ps{\gamma e_j,\alpha_k}}{2(y-z_k)}\Psi_{i,j}(x,y)\right]+\Is\times\Ir\left[\frac{\ps{\gamma e_j,\gamma e_f}}{2(y-z)}\Psi_{i,j,f}(x,y,z)\right],
	\end{eqs}
	\begin{eqs}\label{eq:Theta}
		&\Theta_{i,j}(x,y)\coloneqq \sum_{k=1}^{2N+M}\left(\frac{\ps{3\beta+\gamma e_j,\alpha_k}}{2(t-z_k)}+\frac{\ps{\gamma e_i,\alpha_k}}{2(x-z_k)}\right)\Psi_{i,j}(x,y)\\
		&-\Ir\left[\left(\frac{\ps{3\beta+\gamma e_j,\gamma e_f}}{2(t-z)}+\frac{\ps{\gamma e_i,\gamma e_f}}{2(x-z)}\right)\Psi_{i,j,f}(x,y,z)\right]\\
		&-\Is\left[ \sum_{k=1}^{2N+M}\frac{\ps{\gamma e_f,\alpha_k}}{2(z-z_k)}\Psi_{i,j,f}(x,y,z)\right]+\Is\times\Ir\left[\frac{\ps{\gamma e_f,\gamma e_m}}{2(z-w)}\Psi_{i,j,f,m}(x,y,z,w)\right].
	\end{eqs}
	Finally $\Theta_i(x)$ is defined via the same expression as $\Theta_{i,j}(x,y)$ but with $\Psi_{i,j}(x,y)$ being replaced by $\Psi_{i}(x)$. 
	\begin{proof}
		Like for the previous lemma we only sketch the main steps and leave the details of the computations to the (very motivated!) reader. We can write the potentially singular terms in the expression of $\ps{\L_{-1,1,1}V_\beta(t)\V}_{\delta,\eps,\rho}$ as $I_1+I_2+I_3$, with
		\begin{align*}
			I_1\coloneqq &\Is\left[\left(\frac{\ps{\beta,\gamma e_i}\left(1+\frac{\ps{\beta,\gamma e_i}}2\right)\left(2+\frac{\ps{\beta,\gamma e_i}}2\right)}{2(t-x)^3}+\sum_{k=1}^{2N+M}\frac{3\ps{\beta,\gamma e_i}\left(1+\frac{\ps{\beta,\gamma e_i}}2\right)\ps{\beta,\alpha_k}}{4(t-x)^2(t-z_k)}\right)\Psi_i(x)\right]\\
			+&\Is\left[\frac{3\ps{\beta,\gamma e_i}}{2(t-x)}\left(\sum_{k=1}^{2N+M}\frac{\ps{\beta,\alpha_k}}{2(t-z_k)^2}+\sum_{k,l=1}^{2N+M}\frac{\ps{\beta,\alpha_k}\ps{\beta,\alpha_l}}{4(t-z_k)(t-z_l)}\right)\Psi_i(x)\right]\\
			+&\Is\left[\mathds 1_{x\in\Heps}\frac{3\ps{\beta,\gamma e_i}^2}{4(t-x)(t-\bar x)}\left(\frac{1+\frac{\ps{\beta,\gamma e_i}}2}{t-x}+\sum_{k=1}^{2N+M}\frac{\ps{\beta,\alpha_k}}{2(t-z_k)}\right)\Psi_i(x)\right],
		\end{align*}
		\begin{align*}
			I_2\coloneqq &-\It^2\left[\left(\frac{3\ps{\beta,\gamma e_i}\left(1+\frac{\ps{\beta,\gamma e_i}}2\right)\ps{\beta,\gamma e_j}}{4(t-x)^2(t-y)}+\frac{3\ps{\beta,\gamma e_i}\ps{\beta,\gamma e_j}\ps{\beta,\alpha_k}}{8(t-x)(t-y)(t-z_k)}\right)\Psi_{i,j}(x,y)\right]\\
			&-\It^2\left[\frac{3\ps{\beta,\gamma e_i}^2\ps{\beta,\gamma e_j}}{4(t-x)(t-\bar x)(t-y)}\mathds 1_{x\in\Heps}\Psi_{i,j}(x,y)\right],\qt{and}
		\end{align*}
		\begin{align*}
			I_3\coloneqq \It^3\left[\frac{\ps{\beta,\gamma e_i}\ps{\beta,\gamma e_j}\ps{\beta,\gamma e_f}}{8(t-x)(t-y)(t-z)}\Psi_{i,j,f}(x,y,z)\right].
		\end{align*}
		Our goal is, like before, to rewrite the latter to make appear derivatives of the correlation functions. And to do so we rely on the same techniques as above.
		
		For the sake of simplicity let us first assume that $\mu_{B,1}=\mu_{B,2}=0$ so that we can discard terms with $\bar x$. Then using symmetrization $I_1$ is given by $I_1=\text{$(P)$-class terms }+I_1^1$ where
		\begin{align*}
			&I^1_1\coloneqq\Is\left[\frac{\ps{\beta,\gamma e_i}\left(1+\frac{\ps{\beta,\gamma e_i}}2\right)}{2(t-x)^2}\left(\frac{2+\frac{\ps{\beta,\gamma e_i}}2}{t-x}+\frac{\gamma^2\mathds 1_{x\in\Heps}}{\bar x-x}+\sum_{k=1}^{2N+M}\frac{\ps{\gamma e_i,\alpha_k}}{2(z_k-x)}\right)\Psi_i(x)\right]\\
			&+\Is\left[\frac{\ps{\beta,\gamma e_i}}{2(t-x)}\left(\frac{1+\frac{\ps{\beta,\gamma e_i}}2}{t-x}+\frac{\gamma^2\mathds 1_{x\in\Heps}}{\bar x-x}+\sum_{l=1}^{2N+M}\frac{\ps{\gamma e_i,\alpha_l}}{2(z_l-x)}\right)\theta_i^1(x)\right]+\Is\Big[\frac{\ps{\beta,\gamma e_i}}{2(t-x)}\vartheta_i^1(x)\Big]\\
			&\qt{with}\theta_i^1(x)\coloneqq \sum_{k=1}^{2N+M}\frac{\ps{3\beta+\gamma e_i,\alpha_k}}{2(t-z_k)}\Psi_i(x)\qt{and}\\
			&\vartheta_i^1(x)\coloneqq \left(\sum_{k=1}^{2N+M}\frac{\ps{3\beta+(1+\frac{\ps{\beta,\gamma e_i}}{2})\gamma e_i,\alpha_k}}{2(t-z_k)^2}+\sum_{k,l=1}^{2N+M}\frac{3\ps{\beta,\alpha_k}\ps{\beta,\alpha_l}+\ps{3\beta+\gamma e_i,\alpha_k}\ps{\gamma e_i,\alpha_l}}{4(t-z_k)(t-z_l)}\right)\Psi_i(x).
		\end{align*}
		Based on the same reasoning as in the previous sections we thus end up with
		\begin{align*}
			I_1&=\text{$(P)$-class terms}+R_1^1\\
			&+\Is\left[\partial_x\left(\frac{\ps{\beta,\gamma e_i}\left(1+\frac{\ps{\beta,\gamma e_i}}2\right)}{2(t-x)^2}\Psi_i(x)+\frac{\ps{\beta,\gamma e_i}}{2(t-x)}\theta_i^1(x)+\vartheta_i^1(x)\right)\right]\qt{where we have set}
		\end{align*}
		\begin{align*}
			&R_1^1\coloneqq-\Is\times\It\left[\left(\frac{\ps{\beta,\gamma e_i}\left(1+\frac{\ps{\beta,\gamma e_i}}2\right)}{2(t-x)^2}+\sum_{k=1}^{2N+M}\frac{\ps{\beta,\gamma e_i}\ps{3\beta+\gamma e_i,\alpha_k}}{4(t-x)(t-z_k)}\right)\frac{\ps{\gamma e_i,\gamma e_j}}{2(x-y)}\Psi_{i,j}(x,y)\right].
		\end{align*}
		
		We have thus treated the one-fold integral: let us now turn to the two-fold ones. And to start with, we first rewrite $I_2$ as 
		\begin{align*}
			I_2&=\text{$(P)$-class terms }\\
			&-\Is^2\left[\left(\frac{3\ps{\beta,\gamma e_i}\left(1+\frac{\ps{\beta,\gamma e_i}}2\right)\ps{\beta,\gamma e_j}}{4(t-x)^2(t-y)}+\frac{3\ps{\beta,\gamma e_i}\ps{\beta,\gamma e_j}\ps{\beta,\alpha_k}}{8(t-x)(t-y)(t-z_k)}\right)\Psi_{i,j}(x,y)\right]\\
			&-\Is\times\Ir\left[\left(\frac{3\ps{\beta,\gamma e_i}\left(1+\frac{\ps{\beta,\gamma e_i}}2\right)\ps{\beta,\gamma e_j}}{4(t-x)^2(t-y)}\right)\Psi_{i,j}(x,y)\right]\\
			&-\Is\times\Ir\left[\left(\frac{3\ps{\beta,\gamma e_i}\ps{\beta,\gamma e_j}\left(1+\frac{\ps{\beta,\gamma e_j}}2\right)}{4(t-x)(t-y)^2}+\sum_{k=1}^{2N+M}\frac{3\ps{\beta,\gamma e_i}\ps{\beta,\gamma e_j}\ps{\beta,\alpha_k}}{8(t-x)(t-y)(t-z_k)}\right)\Psi_{i,j}(x,y)\right]
		\end{align*}
		Then using symmetrization and expressing the above in terms of derivatives like before some tedious computations show that
		\begin{align*}
			I_2+R_1^1=\text{$(P)$-class terms }+I_2^1+I_2^2+R_2^1+R_2^2,\qt{where:}
		\end{align*}
		\begin{align*}
			&I_2^1\coloneqq-\Is^2\left[\partial_{x}\left(\left(\frac{\ps{\beta,\gamma e_j}(1+\frac{\ps{\beta,\gamma e_j}}2)}{2(t-y)^2}+\frac{\ps{\beta,\gamma e_i}\ps{\beta,\gamma e_j}}{2(t-x)(t-y)}\right)\Psi_{i,j}(x,y)\right)\right]\\
			&-\Is^2\left[\partial_{x}\left(\sum_{k=1}^{2N+M}\left(\frac{\ps{\beta,\gamma e_i}\ps{\gamma e_j,\alpha_k}}{4(t-x)(y-z_k)}+\frac{\ps{\beta,\gamma e_j}\ps{3\beta+\gamma e_j,\alpha_k}}{4(t-y)(t-z_k)}+\frac{\ps{\beta,\gamma e_j}\ps{\gamma e_i,\alpha_k}}{4(t-y)(x-z_k)}\right)\Psi_{i,j}(x,y)\right)\right]\\
			&-\Is^2\left[\partial_{x}\left(\vartheta_{i,j}^1(x,y)\right)\right]\qt{with}
		\end{align*}
		\begin{align*}
			\vartheta_{i,j}^1(x,y)&\coloneqq \sum_{k=1}^{2N+M}\left(\frac{\ps{\gamma e_j,\alpha_k}}{4(y-z_k)^2}+\frac{\ps{\beta,\gamma e_j}\ps{\gamma e_j,\alpha_k}}{4(t-z_k)(y-z_k)}\right)\Psi_{i,j}(x,y)\\
			&+\sum_{k,l=1}^{2N+M}\left(\frac{\ps{\gamma e_j,\alpha_k}\ps{\gamma e_j,\alpha_l}}{4(y-z_k)(y-z_l)}+\frac{\ps{3\beta+2\gamma e_i,\alpha_k}\ps{\gamma e_j,\alpha_l}}{4(t-z_k)(y-z_l)}\right)\Psi_{i,j}(x,y)\qt{and}
		\end{align*}
		\begin{align*}
			&I_2^2\coloneqq -\Is\times\Ir\left[\partial_x\left(\left(\frac{\ps{\beta,\gamma e_i}\ps{3\beta+\gamma e_i,\gamma e_j}}{4(t-x)(t-y)}+\vartheta_{i,j}^2(x,y)\right)\Psi_{i,j}(x,y)\right)\right]\qt{with}
		\end{align*}
		\begin{align*}
			&\vartheta_{i,j}^2(x,y)\coloneqq\left(\frac{\ps{3\beta+\gamma e_i,\gamma e_j}\left(1+\frac{\ps{\beta+\gamma e_i,\gamma e_j}}2\right)}{2(t-y)^2}\right.\\
			&\left.+\sum_{k=1}^{2N+M}\frac{\ps{3\beta+\gamma e_i,\gamma e_j}\ps{\gamma e_i,\alpha_k}+\ps{3\beta+\gamma e_i,\alpha_k}\ps{\gamma e_i,\gamma e_j}+6\ps{\beta,\gamma e_j}\ps{\beta,\alpha_k}}{4(t-y)(t-z_k)}\right)\Psi_{i,j}(x,y).
		\end{align*}
		The extra terms $R_2^1$ and $R_2^2$ are equal to 
		\begin{align*}
			&R_2^1\coloneqq\Is^2\times\It\left[\left(\frac{\ps{\beta,\gamma e_j}(1+\frac{\ps{\beta,\gamma e_j}}2)}{2(t-y)^2}+\frac{\ps{\beta,\gamma e_i}\ps{\beta,\gamma e_j}}{2(t-x)(t-y)}\right)\frac{\ps{\gamma e_i,\gamma e_f}}{2(x-z)}\Psi_{i,j,f}(x,y,z)\right]\\
			&+\Is^2\times\It\left[\sum_{k=1}^{2N+M}\left(\frac{\ps{\beta,\gamma e_i}\ps{\gamma e_j,\alpha_k}}{4(t-x)(y-z_k)}\right)\frac{\ps{\gamma e_i,\gamma e_f}}{2(x-z)}\Psi_{i,j,f}(x,y,z)\right]\\
			&+\Is^2\times\It\left[\sum_{k=1}^{2N+M}\left(\frac{\ps{\beta,\gamma e_j}\ps{3\beta+\gamma e_j,\alpha_k}}{4(t-y)(t-z_k)}+\frac{\ps{\beta,\gamma e_j}\ps{\gamma e_i,\alpha_k}}{4(t-y)(x-z_k)}\right)\frac{\ps{\gamma e_i,\gamma e_f}}{2(x-z)}\Psi_{i,j,f}(x,y,z)\right]
		\end{align*}
		\begin{align*}
			&R_2^2\coloneqq \Is\times\Ir\times\It\left[\left(\frac{\ps{\beta,\gamma e_i}\ps{3\beta+\gamma e_i,\gamma e_j}}{4(t-x)(t-y)}\right)\frac{\ps{\gamma e_i,\gamma e_f}}{2(x-z)}\Psi_{i,j,f}(x,y,z)\right].
		\end{align*}
		An important point to highlight is that in the $(P)$-class terms we have included expressions of the form $-\Is^2\left[\frac{1}{x-y}f_{i,j}(x,y)\Psi_{i,j}(x,y)\right]$ where $f_{i,j}$ is regular at $x=y$, $x=t$ and $y=t$. At first sight it may not be clear why the singularity $\frac1{x-y}$ would be integrable. However we can separate between the case where $i\neq j$, for which $\Psi_{i,j}(x,y)\simeq \norm{x-y}^{\gamma^2}$ as $\norm{x-y}\to0$, and when $i=j$ in which case by symmetry in $x,y$ the latter can be reduced to a regular quantity. This allows to treat the two-fold integrals and discard such terms.
		
		Now to transform the three-fold integrals we first split $I_3$ as $I_3=\text{$(P)$-class terms}+I_3^1$ with
		\begin{align*}
			I_3^1&\coloneqq\Is^3\left[\frac{\ps{\beta,\gamma e_i}\ps{\beta,\gamma e_j}\ps{\beta,\gamma e_f}}{8(t-x)(t-y)(t-z)}\Psi_{i,j,f}(x,y,z)\right]\\
			&+\Is^2\times\Ir\left[\frac{3\ps{\beta,\gamma e_i}\ps{\beta,\gamma e_j}\ps{\beta,\gamma e_f}}{8(t-x)(t-y)(t-z)}\Psi_{i,j,f}(x,y,z)\right]\\
			&+\Is\times\Ir^2\left[\frac{3\ps{\beta,\gamma e_i}\ps{\beta,\gamma e_j}\ps{\beta,\gamma e_f}}{8(t-x)(t-y)(t-z)}\Psi_{i,j,f}(x,y,z)\right]
		\end{align*}
		Likewise we see that $R_2^1+R_2^2=I_{s,s,s}+I_{s,s,r}+I_{s,r,r}$ where
		\begin{align*}
			I_{s,s,s}\coloneqq&\Is^3\left[\left(\frac{\ps{\beta,\gamma e_j}(1+\frac{\ps{\beta,\gamma e_j}}2)}{2(t-y)^2}+\frac{\ps{\beta,\gamma e_i}\ps{\beta,\gamma e_j}}{2(t-x)(t-y)}\right)\frac{\ps{\gamma e_i,\gamma e_f}}{2(x-z)}\Psi_{i,j,f}(x,y,z)\right]\\
			&+\Is^3\left[\sum_{k=1}^{2N+M}\left(\frac{\ps{\beta,\gamma e_i}\ps{\gamma e_j,\alpha_k}}{4(t-x)(y-z_k)}\right)\frac{\ps{\gamma e_i,\gamma e_f}}{2(x-z)}\Psi_{i,j,f}(x,y,z)\right]\\
			&+\Is^3\left[\sum_{k=1}^{2N+M}\left(\frac{\ps{\beta,\gamma e_j}\ps{3\beta+\gamma e_j,\alpha_k}}{4(t-y)(t-z_k)}+\frac{\ps{\beta,\gamma e_j}\ps{\gamma e_i,\alpha_k}}{4(t-y)(x-z_k)}\right)\frac{\ps{\gamma e_i,\gamma e_f}}{2(x-z)}\Psi_{i,j,f}(x,y,z)\right]
		\end{align*}
		\begin{align*}
			I_{s,s,r}\coloneqq&\Is^2\times\Ir\left[\left(\frac{\ps{\beta,\gamma e_j}(1+\frac{\ps{\beta,\gamma e_j}}2)}{2(t-y)^2}+\frac{\ps{\beta,\gamma e_i}\ps{\beta,\gamma e_j}}{2(t-x)(t-y)}\right)\frac{\ps{\gamma e_i,\gamma e_f}}{2(x-z)}\Psi_{i,j,f}(x,y,z)\right]\\
			&+\Is^2\times\Ir\left[\left(\frac{\ps{\beta,\gamma e_i}\ps{3\beta+\gamma e_i,\gamma e_f}}{4(t-x)(t-z)}\right)\frac{\ps{\gamma e_i,\gamma e_j}}{2(x-y)}\Psi_{i,j,f}(x,y,z)\right]\\
			&+\Is^2\times\Ir\left[\sum_{k=1}^{2N+M}\left(\frac{\ps{\beta,\gamma e_j}\ps{3\beta+\gamma e_j,\alpha_k}}{4(t-y)(t-z_k)}+\frac{\ps{\beta,\gamma e_j}\ps{\gamma e_i,\alpha_k}}{4(t-y)(x-z_k)}\right)\frac{\ps{\gamma e_i,\gamma e_f}}{2(x-z)}\Psi_{i,j,f}(x,y,z)\right]\\
			&+\Is^2\times\Ir\left[\sum_{k=1}^{2N+M}\left(\frac{\ps{\beta,\gamma e_i}\ps{\gamma e_j,\alpha_k}}{4(t-x)(y-z_k)}\right)\frac{\ps{\gamma e_i,\gamma e_f}}{2(x-z)}\Psi_{i,j,f}(x,y,z)\right].
		\end{align*}
		\begin{align*}
			I_{s,r,r}\coloneqq&\Is\times\Ir^2\left[\left(\frac{\ps{\beta,\gamma e_i}\ps{3\beta+\gamma e_i,\gamma e_j}}{4(t-x)(t-y)}\right)\frac{\ps{\gamma e_i,\gamma e_f}}{2(x-z)}\Psi_{i,j,f}(x,y,z)\right].
		\end{align*}
		Gathering these terms together with that coming from $I_3$ we arrive to the equality $$I_3+R_2^1+R_2^2=\text{$(P)$-class terms }+I_3^1+I_3^2+I_3^3+R_3^1+R_3^2,\qt{with}$$
		\begin{align*}
			I_3^1=\Is^3\left[\partial_x\left(\left(\frac{\ps{\beta,\gamma e_j}\ps{\beta,\gamma e_f}}{4(t-y)(t-z)}+\sum_{k=1}^{2N+M}\frac{\ps{\beta,\gamma e_j}\ps{\gamma e_f,\alpha_k}}{4(t-y)(z-z_k)}+\sum_{k,l=1}^{2N+M}\frac{\ps{\gamma e_j,\alpha_k}\ps{\gamma e_f,\alpha_l}}{4(y-z_k)(z-z_l)}\right)\Psi_{i,j,f}(x,y,z)\right)\right];
		\end{align*}
		\begin{align*}
			I_3^2&=\Is^2\times\Ir\left[\partial_x\left(\left(\frac{\ps{\beta,\gamma e_i}\ps{\gamma e_j,\gamma e_f}}{4(t-x)(y-z)}+\vartheta_{i,j,f}(x,y,z)\right)\Psi_{i,j,f}(x,y,z)\right)\right]\\
			&+\Is^2\times\Ir\left[\partial_x\left(\left(\frac{\ps{\beta,\gamma e_j}\ps{3\beta+\gamma e_j,\gamma e_f}}{4(t-y)(t-z)}+\frac{\ps{\beta,\gamma e_j}\ps{\gamma e_i,\gamma e_f}}{4(t-y)(x-z)}\right)\Psi_{i,j,f}(x,y,z)\right)\right];
		\end{align*}
		\begin{align*}
			I_3^3=\Is\times\Ir^2\left[\partial_x\left(\frac{\ps{3\beta+\gamma e_i,\gamma e_j}\ps{\gamma e_j,\gamma e_f}+3\ps{\beta,\gamma e_j}\ps{\beta,\gamma e_f}}{4(t-y)(t-z)}\Psi_{i,j,f}(x,y,z)\right)\right].
		\end{align*}
		In the above $\vartheta_{i,j,f}(x,y,z)$ is defined by
		\begin{eqs}
			&\vartheta_{i,j,f}(x,y,z)\coloneqq\sum_{k=1}^{2N+M}\left(\frac{\ps{3\beta+\gamma e_i,\alpha_k}}{2(t-z_k)}+\frac{\ps{\gamma e_j,\alpha_k}}{2(x-z_k)}\right)\frac{\ps{\gamma e_j,\gamma e_f}}{2(y-z)}+\frac{\ps{\gamma e_j,\alpha_k}\ps{\gamma e_i,\gamma e_f}}{4(y-z_k)(x-z)}\\
			&+\sum_{k=1}^{2N+M}\left(\frac{\ps{3\beta+\gamma e_i,\gamma e_f}}{2(t-z)}+\frac{\ps{\gamma e_j,\gamma e_f}}{2(x-z)}\right)\frac{\ps{\gamma e_j,\alpha_k}}{2(y-z_k)}+\frac{\ps{\gamma e_j,\gamma e_f}\ps{\gamma e_i,\alpha_k}}{4(y-z)(x-z_k)}\\
			&+\left(\frac{\ps{3\beta+\gamma e_i,\gamma e_f}+\ps{\beta,\gamma e_j}}{2(t-z)}+\frac{1+\ps{\gamma e_j,\gamma e_f}}{2(y-z)}\right)\frac{\ps{\gamma e_j,\gamma e_f}}{2(y-z)}
		\end{eqs}
		while the additional terms are given by 
		\begin{align*}
			R_3^1=-\Is^3\times\It\left[\left(\frac{\ps{\beta,\gamma e_j}\ps{\beta,\gamma e_f}}{4(t-y)(t-z)}+\sum_{k=1}^{2N+M}\frac{\ps{\beta,\gamma e_j}\ps{\gamma e_f,\alpha_k}}{4(t-y)(z-z_k)}\right)\frac{\ps{\gamma e_i,\gamma e_m}}{2(x-w)}\Psi_{i,j,f,m}(x,y,z,w)\right];
		\end{align*}
		\begin{align*}
			R_3^2=-&\Is^2\times\Ir\times\It\left[\frac{\ps{\beta,\gamma e_i}\ps{\gamma e_j,\gamma e_f}\ps{\gamma e_i,\gamma e_m}}{8(t-x)(y-z)(x-w)}\Psi_{i,j,f,m}(x,y,z,w)\right]\\
			-&\Is^2\times\Ir\times\It\left[\frac{\ps{\beta,\gamma e_j}\ps{3\beta+\gamma e_j,\gamma e_f}\ps{\gamma e_i,\gamma e_m}}{8(t-y)(t-z)(x-w)}\Psi_{i,j,f,m}(x,y,z,w)\right]\\
			-&\Is^2\times\Ir\times\It\left[\frac{\ps{\beta,\gamma e_j}\ps{\gamma e_i,\gamma e_f}\ps{\gamma e_i,\gamma e_m}}{8(t-y)(x-z)(x-w)}\Psi_{i,j,f,m}(x,y,z,w)\right].
		\end{align*}
		
		We now have some new four-fold integrals to treat. We see that they give rise to the following remainder terms:
		\begin{align*}
			&-\Is^3\times\Ir\left[\partial_x\left(\left(\frac{\ps{\beta,\gamma e_j}\ps{\gamma e_f,\gamma e_m}}{4(t-y)(z-w)}\Psi_{i,j,f,m}(x,y,z,w)+\vartheta_{i,j,f,m}\right)\right)\right]\\
			&-\Is^2\times\Ir^2\left[\partial_x\left(\left(\frac{\ps{\gamma e_i,\gamma e_f}\ps{\gamma e_j,\gamma e_m}}{4(x-z)(y-w)}+\frac{\ps{3\beta+\gamma e_i+\gamma e_j,\gamma e_f}\ps{\gamma e_j,\gamma e_m}}{4(t-z)(y-w)}\right)\Psi_{i,j,f,m}(x,y,z,w)\right)\right]\\
			&+\Is^3\times\Ir^2\left[\partial_x\left(\frac{\ps{\gamma e_j,\gamma e_n}\ps{\gamma e_f,\gamma e_m}}{4(y-w)(z-\varpi)}\Psi_{i,j,f,m,n}(x,y,z,w,\varpi)\right)\right],
		\end{align*}
		\begin{equation}
			\vartheta_{i,j,f,m}(x,y,z,w)\coloneqq\sum_{k=1}^{2N+M}\frac{\ps{\gamma e_j,\alpha_k}\ps{\gamma e_f,\gamma e_m}}{2(y-z_k)(z-w)}\Psi_{i,j,f,m}(x,y,z,w).
		\end{equation}
		Recollecting all the terms that we have obtained above we recover the expression provided in Equation~\eqref{eq:L111} up to the $\indeps$ term in that we get
		\begin{eqs}\label{eq:L111_mu}
			&\mathfrak L^{\mu_B=0}_{-(1,1,1);\delta,\eps,\rho}(\bm\alpha)=\Is\left[\partial_x\left(\vartheta_i^{\mu_B=0}(x)+\frac{\ps{\beta,\gamma e_i}}{2(t-x)}\theta_i(x)\right)\right]\\
			&-\Is^2\left[\partial_x\left(\frac{\ps{\beta,\gamma e_j}}{2(t-y)}\Theta_{i,j}(x,y)\right)\right]\\
			&+\Is\left[\partial_x\left(\frac{\ps{\beta,\gamma e_i}\left(1+\frac{\ps{\beta,\gamma e_i}}2\right)}{2(t-x)^2}\Psi_i(x)\right)\right]\\
			&-\Is^2\left[\partial_{x}\left(\left(\frac{\ps{\beta,\gamma e_j}\left(1+\frac{\ps{\beta,\gamma e_j}}2\right)}{2(t-y)^2}+\frac{\ps{\beta,\gamma e_i}\ps{\beta,\gamma e_j}}{2(t-x)(t-y)}\right)\Psi_{i,j}(x,y)\right)\right]\\
			&+\Is^3\left[\partial_{x}\left(\left(\frac{\ps{\beta,\gamma e_j}\ps{\beta,\gamma e_f}}{4(t-y)(t-z)}\right)\Psi_{i,j,f}(x,y,z)\right)\right]
		\end{eqs}
		where we see that $\vartheta_i^{\mu_B=0}(x)$ admits the very nice, simple and for sure accurate expression
		\begin{eqs}\label{eq:toolongtowrite}
			\vartheta_i^{\mu_B=0}(x)&\coloneqq \vartheta_i^1(x)-\Is\left[\vartheta_{i,j}^1(x,y)\right]-\Ir\left[\vartheta_{i,j}^2(x,y)\right]+\Is\times\Ir\left[\vartheta_{i,j,f}(x,y,z)\right]\\
			&+\Is^2\left[\sum_{k,l=1}^{2N+M}\frac{\ps{\gamma e_j,\alpha_k}\ps{\gamma e_f,\alpha_l}}{4(y-z_k)(z-z_l)}\Psi_{i,j,f}(x,y,z)\right]\\
			&+\Ir^2\left[\frac{\ps{3\beta+\gamma e_i,\gamma e_j}\ps{\gamma e_j,\gamma e_f}+3\ps{\beta,\gamma e_j}\ps{\beta,\gamma e_f}}{4(t-y)(t-z)}\Psi_{i,j,f}(x,y,z)\right]\\
			&-\Is^2\times\Ir\left[\vartheta_{i,j,f,m}(x,y,z,w)\right]\\
			&-\Is\times\Ir^2\left[\left(\frac{\ps{\gamma e_i,\gamma e_f}\ps{\gamma e_j,\gamma e_m}}{4(x-z)(y-w)}+\frac{\ps{3\beta+\gamma e_i+\gamma e_j,\gamma e_f}\ps{\gamma e_j,\gamma e_m}}{4(t-z)(y-w)}\right)\Psi_{i,j,f,m}(x,y,z,w)\right]\\
			&+\Is^2\times\Ir^2\left[\frac{\ps{\gamma e_j,\gamma e_n}\ps{\gamma e_f,\gamma e_m}}{4(y-w)(z-\varpi)}\Psi_{i,j,f,m,n}(x,y,z,w,\varpi)\right].
		\end{eqs}
		
		Now, let us no longer make the simplifying assumption that $\mu_{B,1}=\mu_{B,2}=0$. Then we get the extra term in $I_1$:
		\begin{align*}
			&\Is\left[\mathds 1_{x\in\Heps}\frac{3\ps{\beta,\gamma e_i}^2}{4(t-x)(t-\bar x)}\left(\frac{1+\frac{\ps{\beta,\gamma e_i}}2}{t-x}+\sum_{k=1}^{2N+M}\frac{\ps{\beta,\alpha_k}}{2(t-z_k)}\right)\Psi_i(x)\right],
		\end{align*}
		and its counterpart for $I_2$:
		\begin{align*}
			&-\Is^2\left[\frac{3\ps{\beta,\gamma e_i}^2\ps{\beta,\gamma e_j}}{4(t-x)(t-\bar x)(t-y)}\mathds 1_{x\in\Heps}\Psi_{i,j}(x,y)\right]\\
			&-\Is\times\Ir\left[\frac{3\ps{\beta,\gamma e_i}^2\ps{\beta,\gamma e_j}}{4(t-x)(t-\bar x)(t-y)}\mathds 1_{x\in\Heps}\Psi_{i,j}(x,y)\right]\\
			&-\Is\times\Ir\left[\frac{3\ps{\beta,\gamma e_i}\ps{\beta,\gamma e_j}^2}{4(t-x)(t-y)(t-\bar y)}\mathds 1_{y\in\Heps}\Psi_{i,j}(x,y)\right].
		\end{align*}
		Moreover the algebraic manipulations that allowed to derive the expression of the remainder term in the case $\mu_{B,1}=\mu_{B,2}=0$ remain valid without this assumption, the main difference lying in the fact that derivatives give rise in addition to a $\frac{\gamma^2}{x-\bar x}$ term. For one-fold integrals this corresponds to the additional term
		\begin{align*}
			&+\Is\left[\left(\frac{\ps{\beta,\gamma e_i}\left(1+\frac{\ps{\beta,\gamma e_i}}2\right)}{2(t-x)^2}\Psi_i(x)+\frac{\ps{\beta,\gamma e_i}}{2(t-x)}\theta_i^1(x)\right)\frac{\gamma^2\indeps}{x-\bar x}\right].
		\end{align*}
		For two-fold integrals this yields the extra term (we have used the symmetry in $x,\bar x$ to cancel some quantities)
		\begin{align*}
			&-\Is^2\left[\left(\frac{\ps{\beta,\gamma e_i}\ps{\beta,\gamma e_j}}{2(t-x)(t-y)}+\sum_{k=1}^{2N+M}\frac{\ps{\beta,\gamma e_i}\ps{\gamma e_j,\alpha_k}}{4(t-x)(y-z_k)}\right)\frac{\gamma^2\indeps}{x-\bar x}\Psi_{i,j}(x,y)\right]\\
			&-\Is\times\Ir\left[\left(\frac{\ps{\beta,\gamma e_i}\ps{3\beta+\gamma e_i,\gamma e_j}}{4(t-x)(t-y)}\right)\frac{\gamma^2\indeps}{x-\bar x}\Psi_{i,j}(x,y)\right].
		\end{align*}
		As for the three-fold integrals, we get the additional quantities
		\begin{align*}
			&\Is^2\times\Ir\left[\frac{\ps{\beta,\gamma e_i}\ps{\gamma e_j,\gamma e_f}\gamma^2}{4(t-x)(y-z)(x-\bar x)}\Psi_{i,j,f,m}(x,y,z,w)\indeps\right]\\
			+&\Is^2\times\Ir\left[\frac{\ps{\beta,\gamma e_j}\ps{\gamma e_i,\gamma e_f}\gamma^2}{4(t-y)(x-z)(x-\bar x)}\Psi_{i,j,f,m}(x,y,z,w)\indeps\right].
		\end{align*}
		The four-fold and five-fold integrals only give rise to $(P)$-class terms.
		
		We now treat successively these integrals. In the same fashion as the integral over $\Is^2$, the one-fold integral can be rewritten as (up to $(P)$-class terms)
		\begin{align*}
			&\Is\left[\partial_{x}\left(\left(\frac{\ps{\beta,\gamma e_i}(1+\frac{\ps{\beta,\gamma e_i}}2)}{2(t-\bar x)^2}+\frac{\ps{\beta,\gamma e_i}^2}{2(t-x)(t-\bar x)}\right)\Psi_{i}(x)\indeps\right)\right]\\
			&+\Is\left[\partial_{x}\left(\sum_{k=1}^{2N+M}\left(\frac{\ps{\beta,\gamma e_i}\ps{\gamma e_i,\alpha_k}}{4(t-x)(\bar x-z_k)}+\frac{\ps{\beta,\gamma e_i}\ps{3\beta+\gamma e_i,\alpha_k}}{4(t-\bar x)(t-z_k)}+\frac{\ps{\beta,\gamma e_i}\ps{\gamma e_i,\alpha_k}}{4(t-\bar x)(x-z_k)}\right)\Psi_{i}(x)\indeps\right)\right]\\
			&-\Is\times\It\left[\left(\frac{\ps{\beta,\gamma e_i}(1+\frac{\ps{\beta,\gamma e_i}}2)}{2(t-\bar x)^2}+\frac{\ps{\beta,\gamma e_i}^2}{2(t-x)(t-\bar x)}\right)\frac{\ps{\gamma e_i,\gamma e_j}}{2(x-y)}\Psi_{i,j}(x,y)\indeps\right]\\
			&-\Is\times\It\left[\sum_{k=1}^{2N+M}\frac{\ps{\beta,\gamma e_i}\ps{\gamma e_i,\alpha_k}}{4(t-x)(\bar x-z_k)}\frac{\ps{\gamma e_i,\gamma e_j}}{2(x-y)}\Psi_{i,j}(x,y)\indeps\right]\\
			&-\Is\times\It\left[\sum_{k=1}^{2N+M}\left(\frac{\ps{\beta,\gamma e_i}\ps{3\beta+\gamma e_i,\alpha_k}}{4(t-\bar x)(t-z_k)}+\frac{\ps{\beta,\gamma e_j}\ps{\gamma e_i,\alpha_k}}{4(t-\bar x)(x-z_k)}\right)\frac{\ps{\gamma e_i,\gamma e_j}}{2(x-y)}\Psi_{i,j}(x,y)\indeps\right].
		\end{align*}
		Likewise the remaining two-fold integrals containing the $\bar x$ terms are the sum of $(P)$-class terms and
		\begin{align*}
			-&\Is^2\left[\partial_x\left(\left(\frac{\ps{\beta,\gamma e_i}\ps{\beta,\gamma e_j}}{2(t-\bar x)(t-y)}+\sum_{k=1}^{2N+M}\frac{\ps{\beta,\gamma e_i}\ps{\gamma e_j,\alpha_k}}{4(t-\bar x)(y-z_k)}+\frac{\ps{\beta,\gamma e_j}\ps{\gamma e_i,\alpha_k}}{4(t-y)(\bar x-z_k)}\right)\Psi_{i,j}(x,y)\right)\indeps\right]\\
			-&\Is^2\left[\partial_y\left(\left(\frac{\ps{\beta,\gamma e_i}^2}{2(t-x)(t-\bar x)}+\sum_{k=1}^{2N+M}\frac{\ps{\beta,\gamma e_i}\ps{\gamma e_i,\alpha_k}}{4(t-x)(\bar x-z_k)}+\frac{\ps{\beta,\gamma e_i}\ps{\gamma e_i,\alpha_k}}{4(t-\bar x)(x-z_k)}\right)\Psi_{i,j}(x,y)\right)\indeps\right]\\
			-&\Is\times\Ir\left[\partial_x\left(\frac{\ps{\beta,\gamma e_i}\ps{\gamma e_i,\gamma e_j}}{4(t-x)(\bar x-y)}\Psi_{i,j}(x,y)\right)\indeps\right]\\
			-&\Is\times\Ir\left[\partial_x\left(\left(\frac{\ps{\beta,\gamma e_i}\ps{3\beta+\gamma e_i,\gamma e_j}}{4(t-\bar x)(t-y)}+\frac{\ps{\beta,\gamma e_i}\ps{\gamma e_i,\gamma e_j}}{4(t-\bar x)(x-y)}\right)\Psi_{i,j}(x,y)\right)\indeps\right]
		\end{align*}
		and the extra three-fold integrals given by
		\begin{align*}
			&\Is^2\times\It\left[\left(\frac{\ps{\beta,\gamma e_i}\ps{\beta,\gamma e_j}}{2(t-\bar x)(t-y)}+\sum_{k=1}^{2N+M}\frac{\ps{\beta,\gamma e_i}\ps{\gamma e_j,\alpha_k}}{4(t-\bar x)(y-z_k)}+\frac{\ps{\beta,\gamma e_j}\ps{\gamma e_i,\alpha_k}}{4(t-y)(\bar x-z_k)}\right)\frac{\ps{\gamma e_i,\gamma e_f}}{2(x-z)}\Psi_{i,j,f}(x,y,z)\indeps\right]\\
			&\Is^2\times\It\left[\left(\frac{\ps{\beta,\gamma e_i}^2}{2(t-x)(t-\bar x)}+\sum_{k=1}^{2N+M}\frac{\ps{\beta,\gamma e_i}\ps{\gamma e_i,\alpha_k}}{4(t-x)(\bar x-z_k)}+\frac{\ps{\beta,\gamma e_i}\ps{\gamma e_i,\alpha_k}}{4(t-\bar x)(x-z_k)}\right)\frac{\ps{\gamma e_j,\gamma e_f}}{2(y-z)}\Psi_{i,j,f}(x,y,z)\indeps\right]\\
			&\Is\times\Ir\times\It\left[\frac{\ps{\beta,\gamma e_i}\ps{\gamma e_i,\gamma e_j}}{4(t-x)(\bar x-y)}\frac{\ps{\gamma e_i,\gamma e_f}}{2(x-z)}\Psi_{i,j,f}(x,y,z)\indeps\right]\\
			&\Is\times\Ir\times\It\left[\left(\frac{\ps{\beta,\gamma e_i}\ps{3\beta+\gamma e_i,\gamma e_j}}{4(t-\bar x)(t-y)}+\frac{\ps{\beta,\gamma e_i}\ps{\gamma e_i,\gamma e_j}}{4(t-\bar x)(x-y)}\right)\frac{\ps{\gamma e_i,\gamma e_f}}{2(x-z)}\Psi_{i,j,f}(x,y,z))\indeps\right].
		\end{align*}
		By combining these three-fold integrals with the ones obtained before we get that the three-fold integrals can be rewritten as $(P)$-class terms plus
		\begin{align*}
			&\Is^2\times\Ir\left[\partial_x\left(\left(\frac{\ps{\beta,\gamma e_i}\ps{\gamma e_j,\gamma e_f}}{4(t-\bar x)(y-z)}+\frac{\ps{\beta,\gamma e_j}\ps{\gamma e_i,\gamma e_f}}{4(t-y)(\bar x-w)}\right)\Psi_{i,j,f}(x,y,z)\right)\indeps\right]\\
			+&\Is^2\times\Ir\left[\partial_y\left(\frac{\ps{\beta,\gamma e_i}\ps{\gamma e_i,\gamma e_f}}{4(t-x)(\bar x-z)}\Psi_{i,j,f}(x,y,z)\right)\indeps\right]\\
			-&\Is^2\times\Ir^2\left[\partial_y\left(\frac{\ps{\gamma e_i,\gamma e_f}\ps{\gamma e_i,\gamma e_m}}{4(x-z)(\bar x-w)}\Psi_{i,j,f,m}(x,y,z,w)\right)\indeps\right].
		\end{align*}
		where we have used symmetry in $x,\bar x$ to remove singular terms. Recollecting terms and discarding $(P)$-class terms we get 
		\begin{eqs}
			&\mathfrak L_{-(1,1,1);\delta,\eps,\rho}(\bm\alpha)=\mathfrak L^{\mu_B=0}_{-(1,1,1);\delta,\eps,\rho}(\bm\alpha)+\Is\left[\partial_{x}\left(\frac{\ps{\beta,\gamma e_i}(1+\frac{\ps{\beta,\gamma e_i}}2)}{2(t-\bar x)^2}\Psi_{i}(x)\right)\indeps\right]\\
			&+\Is\left[\partial_{x}\left(\sum_{k=1}^{2N+M}\left(\frac{\ps{\beta,\gamma e_i}\ps{\gamma e_i,\alpha_k}}{4(t-x)(\bar x-z_k)}+\frac{\ps{\beta,\gamma e_i}}{2(t-\bar x)}\Theta_i(x)\right)\Psi_{i}(x)\right)\indeps\right]\\
			-&\Is^2\left[\partial_x\left(\left(\frac{\ps{\beta,\gamma e_i}\ps{\beta,\gamma e_j}}{2(t-\bar x)(t-y)}+\sum_{k=1}^{2N+M}\frac{\ps{\beta,\gamma e_j}\ps{\gamma e_i,\alpha_k}}{4(t-y)(\bar x-z_k)}\right)\Psi_{i,j}(x,y)\right)\indeps\right]\\
			-&\Is^2\left[\partial_x\left(\left(\frac{\ps{\beta,\gamma e_j}^2}{2(t-y)(t-\bar y)}+\sum_{k=1}^{2N+M}\frac{\ps{\beta,\gamma e_j}\ps{\gamma e_j,\alpha_k}}{4(t-y)(\bar y-z_k)}\right)\Psi_{i,j}(x,y)\right)\mathds 1_{y\in\Heps}\right]\\
			-&\Is\times\Ir\left[\partial_x\left(\frac{\ps{\beta,\gamma e_i}\ps{\gamma e_i,\gamma e_j}}{4(t-x)(\bar x-y)}\Psi_{i,j}(x,y)\right)\indeps\right]\\
			+&\Is^2\times\Ir\left[\partial_x\left(\frac{\ps{\beta,\gamma e_j}\ps{\gamma e_i,\gamma e_f}}{4(t-y)(\bar x-w)}\Psi_{i,j,f}(x,y,z)\right)\indeps\right]\\
			+&\Is^2\times\Ir\left[\partial_x\left(\frac{\ps{\beta,\gamma e_j}\ps{\gamma e_j,\gamma e_f}}{4(t-y)(\bar y-z)}\Psi_{i,j,f}(x,y,z)\right)\mathds 1_{y\in\Heps}\right]\\
			-&\Is^2\times\Ir^2\left[\partial_x\left(\frac{\ps{\gamma e_j,\gamma e_f}\ps{\gamma e_j,\gamma e_m}}{4(y-z)(\bar y-w)}\Psi_{i,j,f,m}(x,y,z,w)\right)\mathds 1_{y\in\Heps}\right].
		\end{eqs}

		Finally, the last part of the claim (expressing the $\L_{-1,1,1}$ descendant in terms of derivatives) follows from the very same arguments as before.
	\end{proof}

	\subsection{Singular vector at the level three}
	We have seen in the previous section that fully-degenerate fields, corresponding to Vertex Operators $V_\beta$ with $\beta=-\chi\omega_1$ for $\chi\in\{\gamma,\frac2\gamma\}$, give rise to singular vectors at the second level $\Psi_{-2,\beta}$. Actually from such fields we can further define singular vectors at the level three, leading to higher equations of motion.

    \subsubsection{Higher equations of motion at the level three}
	To do so we define the singular vector $\Psi_{-3,\beta}$ to be given by
	\begin{equation}
		\psi_{-3,\chi}\coloneqq\left(\Wb_{-3}+\left(\frac\chi3+\frac2\chi\right)\L_{-3}-\frac4\chi\L_{-1,2}-\frac8{\chi^3}\L_{-1,1,1}\right)V_{\beta}
	\end{equation}
	where $\beta=-\chi\omega_1$ with $\chi\in\{\gamma,\frac2\gamma\}$. More precisely it is defined as follows: 
	\begin{defi}
		Assume that $\beta=-\chi\omega_1$ with $\chi\in\{\gamma,\frac2\gamma\}$ and take $t\in\R$ as well as $(\beta,\bm\alpha)\in\mc A_{N,M+1}$. The singular vector $\psi_{-3,\chi}$ is defined within half-plane correlation functions via
		\begin{equation}
			\begin{split}
				&\ps{\psi_{-3,\chi}(t)\V}\coloneqq\\
				&\ps{\Wb_{-3}V_\beta(t)\V}+\left(\frac\chi3+\frac2\chi\right)\ps{\L_{-3}V_\beta(t)\V}-\frac4\chi\ps{\L_{-1,2}V_\beta(t)\V}-\frac8{\chi^3}\ps{\L_{-1,1,1}V_\beta(t)\V}.
			\end{split}
		\end{equation}
	\end{defi}
	
	Like before, the singular vector $\Psi_{-3,\chi}$ is actually not null in general but is seen to be equal to a sum of descendants of other primary fields. However here we will make the simplifying assumption that this singular vector gives rise to a null vector, which translates on an assumption on the cosmological constants:
	\begin{theorem}\label{thm:deg3}
		Assume that $\beta=-\chi\omega_1$ with $\chi\in\{\gamma,\frac2\gamma\}$, and assume that $\mu_{L,2}-\mu_{R,2}=0$. Then the singular vector at the third level $\psi_{-3,\chi}$ is equal to
		\begin{equation}
			\psi_{-3,\chi}=-\frac2{\chi^2}\Db_{-1,\chi}\psi_{-2,\chi}
		\end{equation}
		where $\Db_{-1,\chi}$ is of the form $\Db_{-1,\chi}\coloneqq a(\chi)\L_{-1}+b(\chi)\Wb_{-1}$ with $a,b$ such that
		\begin{equation}
			a\left(\chi\right)\L_{-1}^{-\chi\omega_1+\gamma e_1}+b\left(\chi\right)\Wb_{-1}^{-\chi\omega_1+\gamma e_1}=\left\{\begin{matrix}&\gamma e_1-\frac2\gamma\rho\qt{for}\chi=\frac2\gamma\\&\gamma e_1-\gamma e_2\qt{for}\chi=\gamma.\end{matrix}\right.
		\end{equation}
		Like before these equalities are to be understood in the sense of Theorems~\ref{thm:deg1} and~\ref{thm:deg2}.
	\end{theorem}
	\begin{remark}
		Though we have not conducted the computations when we no longer assume that $\mu_{L,2}-\mu_{R,2}=0$, the arguments that we have developed along the proof of Theorem~\ref{thm:deg3} suggest that in general we can write the singular vector at level three as a sum
		\[
		\psi_{-3,\chi}=\left(\mu_{L,2}-\mu_{R,2}\right)\Db_{-2}V_{\beta+\gamma e_2}+\left\{\begin{matrix}&\gamma^2\left(\gamma-\frac2\gamma\right)\left(\mu_{L,1}+\mu_{R,1}\right)\Db_{-1}V_{\beta+\gamma e_1}\qt{if}\chi=\frac2\gamma\\
			&-\frac4\gamma c_\gamma(\bm\mu)\Db_{-1}V_{\beta+2\gamma e_1}\qt{if}\chi=\gamma\end{matrix}\right.
		\]
		where $\Db_{-2}V_{\beta+\gamma e_2}$ is a descendant at the second order of $V_{\beta+\gamma e_2}$ for which
        \begin{align*}
            \Db_{-2}^{\beta+\gamma e_2}(u)&=\frac{8}{\chi^3}\ps{3\beta+\gamma e_2,u}+\frac4\chi\ps{Q+\beta+\gamma e_2,u}\\
            \Db_{-2}^{\beta+\gamma e_2}(u,v)&=\frac{8}{\chi^3}\left(3\ps{\beta,u}\ps{\beta,v}+\frac32(\ps{\beta,u}\ps{\gamma e_2,v}+\ps{\beta,v}\ps{\gamma e_2,u})+\ps{\gamma e_2,u}\ps{\gamma e_2,v}\right)-\frac4\chi\ps{u,v}.
        \end{align*}
        This can be rewritten as
		$\psi_{-3,\chi}=\tilde{\Db}_{-2,\chi}\psi_{-1,\chi}-\frac2{\chi^2}\Db_{-1,\chi}\psi_{-2,\chi}$
		which emphasizes that the singular vector at the level three does not give rise to a new primary field but is rather a linear combination of descendants of the singular vectors at level one and two.
	\end{remark}
	\begin{proof}
		Since at the regularized level we have the equality $$\ps{\left(\Wb_{-3}+\left(\frac\chi3+\frac2\chi\right)\L_{-3}-\frac4\chi\L_{-1,2}-\frac8{\chi^3}\L_{-1,1,1}\right)V_{\beta}(t)\V}_{\delta,\eps,\rho}=0$$ we only need to focus on the remainder terms that occur in the limit. And more specifically we need to show that if $\mu_{L,2}-\mu_{R,2}=0$ then 
		\begin{equation}\label{eq:to_prove_sing3}
			\begin{split}
				\lim\limits_{\delta,\eps,\rho\to0}\tilde{\mathfrak \Wb}_{-3;\delta,\eps,\rho}(\bm\alpha)+&\left(\frac\chi3+\frac2\chi\right)\tilde{\mathfrak L}_{-3;\delta,\eps,\rho}(\bm\alpha)\\
				-\frac4\chi&\tilde{\mathfrak L}_{-(1,2);\delta,\eps,\rho}(\bm\alpha) -\frac8{\chi^3}\tilde{\mathfrak L}_{-(1,1,1);\delta,\eps,\rho}(\bm\alpha)=-\ps{\psi_{-3,\chi}\V}.
			\end{split}
		\end{equation}
		For this purpose we will need the explicit expressions of the remainder terms from Equations~\eqref{eq:rem_Ln} and~\eqref{eq:rem_Wn}, as well as from Lemmas~\ref{lemma:def_L12} and~\ref{lemma:def_L111}. Based on these expressions we can write the left-hand side in Equation~\eqref{eq:to_prove_sing3} as the sum $I_1+I_2+I_3+I_4$,
		where in $I_k$ we have gathered all the $k$-fold integrals that appear in these statements. 
		
		Before entering into the computations, let us stress that since we have made the assumption that $\mu_{L,2}-\mu_{R,2}=0$ and because $\ps{\beta,\gamma e_i}<0$ for $i=1$ remainder terms of the form 
		\[
		      \Is^a\times\Ir^b\left[\partial_{x_1} \left(F_{i_1,\cdots,i_{a+b}}(x_1,\cdots,x_{a+b})\Psi_{\bm i}(x_1,\cdots,x_{a+b})\right)\right]
		\]
		where $F_{i_1,\cdots,i_{a+b}}$ is regular at $x_i=t$ and $x_i=x_j$ for all $i,j\leq a$ vanish in the limit along the same lines as in the proofs of Theorems~\ref{thm:deg1} and~\ref{thm:deg2}. As a consequence we will not take such terms into account later on. The proof will be divided in two parts: first of all we will assume that $\mu_{B,1}=\mu_{B,2}=0$ and then proceed to the general case. To compute the limit of these remainder terms we will rely on Lemmas~\ref{lemma:limtx} and~\ref{lemma:limty} below,      where for the sake of simplicity we will make use of the notations
		\begin{eqs}
			&c_\gamma^1(\bm\mu)\coloneqq -\mu_{B,1}\sin\left(\pi\frac{\gamma^2}{2}\right)\frac{\Gamma\left(\frac{\gamma^2}{2}\right)\Gamma\left(1-\gamma^2\right)}{\Gamma\left(1-\frac{\gamma^2}{2}\right)}\mathds1_{\gamma<1}\qt{and}\\
			&c_\gamma^2(\bm\mu)\coloneqq \left(\mu_{L,1}^2-2\mu_{L,1}\mu_{R,1}\cos\left(\pi\frac{\gamma^2}{2}\right)+\mu_{R,1}^2\right)\frac{\Gamma\left(\frac{\gamma^2}{2}\right)\Gamma\left(1-\gamma^2\right)}{\Gamma\left(1-\frac{\gamma^2}{2}\right)}\mathds1_{\gamma<1}
		\end{eqs}
		so that $c_\gamma(\bm\mu)\mathds1_{\gamma<1}=c_\gamma^1(\bm\mu)+c_\gamma^2(\bm\mu)$.
		
		\subsubsection*{The case $\mu_{B,1}=\mu_{B,2}=0$}
		Thus to start with let us assume that $\mu_{B,1}=\mu_{B,2}=0$, so that we can discard the integrals containing $\indeps$ terms. Then we have, up to $(P)$-class terms,
		\begin{align*}
			I_1=&\Is\left[\partial_x\left(\frac{h_2(e_i)(q+2\omega_{\hat{i}}(\beta))+\left(\frac\chi3+\frac2\chi\right)-\frac4\chi(1+\frac{\ps{\beta,\gamma e_i}}{2})-\frac{4}{\chi^3}\ps{\beta,\gamma e_i}(1+\frac{\ps{\beta,\gamma e_i}}{2})}{(x-t)^{2}}\Psi_i(x)\right)\right]\\
			+&\Is\left[\partial_x\left(\left(\sum_{k=1}^{2N+M}\frac{-2h_2(e_i)\omega_{\hat i}(\alpha_k)-\frac2\chi\ps{\beta,\alpha_k}-\frac2{\chi^3}\ps{\beta,\gamma e_i}\ps{3\beta+\gamma e_i,\alpha_k}}{(t-x)(t-z_k)}\right)\Psi_i(x)\right)\right].
		\end{align*}
		Likewise we can write that $I_2=A+B$ with
		\begin{eqs}\label{eq:defA}
			&A\coloneqq\Is^2\Big[\partial_x\Big(\Big(\frac{2\gamma h_2(e_i)\delta_{i\neq j}}{(t-x)(x-y)}+\frac{\frac2\chi\ps{\beta,\gamma e_j}+\frac4{\chi^3}\ps{\beta,\gamma e_i}\ps{\beta,\gamma e_j}}{(t-x)(t-y)}+\frac4{\chi^3}\frac{\ps{\beta,\gamma e_j}(1+\frac{\ps{\beta,\gamma e_j}}{2})}{(t-y)^2}\Big)\Psi_{i,j}(x,y)\Big)\Big]\\
		\end{eqs}
		\begin{eqs}\label{eq:defB}
			&B\coloneqq \frac8{\chi^3}\Is^2\left[\partial_x\left(\left(\sum_{k=1}^{2N+M}\frac{\ps{\beta,\gamma e_i}\ps{{\gamma e_j,\alpha_k}}}{4(t-x)(y-z_k)}\right)\Psi_{i,j}(x,y)\right)\right]\\
			&+\frac8{\chi^3}\Is^2\left[\partial_x\left(\left(\sum_{k=1}^{2N+M}\frac{\ps{\beta,\gamma e_j}\ps{3\beta+\gamma e_j,\alpha_k}}{4(t-y)(t-z_k)}+\frac{\ps{\beta,\gamma e_j}\ps{\gamma e_i,\alpha_k}}{4(t-y)(x-z_k)}\right)\Psi_{i,j}(x,y)\right)\right]\\
			&+\Is\times\Ir\left[\partial_x\left(\left(\frac{2\gamma h_2(e_i)\delta_{i\neq j}}{(t-x)(x-y)}+\frac2\chi\frac{\ps{\beta,\gamma e_j}}{(t-x)(t-y)}+\frac{8}{\chi^3}\frac{\ps{\beta,\gamma e_i}\ps{3\beta+\gamma e_i,\gamma e_j}}{4(t-x)(t-y)}\right)\Psi_{i,j}(x,y)\right)\right].
		\end{eqs}
		Finally the three- and four-fold integrals are given by
		\begin{eqs}\label{eq:defI3}
			&I_3=-\frac8{\chi^3}\Is^3\left[\partial_{x}\left(\left(\frac{\ps{\beta,\gamma e_j}\ps{\beta,\gamma e_f}}{4(t-y)(t-z)}+\sum_{k=1}^{2N+M}\frac{\ps{\beta,\gamma e_j}\ps{\gamma e_f,\alpha_k}}{4(t-y)(z-z_k)}\right)\Psi_{i,j,f}(x,y,z)\right)\right]\\
			&-\frac8{\chi^3}\Is^2\times\Ir\left[\partial_{x}\left(\left(\frac{\ps{\beta,\gamma e_j}\ps{3\beta+\gamma e_j,\gamma e_f}}{4(t-y)(t-z)}+\frac{\ps{\beta,\gamma e_j}\ps{\gamma e_i,\gamma e_f}}{4(t-y)(x-z)}\right)\Psi_{i,j,f}(x,y,z)\right)\right]\\
			&-\frac8{\chi^3}\Is^2\times\Ir\left[\partial_{x}\left(\frac{\ps{\beta,\gamma e_i}\ps{\gamma e_j,\gamma e_f}}{4(t-x)(y-z)}\Psi_{i,j,f}(x,y,z)\right)\right]\qt{and}
		\end{eqs}
		\begin{equation}\label{eq:defI4}
			I_4=\frac8{\chi^3}\Is^3\times\Ir\left[\partial_{x}\left(\frac{\ps{\beta,\gamma e_j}\ps{\gamma e_f,\gamma e_m}}{4(t-y)(z-w)}\Psi_{i,j,f,m}(x,y,z,w)\right)\right].
		\end{equation}
		We will treat successively these four types of remainder terms to see that they give rise to the right-hand side in Equation~\eqref{eq:to_prove_sing3}. 
		
		\subsubsection*{One-fold integrals}
		For this let us start by looking at the one-fold integrals, and first of all to the ones do not contain any $\bar x$ term. To treat such integrals we begin with terms that contain a singularity $\frac{1}{(t-x)^2}$. The corresponding quantity is then given by
		\begin{align*}
			\Is\left[\partial_x\left(\left(\frac{h_2(e_i)(q+2\omega_{\hat{i}}(\beta))+\left(\frac\chi3+\frac2\chi\right)-\frac4\chi(1+\frac{\ps{\beta,\gamma e_i}}{2})-\frac{4}{\chi^3}\ps{\beta,\gamma e_i}(1+\frac{\ps{\beta,\gamma e_i}}{2})}{(x-t)^{2}}\right)\Psi_i(x)\right)\right].
		\end{align*}
		Using the value of $\beta=-\chi\omega_1$, the numerator is found, somewhat amazingly, to be equal to $\delta_{i,1}(2\gamma-\frac4\chi)(1-\frac{\gamma}{\chi})$ which vanishes for $\chi\in\{\gamma,\frac2\gamma\}$. The subsequent term in one-fold integrals, corresponding to singularities of the form $\frac{1}{t-x}$, is given (up to negligible terms) by
		\begin{align*}
			\Is\left[\partial_x\left(\left(\sum_{k=1}^{2N+M}\frac{-2h_2(e_i)\omega_{\hat i}(\alpha_k)-\frac2\chi\ps{\beta,\alpha_k}-\frac2{\chi^3}\ps{\beta,\gamma e_i}\ps{3\beta+\gamma e_i,\alpha_k}}{(t-x)(t-z_k)}\right)\Psi_i(x)\right)\right].
		\end{align*}
		To start with for $i=2$ this coefficient is seen to be zero for both $\chi=\gamma$ and $\chi=\frac2\gamma$. Likewise for $\chi=\gamma$ and $i=1$ miraculous simplifications occur and this coefficient is actually seen to vanish.  If however we assume that $\chi=\frac2\gamma$ and $i=1$ then the coefficient is equal to $$2\omega_2(\alpha_k)+2\omega_1(\alpha_k)+\frac{\gamma^3}2\ps{-\frac6\gamma\omega_1+\gamma e_1,\alpha_k}=\frac{\gamma^2}2(\gamma-\frac2\gamma)\ps{\gamma e_1-\frac2\gamma\rho,\alpha_k},$$ where we have used that $\rho=3\omega_1-e_1$. As a consequence and in agreement with Lemma~\ref{lemma:limtx} the corresponding remainder term converges towards
		\begin{equation}
			r_1\coloneqq-\gamma^2\left(\gamma-\frac2\gamma\right)(\mu_{L,1}+\mu_{R,1})\sum_{k=1}^{2N+M}\frac{\ps{\gamma e_1-\frac2\gamma\rho,\alpha_k}}{2(z_k-t)} \ps{V_{\gamma e_1-\frac2\gamma\omega_1}(t)\V}.
		\end{equation}
		To summarize we have that $I_1=\text{$(P)$-class terms}+r_1\mathds1_{\chi=\frac2\gamma}$.

		\subsubsection*{Two- and three-fold integrals: the largest singularities}
		Let us now turn to the two- and three-fold integrals, in which case the remainder terms are given by Equations~\eqref{eq:defA},~\eqref{eq:defB} and~\eqref{eq:defI3}. And to start with like before we first consider the largest singularities, that is we gather terms with integrands of the form $\frac{1}{(t-u)(t-v)}$ where both $u$ and $v$ are in $\Is$: this corresponds to $A$ defined in Equation~\eqref{eq:defA} and the very first term in Equation~\eqref{eq:defI3}.
		
		These terms are a priori divergent in the $\delta,\eps,\rho\to0$ limit. However miraculous simplifications allow to show that there are actually asymptotically given by
		\begin{equation}\label{eq:IIIII}
			\begin{split}
				&A-\frac{8}{\chi^3}\Is^3\left[\partial_x\left(\frac{\ps{\beta,\gamma e_j}\ps{\beta,\gamma e_f}}{4(t-y)(t-z)}\Psi_{i,j,f}(x,y,z)\right)\right]=\text{$(P)$-class terms}+r_2\mathds 1_{\chi=\frac2\gamma}\\
				&-\frac4\chi\Is^2\Big[\partial_x\Big(\sum_{k=1}^{2N+M}\Big(\frac{\ps{\gamma e_i,\alpha_k}}{2(t-x)(z_k-y)}\Psi_{1,i}(x,y)-\frac{\ps{\gamma e_1,\alpha_k}}{2(t-y)(z_k-y)}\Psi_{i,1}(x,y)\Big)\Big)\Big]\\
				&+\frac4\chi\Is^2\times\Ir\left[\partial_x\left(\frac{\ps{\gamma e_i,\gamma e_j}}{2(t-x)(z-y)}\Psi_{1,i,j}(x,y,z)-\frac{\ps{\gamma e_1,\gamma e_j}}{2(t-y)(z-y)}\Psi_{i,1,j}(x,y,z)\right)\right]
			\end{split}
		\end{equation}
		with $r_2$ described in Lemma~\ref{lemma:sing3} below, and where this limit is justified. 
		
		\subsubsection*{The remaining two-fold integrals}
		As a consequence we are only left with integrals that have a singularity of order $1$ at $t=x$ or $t=y$. Let us thus consider the remaining two-fold integrals from Equations~\eqref{eq:defB} and~\eqref{eq:IIIII}: 
		\begin{align*}
			&-\frac4\chi\Is^2\Big[\partial_x\Big(\sum_{k=1}^{2N+M}\Big(\frac{\ps{\gamma e_i,\alpha_k}}{2(t-x)(z_k-y)}\Psi_{1,i}(x,y)-\frac{\ps{\gamma e_1,\alpha_k}}{2(t-y)(z_k-y)}\Psi_{i,1}(x,y)\Big)\Big)\Big]\\
			&-\frac8{\chi^3}\Is^2\left[\partial_x\left(\left(\sum_{k=1}^{2N+M}\frac{\ps{\beta,\gamma e_i}\ps{{\gamma e_j,\alpha_k}}}{4(t-x)(z_k-y)}\right)\Psi_{i,j}(x,y)\right)\right]\\
			&-\frac8{\chi^3}\Is^2\left[\partial_x\left(\left(\sum_{k=1}^{2N+M}\frac{\ps{\beta,\gamma e_j}\ps{3\beta+\gamma e_j,\alpha_k}}{4(t-y)(z_k-y)}+\frac{\ps{\beta,\gamma e_j}\ps{{\gamma e_i,\alpha_k}}}{4(t-y)(z_k-x)}\right)\Psi_{i,j}(x,y)\right)\right]\\
			&+\Is\times\Ir\left[\partial_x\left(\left(\frac{2\gamma h_2(e_i)\delta_{i\neq j}}{(t-x)(x-y)}+\frac2\chi\frac{\ps{\beta,\gamma e_j}}{(t-x)(t-y)}+\frac{8}{\chi^3}\frac{\ps{\beta,\gamma e_i}\ps{3\beta+\gamma e_i,\gamma e_j}}{4(t-x)(t-y)}\right)\Psi_{i,j}(x,y)\right)\right].
		\end{align*}
		
		Let us first treat the three integrals ranging over $\Is^2$, and to begin with let us assume that $\chi=\frac2\gamma$. Then in agreement with Lemmas~\ref{lemma:limtx} and~\ref{lemma:limty} the integral over $\Is^2$ will contribute to the limit by
		\begin{eqs}
			r_3\coloneqq-\gamma^2\left(\gamma-\frac2\gamma\right)(\mu_{L,1}+\mu_{R,1})\Is\Big[\sum_{k=1}^{2N+M}\frac{\ps{\gamma e_i,\alpha_k}}{2(y-z_k)}\ps{V_{\gamma e_i}(y)V_{\beta+\gamma e_1}(t)\V}\Big)\Big].
		\end{eqs}
		If we take $\chi=\gamma$ then we see that such integrals simplify to
		\begin{align*}
			&\frac4\gamma\Is^2\Big[\partial_x\Big(\Big(\sum_{k=1}^{2N+M}\frac{\ps{3\beta+2\gamma e_1,\alpha_k}}{2(t-y)(z_k-y)}+\frac{\ps{\gamma e_i,\alpha_k}}{2(t-y)(z_k-x)}\Big)\Psi_{i,1}(x,y)\Big)\Big].
		\end{align*}
		For $i=1$, we can use the equality $3\beta+2\gamma e_1+\gamma e_2=0$ to infer that the corresponding term vanishes in the limit. If we now assume that $i=1$ then thanks to Lemma~\ref{lemma:limty} below we have that the remainder term converges toward (recall that $c_\gamma^2(\bm\mu)=0$ for $\gamma >1$)
		\begin{equation}
			s_1\coloneqq\frac4\gamma c_\gamma^2(\bm\mu)\sum_{k=1}^{2N+M}\frac{\ps{3\beta+3\gamma e_1,\alpha_k}}{2(z_k-t)}\ps{V_{\beta+2\gamma e_1}(t)\V}.
		\end{equation}
		
		For the integrals over $\Is\times\Ir$ like before we rely on the theory of voodoo calculus: in the case where $(i,j)=(2,1)$ simplifications occur and we see that the corresponding $\Is\times\Ir$ integral is given, up to a term that vanishes in the limit, by
		\begin{align*}
			\Is\times\Ir\left[\partial_x\left(\frac{2\gamma}{(t-y)(x-y)}\Psi_{2,1}(x,y)\right)\right]
		\end{align*} 
		which is seen to converge to zero in the limit thanks to the assumption that $\mu_{L,2}-\mu_{R,2}=0$. For $i=j=2$ we see that the integrand is identically zero, so we are left with the cases where $i=1$. If we take $j=2$ then we get
		\begin{align*}
			\Is\times\Ir\left[\partial_x\left(\frac{2\gamma}{(t-y)(y-x)}\Psi_{1,2}(x,y)\right)\right]
		\end{align*}
		which converges to $0$ since $\ps{\beta,\gamma e_1}<0$. When $j=1$ and $\chi=\gamma$ we again rely on our faith in the beauty of CFT, thanks to which the integrand is actually zero.  
		Finally when we take $j=1$ and $\chi=\frac2\gamma$ the corresponding term is given by
		\begin{align*}
			\Is\times\Ir\left[\partial_x\left(\left(\frac{-2\gamma \delta_{j,2}}{(t-x)(x-y)}-\frac{2\gamma\delta_{j,1}}{(t-x)(t-y)}-\frac{\gamma^3 \ps{3\beta+\gamma e_1,\gamma e_j}}{2(t-x)(t-y)}\right)\Psi_{1,j}(x,y)\right)\right]
		\end{align*}
		which in the limit (via Lemma~\ref{lemma:limtx}) will yield a term
		\begin{equation}
			r_4\coloneqq \gamma^2(\frac2\gamma-\gamma) (\mu_{L,1}+\mu_{R,1})\Ir\left[\left(\frac{\ps{\gamma e_1-\frac2\gamma\rho,\gamma e_j}}{2(t-y)}\right)\Psi_{1,j}(t,y)\right]
		\end{equation}
		
		To wrap up for the two-fold integrals, we see that we arrive to the following conclusion:
		\begin{eqs}
			&A+B-\frac{8}{\chi^3}\Is^3\left[\partial_x\left(\frac{\ps{\beta,\gamma e_j}\ps{\beta,\gamma e_f}}{4(t-y)(t-z)}\Psi_{i,j,f}(x,y,z)\right)\right]\\
			&=\text{$(P)$-class terms }+(r_2+r_3+r_4)\mathds1_{\chi=\frac2\gamma}+s_1\mathds1_{\chi=\gamma}.
		\end{eqs}
		
		\subsubsection*{The three- and four-fold integrals}
		Therefore it only remains to treat the remaining three- and four-fold integrals, coming either from Equations~\eqref{eq:defI3} and~\eqref{eq:defI4}, and given by:
		\begin{align*}
			&-\frac8{\chi^3}\Is^3\left[\partial_{x}\left(\left(\sum_{k=1}^{2N+M}\frac{\ps{\beta,\gamma e_j}\ps{\gamma e_f,\alpha_k}}{4(t-y)(z-z_k)}\right)\Psi_{i,j,f}(x,y,z)\right)\right]\\
			&-\frac8{\chi^3}\Is^2\times\Ir\left[\partial_{x}\left(\left(\frac{\ps{\beta,\gamma e_j}\ps{3\beta+\gamma e_j,\gamma e_f}}{4(t-y)(t-z)}+\frac{\ps{\beta,\gamma e_j}\ps{\gamma e_i,\gamma e_f}}{4(t-y)(x-z)}\right)\Psi_{i,j,f}(x,y,z)\right)\right]\\
			&-\frac8{\chi^3}\Is^2\times\Ir\left[\partial_{x}\left(\frac{\ps{\beta,\gamma e_i}\ps{\gamma e_j,\gamma e_f}}{4(t-x)(y-z)}\Psi_{i,j,f}(x,y,z)\right)\right]\\
			&+\frac4\chi\Is^2\times\Ir\left[\partial_x\left(\frac{\ps{\gamma e_i,\gamma e_j}}{2(t-x)(z-y)}\Psi_{1,i,j}(x,y,z)-\frac{\ps{\gamma e_1,\gamma e_j}}{2(t-y)(z-y)}\Psi_{i,1,j}(x,y,z)\right)\right]\\
			&+\frac8{\chi^3}\Is^3\times\Ir\left[\partial_{x}\left(\frac{\ps{\beta,\gamma e_j}\ps{\gamma e_f,\gamma e_m}}{4(t-y)(z-w)}\Psi_{i,j,f,m}(x,y,z,w)\right)\right].
		\end{align*}
		The integral over $\Is^3$ is treated using Lemma~\ref{lemma:limty} below and yields
		\begin{equation}
			r_2\coloneqq\frac4\gamma c_\gamma^2(\bm\mu)\Is\left[\sum_{k=1}^{2N+M}\frac{\ps{\gamma e_f,\alpha_k}}{2(z-z_k)}\ps{V_{\beta+2\gamma e_1}(t)V_{\gamma e_f}(z)\V}\right].
		\end{equation}
		As for the other integrals, by explicit computations they simplify to
		\begin{align*}
			&\frac{2}{\chi^2}(\gamma-\chi)\Is^2\times\Ir\left[\partial_x\left(\frac{\ps{\gamma e_j,\gamma e_f}}{(t-x)(y-z)}\Psi_{1,j,f}(x,y,z)\right)\right]\\
			&+\frac{2\gamma}{\chi^2}\Is^2\times\Ir\left[\partial_x\left(\left(\frac{\ps{3\beta+\gamma e_1,\gamma e_f}}{(t-y)(t-z)}+\frac{\ps{\gamma e_i,\gamma e_f}}{(t-y)(x-z)}+\frac\chi\gamma\frac{\ps{\gamma e_1,\gamma e_f}}{(t-y)(y-z)}\right)\Psi_{i,1,f}(x,y,z)\right)\right]\\
			&-\frac{2\gamma}{\chi^2}\Is^3\times\Ir\left[\partial_x\left(\frac{\ps{\gamma e_f,\gamma e_m}}{(t-y)(z-w)}\Psi_{i,1,f,m}(x,y,z,w)\right)\right].
		\end{align*}
		According to Lemma~\ref{lemma:limtx} the latter will give a remainder term for $\chi=\frac2\gamma$:
		\begin{equation}
			r_5\coloneqq \gamma^2(\gamma-\frac2\gamma)(\mu_{L,1}+\mu_{R,1})\Is\times\Ir\left[\frac{\ps{\gamma e_j,\gamma e_f}}{2(y-z)}\Psi_{1,j,f}(t,y,z)\right],
		\end{equation}
		while for $\chi=\gamma$ the corresponding remainder term is
		\begin{eqs}
			s_3\coloneqq &\frac4\gamma c_\gamma^2(\bm\mu)\Ir\left[\frac{\ps{3\beta+3\gamma e_1,\gamma e_f}}{2(t-z)}\ps{V_{\beta+2\gamma e_1}(t)V_{\gamma e_f}(z)\V}\right]\\
			&-\frac4\gamma c_\gamma^2(\bm\mu)\Is\times\Ir\left[\frac{\ps{\gamma e_f,\gamma e_m}}{2(z-w)}\Psi_{1,1,f,m}(t,t,z,w)\right].
		\end{eqs}
		
		\subsubsection*{Recollecting terms}
		Finally we have treated all the remainder terms that appear in the expression of the singular vectors. After this process the only remaining terms are
		\[
		(r_1+r_2+r_3+r_4+r_5)\mathds{1}_{\chi=\frac2\gamma}+(s_1+s_2+s_3)\mathds{1}_{\substack{\chi=\gamma\\\gamma<1}}.
		\]
		
		To start with, in the case where $\chi=\gamma$ we obtain a remainder term given by $\frac4\gamma c_\gamma^2(\bm\mu)\times$
		\begin{align*}
			&\sum_{k=1}^{2N+M}\frac{\ps{3\beta+3\gamma e_1,\alpha_k}}{2(z_k-t)}\ps{V_{2\gamma e_1-\gamma\omega_1}(t)\V}-\Ir\left[\frac{\ps{3\beta+3\gamma e_1,\gamma e_f}}{2(x-t)}\ps{V_{\gamma e_i}(x)V_{\beta+2\gamma e_1}(t)\V}\right]\\
			&+\Is\left[\sum_{k=1}^{2N+M}\frac{\ps{\alpha_k,\gamma e_i}}{2(x-z_k)}\ps{V_{\gamma e_i}(x)V_{\beta+2\gamma e_1}(t)\V}\right]\\
			&-\Is\times\Ir\left[\frac{\ps{\gamma e_i,\gamma e_j}}{2(x-y)}\ps{V_{\gamma e_i}(x)V_{\gamma e_j}(y)V_{\beta+2\gamma e_1}(t)\V}\right].
		\end{align*}
		The above is nothing but the descendant $(a\L_{-1}+b\Wb_{-1})V_{\beta+2\gamma e_1}$ where $a,b$ satisfy
		\[
		a\L_{-1}^{\beta+2\gamma e_1}+b\Wb_{-1}^{\beta+2\gamma e_1}=3(\gamma e_1-\gamma \omega_1)=-3\gamma h_2.
		\]
		Therefore we infer that when $\mu_{B,1}=\mu_{B,2}=0$:
		\begin{eqs}
			\psi_{-3,\gamma}=-\frac4\gamma c_\gamma^2(\bm\mu)\Db_{-1,\gamma}V_{\beta+2\gamma e_1}.
		\end{eqs}
		
		When $\chi=\frac2\gamma$ the remaining terms coming from the $r_k$ are given by $\gamma^2(\frac2\gamma-\gamma)(\mu_{L,1}+\mu_{R,1})\times$
		\begin{align*}
			& \sum_{k=1}^{2N+M}\frac{\ps{\gamma e_1-\frac2\gamma\rho,\alpha_k}}{2(z_k-t)} \ps{V_{\gamma e_1-\frac2\gamma\omega_1}(t)\V}-\Ir\left[\frac{\ps{\gamma e_1-\frac2\gamma\rho,\gamma e_j}}{2(y-t)}\ps{V_{\gamma e_j}(y)V_{\gamma e_1-\frac2\gamma\omega_1}(t)\V}\right]\\
			&+\Is\Big[\frac{1}{y-t}\Psi_{1,2}(t,y)\Big]+\Is\Big[\sum_{k=1}^{2N+M}\frac{\ps{\gamma e_i,\alpha_k}}{2(y-z_k)}\ps{V_{\gamma e_j}(y)V_{\gamma e_1-\frac2\gamma\omega_1}(t)\V}\Big]\\
			&-\Is\times\Ir\Big[\frac{\ps{\gamma e_i,\gamma e_j}}{2(y-z)}\ps{V_{\gamma e_i}(y)V_{\gamma e_j}(z)V_{\gamma e_1-\frac2\gamma\omega_1}(t)\V}\Big].
		\end{align*}
        Like in the proof of Theorem~\ref{thm:deg2}, since $\ps{\gamma e_1-\frac2\gamma\omega_1,\gamma e_i}<0$ for $i=1,2$ we can rewrite the latter using integration by parts, and get
        \begin{align*}
			& \sum_{k=1}^{2N+M}\frac{\ps{\gamma e_1-\frac2\gamma\rho,\alpha_k}}{2(z_k-t)} \ps{V_{\gamma e_1-\frac2\gamma\omega_1}(t)\V}-\Ir\left[\frac{\ps{\gamma e_1-\frac2\gamma\rho,\gamma e_j}}{2(y-t)}\ps{V_{\gamma e_j}(y)V_{\gamma e_1-\frac2\gamma\omega_1}(t)\V}\right]\\
			&-\Is\Big[\frac{\ps{\gamma e_1-\frac2\gamma\omega_1,\gamma e_i}-2\delta_{j=2}}{2(y-t)}\ps{V_{\gamma e_j}(y)V_{\gamma e_1-\frac2\gamma\omega_1}(t)\V}\Big].
		\end{align*}
        We recognize here the definition of the descendant $(a\L_{-1}+b\Wb_{-1})V_{\gamma e_1-\frac2\gamma\omega_1}$ where $a$ and $b$ are chosen in such a way that
		\[
		a\L_{-1}^{\gamma e_1-\frac2\gamma\omega_1}+b\Wb_{-1}^{\gamma e_1-\frac2\gamma\omega_1}=\gamma e_1-\frac2\gamma\rho.
		\]
		As a consequence we arrive to the equality
		\begin{eqs}
			\psi_{-3,\frac2\gamma}&=\gamma^2\left(\gamma-\frac2\gamma\right)\left(\mu_{L,1}+\mu_{R,1}\right)\Db_{-1,\frac2\gamma}V_{\beta+\gamma e_1}.
		\end{eqs}
		By remembering the expression of the singular at the level two from Theorem~\ref{thm:deg2} we recover the desired result in this case.
		
		\subsubsection*{The general case}
		Let us now no longer make the assumption that $\mu_{B,1}=\mu_{B,2}=0$. Then the remainder terms feature extra terms, given by
		\begin{eqs}
			&\mathfrak I\coloneqq-\frac{8}{\chi^3}\Is\left[\partial_{x}\left(\frac{\ps{\beta,\gamma e_i}(1+\frac{\ps{\beta,\gamma e_i}}2)}{2(t-\bar x)^2}\Psi_{i}(x)\right)\indeps\right]\\
			&-\frac{8}{\chi^3}\Is\left[\partial_{x}\left(\sum_{k=1}^{2N+M}\left(\frac{\ps{\beta,\gamma e_i}\ps{\gamma e_i,\alpha_k}}{4(t-x)(\bar x-z_k)}+\frac{\ps{\beta,\gamma e_i}}{2(t-\bar x)}\Theta_i(x)\right)\Psi_{i}(x)\right)\indeps\right]\\
			&+\frac{8}{\chi^3}\Is^2\left[\partial_x\left(\left(\frac{\ps{\beta,\gamma e_i}\ps{\beta,\gamma e_j}}{2(t-\bar x)(t-y)}+\sum_{k=1}^{2N+M}\frac{\ps{\beta,\gamma e_j}\ps{\gamma e_i,\alpha_k}}{4(t-y)(\bar x-z_k)}\right)\Psi_{i,j}(x,y)\right)\indeps\right]\\
			&+\frac{8}{\chi^3}\Is^2\left[\partial_x\left(\left(\frac{\ps{\beta,\gamma e_j}^2}{2(t-y)(t-\bar y)}+\sum_{k=1}^{2N+M}\frac{\ps{\beta,\gamma e_j}\ps{\gamma e_j,\alpha_k}}{4(t-y)(\bar y-z_k)}\right)\Psi_{i,j}(x,y)\right)\mathds 1_{y\in\Heps}\right]\\
			&+\frac{8}{\chi^3}\Is\times\Ir\left[\partial_x\left(\frac{\ps{\beta,\gamma e_i}\ps{\gamma e_i,\gamma e_j}}{4(t-x)(\bar x-y)}\Psi_{i,j}(x,y)\right)\indeps\right].
		\end{eqs}
		In the above we have removed terms that vanish in the limit. Then we can rewrite the latter via the same reasoning as in Lemma~\ref{lemma:sing3}: up to $(P)$-class terms
		\begin{align*}
			&-\frac{8}{\chi^3}\Is\left[\partial_{x}\left(\frac{\ps{\beta,\gamma e_i}(1+\frac{\ps{\beta,\gamma e_i}}2)}{2(t-\bar x)^2}\Psi_{i}(x)\right)\indeps\right]\\
			&+\frac{8}{\chi^3}\Is^2\left[\partial_x\left(\left(\frac{\ps{\beta,\gamma e_i}\ps{\beta,\gamma e_j}}{2(t-\bar x)(t-y)}\indeps+\frac{\ps{\beta,\gamma e_j}^2}{2(t-y)(t-\bar y)}\mathds 1_{y\in\Heps}\right)\Psi_{i,j}(x,y)\right)\right]\\
			&=\frac4\chi\Is\Big[\partial_x\left(\sum_{k=1}^{2N+M}\left(\frac{\ps{\gamma e_1,\alpha_k}}{2(t-x)(z_k-\bar x)}\Psi_{1}(x)-\frac{\ps{\gamma e_1,\alpha_k}}{2(t-\bar x)(z_k-\bar x)}\Psi_{1}(x)\right)\right)\Big]\\
			&-\frac4\chi\Is\times\Ir\left[\partial_x\left(\frac{\ps{\gamma e_1,\gamma e_i}}{2(t-x)(y-\bar x)}\Psi_{1,i}(x,y)-\frac{\ps{\gamma e_1,\gamma e_i}}{2(t-\bar x)(y- \bar x)}\Psi_{1,i}(x,y)\right)\right].
		\end{align*}
		As a consequence we obtain 
		\begin{align*}
			&\mathfrak I=\frac{8\gamma }{\chi^2}\Is\left[\partial_{x}\left(\frac{1}{2(t-\bar x)}\left(\sum_{k=1}^{2N+M}\frac{\ps{\gamma e_i,\alpha_k}}{4(t-\bar x)(\bar x-z_k)}+\Theta_i(x)\right)\Psi_{i}(x)\right)\indeps\right]\\
			&+\frac{8}{\chi^3}\Is^2\left[\partial_x\left(\left(\sum_{k=1}^{2N+M}\frac{\ps{\beta,\gamma e_j}\ps{\gamma e_i,\alpha_k}}{4(t-y)(\bar x-z_k)}\right)\Psi_{i,j}(x,y)\right)\indeps\right]\\
			&+\frac{8}{\chi^3}\Is^2\left[\partial_x\left(\left(\sum_{k=1}^{2N+M}\frac{\ps{\beta,\gamma e_j}\ps{\gamma e_j,\alpha_k}}{4(t-y)(\bar y-z_k)}\right)\Psi_{i,j}(x,y)\right)\mathds 1_{y\in\Heps}\right]\\
			&-\frac{4}{\chi^2}(\gamma-\chi)\Is\times\Ir\left[\partial_x\left(\frac{\ps{\gamma e_1,\gamma e_i}}{2(t-x)(\bar x-y)}\Psi_{1,i}(x,y)\right)\indeps\right]\\
			&+\frac4\chi\Is\times\Ir\left[\partial_x\left(\frac{\ps{\gamma e_1,\gamma e_i}}{2(t-\bar x)(y- \bar x)}\Psi_{1,i}(x,y)\right)\right].
		\end{align*}
		The limit is computed in Lemma~\ref{lemma:limty}: for $\chi=\frac2\gamma$ it vanishes while for $\chi=\gamma$ we get
		\begin{align*}
			s_4\coloneqq&\frac4\gamma c_\gamma^1(\bm\mu) \left(\sum_{k=1}^{2N+M}\frac{\ps{\gamma e_i,\alpha_k}}{2(z_k-t)}\ps{V_{\beta+2\gamma e_1}(t)\V}-\Theta_i(t)\right)\\
			-&\frac4\gamma c_\gamma^1(\bm\mu)\Ir\left[\frac{\ps{\gamma e_1,\gamma e_i}}{2(y-t)}\ps{V_{\beta+2\gamma e_1}(t)V_{\gamma e_i}(y)\V}\right].
		\end{align*}
		By definition of $\Theta_i$ from Equation~\eqref{eq:Theta} we see that this extra term is given by
		\begin{equation}
			s_4=\frac4\gamma c_\gamma^1(\bm\mu)\ps{\Db_{-1,\gamma}V_{\beta+2\gamma e_1}(t)\V}.
		\end{equation}
		By combining with the expression obtained above for the case with $\mu_{B}=0$ we arrive to the desired conclusion when this assumption is no longer made:
		\begin{eqs}
			\psi_{-3,\gamma}=-\frac4\gamma c_\gamma(\bm\mu)\Db_{-1,\gamma}V_{\beta+2\gamma e_1}.
		\end{eqs}
		
		All in all, this allows to conclude for the proof of Theorem~\ref{thm:deg3}.
	\end{proof}

    \appendix

\section{Technical estimates and auxiliary computations}

We gather here some technical estimates and auxiliary computations that are used throughout the paper. For the sake of completeness we reproduce the results~\cite[Lemmas 2.4 to 2.8]{CH_sym1}, and also prove several convergence results that are used for the computations of the singular vectors at levels two and three.

\subsection{Technical estimates}
The statements appearing in this Subsection can be found in~\cite{CH_sym1} to which we refer for more details.
\begin{lemma}\label{lemma:inf_integrability_toda}
    Assume that $\bm{\alpha}\in\mc A_{N,M}$ and consider for $i=1,2$ additional insertions $\bm{x}_i \coloneqq\left(x^{(1)}_i,\cdots,x^{(N_i)}_i\right)\in\H^{N_i}$ and $\bm{y}_i \coloneqq\left(y^{(1)}_i,\cdots,y^{(M_i)}_i\right)\in\R^{M_i}$. Then for any  $h>0$, if for $i=1,2$ the $\bm{x}_i$, $\bm{y}_i$ and $\bm{z}$ stay in the domain $U_h\coloneqq \left\lbrace{\bm w, h<\min\limits_{n\neq m}\norm{w^{(n)}-w^{(m)}}}\right\rbrace$, then there exists $C=C_h$ such that, uniformly in $\delta,\eps,\rho$,
    \begin{equation*}
        \ps{\prod_{i=1}^2 \prod_{n=1}^{N_i}V_{\gamma e_i}\left(x^{(n)}_i\right) \prod_{m=1}^{M_i}V_{\gamma e_i}\left(y^{(m)}_i\right)\V}_{\delta,\eps,\rho}\leq C_h \prod_{i=1}^2 \prod\limits_{n=1}^{N_i}\left(1+\norm{x^{(n)}_i}\right)^{-4}\prod\limits_{m=1}^{M_i}\left(1+\norm{y^{(m)}_i}\right)^{-2}.
    \end{equation*}
\end{lemma}

\begin{lemma}\label{lemma:fusion}
    Assume that $\bm{\alpha}\in\mc A_{N,M}$ and that all pairs of points in $\bm{z}$ are separated by some distance $h>0$ except for one pair $(z_1,z_2)$. Further assume that $\ps{\alpha_1+\alpha_2-Q,e_1}<0$. 
    Then as $z_1\to z_2$, for any positive $\eta$ there exists a positive constant $K$ such that, uniformly on $\delta,\eps,\rho$:
    \begin{enumerate}
        \item if $z_1,z_2\in\H$ then
        \begin{equation}\label{eq:fusion_hh}
            \ps{\V}_{\delta,\eps,\rho}\leq K \norm{z_1-z_2}^{-\ps{\alpha_1,\alpha_2}+\left(\frac{1}{2}\ps{\alpha_1+\alpha_2-Q,e_2}^2-\eta\right)\mathds{1}_{\ps{\alpha_1+\alpha_2-Q,e_2}>0}};
        \end{equation}
        \item if $z_1,z_2\in\R$ then
        \begin{equation}\label{eq:fusion_rr}
            \ps{\V}_{\delta,\eps,\rho}\leq K \norm{z_1-z_2}^{-\frac{\ps{\alpha_1,\alpha_2}}2+\left(\frac{1}{4}\left(\ps{\alpha_1+\alpha_2-Q,e_2}\right)^2-\eta\right)\mathds{1}_{\ps{\alpha_1+\alpha_2-Q,e_2}>0}};
        \end{equation}
        \item if $z_1\in\H$ while $z_2\in\R$ with in addition $\ps{\alpha_1-\frac Q2,e_1}<0$ then
        \begin{equation}\label{eq:fusion_hr}
        \begin{split}
            &\ps{\V}_{\delta,\eps,\rho}\leq K\times \norm{z_1-\bar z_1}^{-\frac{\norm{\alpha_1}^2}2+\left(\left(\ps{\alpha_1-\frac Q2,e_2}\right)^2-\eta\right)\mathds{1}_{\ps{\alpha_1-\frac Q2,e_2}>0}} \\
            & \norm{z_1-z_2}^{-\ps{\alpha_1,\alpha_2}+\left(\left(\ps{\alpha_1+\frac{\alpha_2}2-\frac Q2,e_2}\right)^2-\eta\right)\mathds{1}_{\ps{\alpha_1+\frac{\alpha_2}2-\frac Q2,e_2}>0}}.
        \end{split}
        \end{equation}
    \end{enumerate}
    As a consequence for $\beta$ such that $\ps{\beta,e_i}<0$ we have that $\norm{x-t}^{\frac{\ps{\beta,\gamma e_i}}2}F_i(x)$ is continuous at $x=t$, where $F_i$ is of the form
    $$
        F_i(x) = \Is^p\times\Ir^q \left[r_{i,\bm{i}}(x,\bm{x})\Psi_{i,\bm{i}}(x,\bm{x})\right]
    $$
    with $r_{i,\bm{i}}=r_{i,i_1,...,i_{(p+q)}}$ regular at $x=t$ and such that $r_{i,\bm{i}}(x,\bm{x})\Psi_{i,\bm{i}}(\bm{x})$ is integrable on $\Is^p\times\Ir^q$, while $\Psi_{i,\bm{i}}(x,\bm{x}) = \Psi_{i,i_1,...,i_{p+q}}(x,x_1,...,x_{p+q})$.
\end{lemma}

\begin{lemma}~\label{lemma:fusion_integrability}
    The regularized correlation functions $\ps{\prod_{k=1}^NV_{\alpha_k}(z_k)\prod_{l=1}^MV_{\beta_l}(s_l)}_{\delta,\eps,\rho}$ are $(P)$-class. Moreover for any $i$ and $j$ in $\{1,2\}$ the following integrals are $(P)$-class:
    \begin{equation}\label{eq:fusion_int}
        \begin{split}
            &\int_{\frac12\D\times(\D\setminus\frac12\D)}\frac1{y-x}\ps{V_{\gamma e_i}(x+i)V_{\gamma e_j}(y+i)\V}_{\delta,\eps,\rho}d^2xd^2y,\\
            &\int_{-1}^0\int_0^1\frac1{y-x}\ps{V_{\gamma e_i}(x)V_{\gamma e_j}(y)\V}_{\delta,\eps,\rho}dxdy\quad\text{and}\\
            &\int_{\D\cap\H}\int_1^2\frac1{y-x}\ps{V_{\gamma e_i}(x)V_{\gamma e_j}(y)\V}_{\delta,\eps,\rho}d^2xdy.
        \end{split}	
    \end{equation}
In particular the following integral is $(P)$-class:  
\begin{equation*}
    \Is\times\Ir\left[\frac{1}{x-y}\Psi_{i,j}(x,y)\right]
\end{equation*}
\end{lemma}

\subsection{Auxiliary computations}\label{subsec:appendix_auxiliary}
    
    \begin{lemma}\label{lemma:limtx} Assume that $\beta=-\frac2\gamma\omega_1$ and that $F$ is of the form
        $$
        F_i(x) = \Is^p\times\Ir^q \left[r_{i,\bm{i}}(x,\bm{x})\Psi_{i,\bm{i}}(x,\bm{x})\right]
        $$
        where $r_{i,\bm{i}}$ and $\Psi_{i,\bm{i}}(x,\bm{x})$ are as in Lemma~\ref{lemma:fusion}. Then we have
		\begin{eqs} \label{eq:lemma_limtx}
			\Is\left[\partial_x\left(\frac{\ps{\beta,\gamma e_i}}{2(t-x)}F_i(x)\right)\right]=\text{$(P)$-class terms }-(\mu_{L,1}+\mu_{R,1})\tilde F_1(t)
		\end{eqs}
        with $\tilde F$ given by $\tilde F_i(x)\coloneqq \norm{t-x}^{\frac{\ps{\beta,\gamma e_i}}2}F_{i}(x)$, continuous at $x=t$ (via Lemma~\ref{lemma:fusion}).
	\end{lemma}

    \begin{proof}
        Using Stokes' formula, the only terms that may not be $(P)$-class in the left-hand side of \eqref{eq:lemma_limtx} are given by
        $$
        -\frac{1}{\eps}(\mu_{L,1}F_1(t-\eps)+\mu_{R,1}F_1(t+\eps)) + \mu_{B,1}\int_{(t-r,t+r)} \Im \left(\frac{1}{t-x-i\delta}\right) F_1(x+i\delta)dx.
        $$
        Since $\ps{\gamma e_1-\frac{2}{\gamma}\omega_1-Q,e_1}<0$ we know thanks to Lemma~\ref{lemma:fusion} that $\Psi_{1,\bm{i}}(t\pm\eps,\bm{x})$ scales like $\eps$ when $\eps\to 0$, thus we readily see that the first term of the above converges to
        $$
        -(\mu_{L,1}+\mu_{R,1})\tilde F_1(t).
        $$
        The second term can be rewritten as
        $$
        \mu_{B,1}\delta^{1-\gamma^2}2^{-\gamma^2}\int_{(t-r,t+r)} |t-x-i\delta|^{-2} (2\delta)^{\gamma^2} F_1(x+i\delta)dx
        $$
        so we have to separate the cases depending on the sign of $\gamma^2-\frac23$. If $\gamma^2<\frac23$ then from the fusion asymptotics (Lemma~\ref{lemma:fusion}) we have that 
        $$
        |t-x-i\delta|^{-2} (2\delta)^{\gamma^2} \Psi_{1,\bm{i}}(x+i\delta,\bm{x}) \to\ps{V_{\gamma e_{i_1}}(x_1)...V_{
        \gamma e_{i_(p+q)}}(x_{p+q})V_{2\gamma e_1-\frac{2}{\gamma}\omega_1}(t)\V}
        $$
        
        uniformly in $\bm{x}$ so that we can write the above as
        $$
        \delta^{1-\gamma^2}\left(2^{-\gamma^2}\mu_{B,1}\int_{(t-r,t+r)} (t-x) F_1(t)dx+o(1)\right)
        $$
        which tends to $0$ as $\delta\to0$. If $\gamma^2 \ge \frac23$ then again thanks to the fusion asymptotics there exists a positive constant $C$ such that for $\eta>0$ small enough,
        $$
        \delta\int_{(t-r,t+r)} \frac{F_1(x+i\delta)}{(t-x)^2+\delta^2} dx \le C\delta^{1-\gamma^2+\left(\frac32 \gamma - \frac1\gamma \right)^2-\eta}
        $$
        with the exponent being positive, thus finishing the proof.
    \end{proof}
    
	\begin{lemma}\label{lemma:limty} 
		Assume that $G$ is of the form 
        $$
        G_{i,j}(x,y) =  \Is^p\times\Ir^q \left[r_{i,j,\bm{i}}(x,y,\bm{x})\Psi_{i,j,\bm{i}}(x,y,\bm{x})\right]
        $$
        where $r_{i,j,\bm{i}}=r_{i,j,i_1,...,i_{(p+q)}}$ is regular at $x=y=t$, and such that $r_{i,j,\bm{i}}(x,y,\bm{x})\Psi_{i,\bm{i}}(\bm{x})$ is integrable on $\Is^p\times\Ir^q$, with $\Psi_{i,j,\bm{i}}(x,y,\bm{x}) = \Psi_{i,j,i_1,...,i_{p+q}}(x,y,x_1,...,x_{p+q})$.
        and that $F$ is as in Lemma~\ref{lemma:limtx}. Then for $\beta=-\gamma\omega_1$:
		\begin{eqs}
			\Is^2&\left[\partial_x\left(\frac{\ps{\beta,\gamma e_j}\delta_{i=j}}{2(t-y)}G_{i,j}(x,y)\right)\right]=\text{ $(P)$-class terms }-\frac{\gamma^2}{2}c_\gamma^2(\bm\mu)\tilde G_{1,1}(t,t)\qt{and}\\
			\Is&\left[\partial_x\left(\frac{\ps{\beta,\gamma e_i}}{2(t-\bar x)}F_{i}(x)\right)\right]=\text{ $(P)$-class terms }+\frac{\gamma^2}{2}c_\gamma^1(\bm\mu)\tilde F_{1}(t)
		\end{eqs}
        where $\tilde G$ is defined by
        \[
            \tilde G_{i,j}(x,y)=\norm{t-x}^{\frac{\ps{\beta,\gamma e_i}}2}\norm{t-y}^{\frac{\ps{\beta,\gamma e_i}}2}\norm{y-x}^{\frac{\ps{\gamma e_i,\gamma e_j}}2}G_{i,j}(x,y).
        \]
		For $\beta=-\frac2\gamma\omega_1$ these two limits are given by $(P)$-class terms.
	\end{lemma}
	
    \begin{proof}
        We proceed as in Lemma~\ref{lemma:limtx} and apply Stokes' formula to the first integral to obtain
        \begin{eqs}
            &\Is^2\left[\partial_x\left(\frac{1}{t-y}G_{1,1}(x,y)\right)\right] = (P)\text{-class terms} \\&+ \int_{(t-r,t-\eps)\cup(t+\eps,t+r)} \frac{1}{t-y}\left( \mu_{L,1}G_{1,1}(t-\eps,y)-\mu_{R,1}G_{1,1}(t+\eps,y)\right) \mu_{1}(dy).
        \end{eqs}
        Developing the integral yields four terms, proportional to $\mu_{L,1}^2,\mu_{R,1}^2$ and $\mu_{L,1}\mu_{R,1}$. We first focus on the one with coefficient $\mu_{L,1}^2$ which writes
        $$
        \mu_{L,1}^2 \int_\eps^r |y|^{-1} G_{1,1}(t-\eps,t-y)dy.
        $$
        Thanks to Lemma~\ref{lemma:fusion}, in the case where $\beta=-\gamma\omega_1$ and $\gamma<1$ we know that $\tilde G_{1,1}(x,y)$ is continuous at $x=y=t$ since $\ps{2\gamma e_1-\gamma\omega_1-Q,e_1}<0$. As a consequence we can write the latter as
        $$
        \mu_{L,1}^2 \left( \tilde{G}_{1,1}(t,t)\int_\eps^r |y|^{-1+\frac{\gamma^2}{2}} \eps^{\frac{\gamma^2}{2}}|y-\eps|^{-\gamma^2}dy + o(1)\right)
        $$
        and make the change of variable $y\leftrightarrow \eps y$ to obtain
        $$
        \mu_{L,1}^2 \left( \tilde{G}_{1,1}(t,t)\int_1^{r/\eps} |y|^{-1+\frac{\gamma^2}{2}} |y-1|^{-\gamma^2}dy + o(1)\right).
        $$
        The integral is evaluated in \cite[Lemma A.2]{Cer_HEM} and is equal to $\frac{\Gamma\left(\frac{\gamma^2}{2}\right)\Gamma\left(1-\gamma^2\right)}{\Gamma\left(1-\frac{\gamma^2}{2}\right)}$. For the terms proportional to $\mu_{L,1}\mu_{R,1}$ we proceed exactly the same way to obtain
        $$
        \mu_{L,1}\mu_{R,1}\left( \tilde{G}_{1,1}(t,t)\int_1^{r/\eps} |y|^{-1+\frac{\gamma^2}{2}} |y+1|^{-\gamma^2}dy + o(1)\right).
        $$
        We know that this integral is equal to $\cos\left(\pi\frac{\gamma^2}{2}\right)\frac{\Gamma\left(\frac{\gamma^2}{2}\right)\Gamma\left(1-\gamma^2\right)}{\Gamma\left(1-\frac{\gamma^2}{2}\right)}$, thus recollecting the terms we conclude for the proof of the first point when $\beta=-\gamma\omega_1$ and $\gamma<1$. When $\gamma>1$ we now have $\ps{2\gamma e_1-\gamma\omega_1-Q,e_1}>0$, and therefore there is an extra term in the fusion estimates. This yields an extra scaling factor of $\ps{\gamma e_1-\frac\gamma2\omega_1-\frac Q2,e_1}^2=(\gamma-\frac1\gamma)^2$, and thus an overall scaling factor of $\eps^{-2+\gamma^2+\frac{1}{\gamma^2}}$. This exponent being positive we can conclude for the case $\gamma>1$. For $\beta=-\frac2\gamma\omega_1$ the very same reasoning shows that these integrals scale like $\eps^{2-\gamma^2}$ and therefore vanish in the limit, concluding the proof.

        Let us now turn to the integral with $\bar x$ instead of $y$. In that case we can proceed as in Lemma~\ref{lemma:limtx} to see that the corresponding integral, after application of Stokes' formula, is given by
        $$
            -\mu_{B,1}\delta^{1-\gamma^2}2^{-\gamma^2}\int_{(t-r,t+r)} |t-x-i\delta|^{-2+\gamma^2} \tilde F_1(x+i\delta)dx.
        $$
        We make the change of variables $x\leftrightarrow t+\delta x$ to see that the latter is given by
        \[
            -\mu_{B,1}2^{-\gamma^2}\int_{\R} (1+x^2)^{-1+\frac{\gamma^2}2}dx\tilde F_1(t)+o(1).
        \]
        The integral is evaluated in \cite[Lemma A.1]{Cer_HEM} and is given by $2^{\gamma^2}\sin\left(\pi\frac{\gamma^2}{2}\right)\frac{\Gamma\left(\frac{\gamma^2}{2}\right)\Gamma\left(1-\gamma^2\right)}{\Gamma\left(1-\frac{\gamma^2}{2}\right)}$, thus concluding for the proof of the lemma.
    \end{proof}

    We eventually gather here some statements we used along the proof of Theorem~\ref{thm:deg3}.
	\begin{lemma}\label{lemma:sing3}
		Recall the definition of $A$ from Equation~\eqref{eq:defA}. In the $\delta,\eps,\rho\to 0$ limit:
		\begin{equation}
			\begin{split}
				A&-\frac{8}{\chi^3}\Is^3\left[\partial_x\left(\frac{\ps{\beta,\gamma e_j}\ps{\beta,\gamma e_f}}{4(t-y)(t-z)}\Psi_{i,j,f}(x,y,z)\right)\right]=\text{ $(P)$-class terms}+r_2\mathds 1_{\chi=\frac2\gamma}\\
				&-\frac4\chi\Is\times\Is\Big[\partial_x\left(\sum_{k=1}^{2N+M}\left(\frac{\ps{\gamma e_i,\alpha_k}}{2(t-x)(z_k-y)}\Psi_{1,i}(x,y)-\frac{\ps{\gamma e_1,\alpha_k}}{2(t-y)(z_k-y)}\Psi_{i,1}(x,y)\right)\right)\Big]\\
				&+\frac4\chi\Is^2\times\Ir\left[\partial_x\left(\frac{\ps{\gamma e_i,\gamma e_j}}{2(t-x)(z-y)}\Psi_{1,i,j}(x,y,z)-\frac{\ps{\gamma e_1,\gamma e_j}}{2(t-y)(z-y)}\Psi_{i,1,j}(x,y,z)\right)\right]
			\end{split}
		\end{equation}
		where we have set
		\begin{equation}
			r_2\coloneqq \gamma^2(\frac2\gamma-\gamma)(\mu_{L,1}+\mu_{R,1})\Is\Big[\frac{\Psi_{1,2}(t,y)}{t-y}\Big].
		\end{equation}
	\end{lemma}
	\begin{proof}
		Let us write $\mathfrak J\coloneqq A-\frac{8}{\chi^3}\Is^3\left[\partial_x\left(\frac{\ps{\beta,\gamma e_j}\ps{\beta,\gamma e_f}}{4(t-y)(t-z)}\Psi_{i,j,f}(x,y,z)\right)\right]$. To start with we use the specific value of $\beta=-\chi\omega_1$ with $\chi\in\{\gamma,\frac2\gamma\}$ to simplify the expression of $\mathfrak J$:
		\begin{equation}\label{eq:mathfrakI}
			\begin{split}
				\mathfrak J&=\Is^2\Big[\partial_x\Big(\frac{2\gamma}{(t-y)(x-y)}\Psi_{2,1}(x,y)-\frac{2\gamma}{(t-x)(x-y)}\Psi_{1,2}(x,y)\Big)\Big]\\
				&-\frac4\chi\Is^2\Big[\partial_x\left(\frac{1+\frac{\ps{\beta,\gamma e_1}}{2}}{(t-y)^2}\Psi_{i,1}(x,y)+\frac{\frac{\gamma\chi}2-\gamma^2}{(t-x)(t-y)}\Psi_{1,1}(x,y)\right)\Big]\\
				&-\frac{4}{\chi}\Is^3\left[\partial_x\left(\frac{\gamma^2}{2(t-y)(t-z)}\Psi_{i,1,1}(x,y,z)\right)\right]. 
			\end{split}
		\end{equation}
		Now we claim that the following identity holds true:
		\begin{equation}\label{eq:derpsi}
			\begin{split}
				&\Is^2\Big[\partial_x\left(\frac{1+\frac{\ps{\beta,\gamma e_1}}{2}}{(t-y)^2}\Psi_{i,1}(x,y)+\frac{\frac{\gamma\chi}2-\gamma^2}{(t-x)(t-y)}\Psi_{1,1}(x,y)\right)\Big]\\
				&-\Is^2\Big[\partial_x\Big(\frac{\gamma^2}{2(t-y)(x-y)}\Psi_{2,1}(x,y)-\frac{\gamma^2}{2(t-x)(x-y)}\Psi_{1,2}(x,y)\Big)\Big]\\
				&-\Is^2\times\It\left[\partial_x\left(\frac{\ps{\gamma e_1,\gamma e_j}}{2(t-y)(z-y)}\Psi_{i,1,j}(x,y,z)-\frac{\ps{\gamma e_i,\gamma e_j}}{2(t-x)(z-y)}\Psi_{1,i,j}(x,y,z)\right)\right]\\
				&=\Is^2\Big[\partial_x\left(\sum_{k=1}^{2N+M}\left(\frac{\ps{\gamma e_1,\alpha_k}}{2(t-x)(z_k-y)}-\frac{\ps{\gamma e_1,\alpha_k}}{2(t-y)(z_k-y)}\right)\Psi_{1,1}(x,y)\right)\Big]\\
				&+\Is^2\Big[\partial_x\left(\sum_{k=1}^{2N+M}\left(\frac{\ps{\gamma e_2,\alpha_k}}{2(t-x)(z_k-y)}\Psi_{1,2}(x,y)-\frac{\ps{\gamma e_1,\alpha_k}}{2(t-y)(z_k-y)}\Psi_{2,1}(x,y)\right)\right)\Big].
			\end{split}
		\end{equation}
		To see why this is the case we note that by symmetry in $x,y$ we have $$0=\Is\times\Is\Big[\partial_x\partial_y\left(\frac{1}{(t-y)}\Psi_{i,1}(x,y)\right)\Big]-\Is\times\Is\Big[\partial_x\partial_y\left(\frac{1}{(t-x)}\Psi_{1,i}(x,y)\right)\Big].$$
		We then develop the derivative in $y$: this is done by means of the equality
		\begin{align*}
			&\partial_y\left(\frac{1}{t-y}\Psi_{i,1}(x,y)\right)=\left(\frac{1+\frac{\ps{\beta,\gamma e_1}}{2}}{(t-y)^2}+\frac{\ps{\gamma e_i,\gamma e_1}}{2(t-y)(x-y)}+\sum_{k=1}^{2N+M}\frac{\ps{\alpha_k,\gamma e_1}}{2(t-y)(z_k-y)}\right)\Psi_{i,1}(x,y)\\
			&-\It\left[\frac{\ps{\gamma e_1,\gamma e_j}}{2(t-y)(z-y)}\Psi_{i,1,j}(x,y,z)\right]
		\end{align*}
		and likewise for the other term that appears there. Recollecting terms and simplifying the expression obtained using that $\beta=-\chi\omega_1$ we then arrive to Equation~\eqref{eq:derpsi}.
		Thanks to this equality we can rewrite Equation~\eqref{eq:mathfrakI} under the form
		\begin{align*}
			&\mathfrak I=2\gamma\left(1-\frac\gamma\chi\right)\Is\times\Is\Big[\partial_x\Big(\frac{\Psi_{2,1}(x,y)}{(t-y)(x-y)}-\frac{\Psi_{1,2}(x,y)}{(t-x)(x-y)}\Big)\Big]\\
			&-\frac{4}{\chi}\Is^3\left[\partial_x\left(\frac{\gamma^2}{2(t-y)(t-z)}\Psi_{i,1,1}(x,y,z)\right)\right]\\
			&-\frac4\chi\Is^2\times\It\left[\partial_x\left(\frac{\ps{\gamma e_1,\gamma e_j}}{2(t-y)(z-y)}\Psi_{i,1,j}(x,y,z)-\frac{\ps{\gamma e_i,\gamma e_j}}{2(t-x)(z-y)}\Psi_{1,i,j}(x,y,z)\right)\right]\\
			&-\frac4\chi\Is\times\Is\Big[\partial_x\left(\sum_{k=1}^{2N+M}\left(\frac{\ps{\gamma e_1,\alpha_k}}{2(t-x)(z_k-y)}-\frac{\ps{\gamma e_1,\alpha_k}}{2(t-y)(z_k-y)}\right)\Psi_{1,1}(x,y)\right)\Big]\\
			&-\frac4\chi\Is\times\Is\Big[\partial_x\left(\sum_{k=1}^{2N+M}\left(\frac{\ps{\gamma e_2,\alpha_k}}{2(t-x)(z_k-y)}\Psi_{1,2}(x,y)-\frac{\ps{\gamma e_1,\alpha_k}}{2(t-y)(z_k-y)}\Psi_{2,1}(x,y)\right)\right)\Big]. 
		\end{align*}
		
		The three-fold integrals can be dealt with by using the symmetries in the different variables involved. We see that by doing so the integrals simplify to
		\begin{align*}
			&\frac4\chi\Is^3\left[\partial_x\left(\frac{\gamma^2}{2(t-y)(z-y)}\Psi_{i,1,2}(x,y,z)\right)\right]\\
			&+\frac4\chi\Is^2\times\Ir\left[\partial_x\left(\frac{\ps{\gamma e_i,\gamma e_j}}{2(t-x)(z-y)}\Psi_{1,i,j}(x,y,z)-\frac{\ps{\gamma e_1,\gamma e_j}}{2(t-y)(z-y)}\Psi_{i,1,j}(x,y,z)\right)\right].
		\end{align*}
		Now thanks to our assumption that $\mu_{L,2}-\mu_{R,2}=0$ together with the fact that, in virtue of Lemma~\ref{lemma:fusion}, the integral 
		\begin{align*}
			\Is^2\left[\frac{1}{2(t-y)(z-y)}\ps{V_{\beta+\gamma e_2}(t)V_{\gamma e_1}(y)V_{\gamma e_2}(z)\V}\right]
		\end{align*}
		is absolutely convergent, we see that the term corresponding to $i=2$ in the integral over $\Is^3$ vanishes in the limit. As for the case where $i=1$, along the same lines as in the proof of Theorem~\ref{thm:deg2} we see that for $\chi=\frac2\gamma$ this term goes to $0$ as well. The case where $i=1$ and $\chi=\gamma$ requires more care: we develop the derivative in $x$ to rewrite the integral over $\Is^3$ as
		\begin{align*}
			&2\gamma\Is^3\left[\left(\frac{-\gamma^2}{2(t-y)(z-y)(t-x)}+\frac{\gamma^2}{(t-y)(z-y)(y-x)}+\frac{-\gamma^2}{2(t-y)(z-y)(z-x)}\right)\Psi_{1,1,2}(x,y,z)\right]\\
			&+2\gamma\Is^3\left[\sum_{k=1}^{2N+M}\frac{\ps{\gamma e_1,\alpha_k}}{2(t-y)(z-y)(z_k-x)}\Psi_{1,1,2}(x,y,z)\right]\\
			&-2\gamma\Is^3\times\It\left[\frac{\ps{\gamma e_1,\gamma e_i}}{2(t-y)(z-y)(w-x)}\Psi_{1,1,2,i}(x,y,z,w)\right].
		\end{align*}
		Then the miracle happens again: by symmetry in $x,y$ cancellations occur and the first line is seen to vanish. As for the other terms we can write them as derivatives in $y$ of a function that is well-defined at $y=t$ and as such yield a term that vanishes in the limit.  
		This allows to conclude that the three-fold integrals contribute to the limit only by
		\[
		\frac4\chi\Is^2\times\Ir\left[\partial_x\left(\frac{\ps{\gamma e_i,\gamma e_j}}{2(t-x)(z-y)}\Psi_{1,i,j}(x,y,z)-\frac{\ps{\gamma e_1,\gamma e_j}}{2(t-y)(z-y)}\Psi_{i,1,j}(x,y,z)\right)\right].
		\]
		And as a consequence we can write that
		\begin{align*}
			&\mathfrak I=\text{ $(P)$-class terms }+2\gamma\left(1-\frac\gamma\chi\right)\Is\times\Is\Big[\partial_x\Big(\frac{\Psi_{2,1}(x,y)}{(t-y)(x-y)}-\frac{\Psi_{1,2}(x,y)}{(t-x)(x-y)}\Big)\Big]\\
			&-\frac4\chi\Is\times\Is\Big[\partial_x\left(\sum_{k=1}^{2N+M}\left(\frac{\ps{\gamma e_1,\alpha_k}}{2(t-x)(z_k-y)}-\frac{\ps{\gamma e_1,\alpha_k}}{2(t-y)(z_k-y)}\right)\Psi_{1,1}(x,y)\right)\Big]\\
			&-\frac4\chi\Is\times\Is\Big[\partial_x\left(\sum_{k=1}^{2N+M}\left(\frac{\ps{\gamma e_2,\alpha_k}}{2(t-x)(z_k-y)}\Psi_{1,2}(x,y)-\frac{\ps{\gamma e_1,\alpha_k}}{2(t-y)(z_k-y)}\Psi_{2,1}(x,y)\right)\right)\Big]\\
			&+\frac4\chi\Is^2\times\Ir\left[\partial_x\left(\frac{\ps{\gamma e_i,\gamma e_j}}{2(t-x)(z-y)}\Psi_{1,i,j}(x,y,z)-\frac{\ps{\gamma e_1,\gamma e_j}}{2(t-y)(z-y)}\Psi_{i,1,j}(x,y,z)\right)\right].
		\end{align*}
		Hence it remains to treat the first line, for which it is readily seen that it vanishes if $\chi=\gamma$ so that we may assume that $\chi=\frac2\gamma$.

		Now thanks to Lemma~\ref{lemma:fusion} we know that for $\beta=-\frac2\gamma\omega_1$ the integral
		\[
		\Is\left[\frac{1}{(t-y)^2}\ps{V_{\beta+\gamma e_2}(t)V_{\gamma e_1}(y)\V}\right]
		\]
		is absolutely convergent. As a consequence and since we have assumed $\mu_{L,2}-\mu_{R,2}=0$ the term in the first line corresponding to $\Psi_{2,1}$ is a $o(1)$. As for the second one it is seen to converge to
		\begin{equation*}
			r_2\coloneqq \gamma^2(\frac2\gamma-\gamma)(\mu_{L,1}+\mu_{R,1})\Is\Big[\frac{\Psi_{1,2}(t,y)}{t-y}\Big].
		\end{equation*}
		
		To summarize and after all these remarkable simplifications we see that all that is left from Equation~\eqref{eq:mathfrakI} is, as expected,
		\begin{align*}
			&\mathfrak I=\text{ $(P)$-class terms }+r_2\mathds 1_{\chi=\frac2\gamma}\\
			&-\frac4\chi\Is\times\Is\Big[\partial_x\left(\sum_{k=1}^{2N+M}\left(\frac{\ps{\gamma e_1,\alpha_k}}{2(t-x)(z_k-y)}-\frac{\ps{\gamma e_1,\alpha_k}}{2(t-y)(z_k-y)}\right)\Psi_{1,1}(x,y)\right)\Big]\\
			&-\frac4\chi\Is\times\Is\Big[\partial_x\left(\sum_{k=1}^{2N+M}\left(\frac{\ps{\gamma e_2,\alpha_k}}{2(t-x)(z_k-y)}\Psi_{1,2}(x,y)-\frac{\ps{\gamma e_1,\alpha_k}}{2(t-y)(z_k-y)}\Psi_{2,1}(x,y)\right)\right)\Big]\\
			&+\frac4\chi\Is^2\times\Ir\left[\partial_x\left(\frac{\ps{\gamma e_i,\gamma e_j}}{2(t-x)(z-y)}\Psi_{1,i,j}(x,y,z)-\frac{\ps{\gamma e_1,\gamma e_j}}{2(t-y)(z-y)}\Psi_{i,1,j}(x,y,z)\right)\right]. 
		\end{align*}
	\end{proof}

	\bibliography{main}
	\bibliographystyle{plain}
	
\end{document}